\documentclass[a4paper,10pt,reqno]{amsart}

\usepackage{amsmath,amssymb,mathtools}
\usepackage{booktabs,array,longtable}
\usepackage[numbers,sort&compress]{natbib}
\usepackage{petkov-preprint}

\theoremstyle{plain}
\newtheorem{theorem}{Theorem}[section]
\newtheorem{proposition}[theorem]{Proposition}
\newtheorem{lemma}[theorem]{Lemma}
\newtheorem{corollary}[theorem]{Corollary}
\newtheorem{conjecture}[theorem]{Conjecture}
\newtheorem*{wowconjecture}{Conjecture (WOW-284)}
\theoremstyle{definition}
\newtheorem{definition}[theorem]{Definition}
\theoremstyle{remark}
\newtheorem{remark}[theorem]{Remark}

\BeforeBeginEnvironment{theorem}{\Needspace{4\baselineskip}}
\BeforeBeginEnvironment{proposition}{\Needspace{4\baselineskip}}
\BeforeBeginEnvironment{lemma}{\Needspace{4\baselineskip}}
\BeforeBeginEnvironment{corollary}{\Needspace{4\baselineskip}}
\BeforeBeginEnvironment{definition}{\Needspace{4\baselineskip}}
\BeforeBeginEnvironment{remark}{\Needspace{3\baselineskip}}
\BeforeBeginEnvironment{proof}{\Needspace{3\baselineskip}}

\newcommand{\F}{\mathbb F_5}
\newcommand{\one}{\mathbf 1}
\newcommand{\Spec}{\operatorname{Spec}}
\newcommand{\tr}{\operatorname{tr}}
\newcommand{\diam}{\operatorname{diam}}
\newcommand{\resultbox}[1]{\boxed{#1}}
\newcommand{\RepoTag}{v2.2.8}
\newcommand{\codefile}[1]{%
  \href{https://github.com/SamPetkov/wow284/blob/\RepoTag/#1}{\path{#1}}}
\newcommand{\datafile}[1]{%
  \href{https://github.com/SamPetkov/wow284/blob/\RepoTag/#1}{\path{#1}}}

\hypersetup{
  colorlinks=true,
  linkcolor=InternalLink,
  citecolor=CitationLink,
  urlcolor=ExternalLink,
  pdftitle={Counterexamples, Spectral Obstructions, and Deletion Stability for WOW-284},
  pdfauthor={Samuil Petkov},
  pdfsubject={Counterexamples and structural spectral theory for the minimum-dual-degree conjecture WOW-284},
  pdfkeywords={distance spectrum, dual degree, Moore graph},
  pdfcreator={LaTeX with amsart},
  pdfproducer={pdfTeX},
  pdfdisplaydoctitle=true
}

\title[Counterexamples and obstructions for WOW-284]
{Counterexamples, Spectral Obstructions, and Deletion Stability for WOW-284}
\author{Samuil Petkov}
\address{Department of Physics, \'Ecole normale sup\'erieure,
Universit\'e PSL, Paris, France}
\email{samuil.petkov@phys.ens.psl.eu}
\date{29 July 2026}

\begin{document}

\begin{abstract}
WOW-284 is a Graffiti conjecture asserting that every connected graph of order at least three and girth at least five has minimum dual degree at most \(-\lambda_{\min}(D)\), where \(D\) is the distance matrix and the dual degree of a vertex is the average degree of its neighbors. We refute it with exact counterexamples on \(38\), \(39\), \(40\), \(42\) and \(50\) vertices and study the structure of regular counterexamples. For a connected \(k\)-regular graph \(G\) of girth at least five and diameter three, we prove that \(\delta^*(G)+\lambda_{\min}(D)=2k-2-\max_\theta(\theta+1)^2\), the maximum being over the nonprincipal adjacency eigenvalues, so strict failure is a spectral window condition. The closed window \(|\theta+1|\le\sqrt{2k-2}\) alone, without any diameter assumption, forces \(|V(G)|<(k+2)(k^2+3)/6\) when \(k\ge4\), and an integral excess matrix derived from the optimal polynomial sharpens this bound for diameter-three counterexamples. Every regular strict counterexample has degree at least six and diameter at most four, and in degree six it has at most \(50\) vertices.

Under vertex deletion, the Hoffman--Singleton graph remains a strict counterexample after the removal of any five vertices, while an explicit six-vertex deletion destroys strictness. A degree-\(57\) Moore graph, if one exists, remains a connected strict counterexample after the removal of any \(56\) vertices. Every strict counterexample has at least \(15\) vertices, and an exhaustive search excludes orders up to \(19\). The strict induced subgraphs of the Hoffman--Singleton graph of orders \(36\) to \(50\) are classified up to automorphism; among them are three further irregular counterexamples of order \(40\).
\end{abstract}

\maketitle

\section{WOW-284 and the main results}\label{sec:introduction}

Let \(G\) be a finite simple connected graph of order
\(n=|V(G)|\ge2\).  Let \(A=A(G)\) and \(D=D(G)\) be its adjacency and
distance matrices, let \(I\) and \(J\) denote the identity and all-ones
matrices of order \(n\), and let \(\one\) be the all-ones vector.  Write
\(d(v)\) for the degree of \(v\), \(N(v)\) for its open neighborhood, and, for
\(S\subseteq V(G)\), let \(G-S\) denote the subgraph induced by
\(V(G)\setminus S\).  For \(U\subseteq V(G)\), \(G[U]\) is the induced
subgraph on \(U\), and for a matrix indexed by \(V(G)\) the suffix \([U]\)
denotes its principal submatrix on \(U\).  Define
\[
 d^*(v)=\frac1{d(v)}\sum_{u\in N(v)}d(u),
 \qquad
 \delta^*(G)=\min_{v\in V(G)}d^*(v).
\]
Here \(d_G(u,v)\) is the length of a shortest path from \(u\) to \(v\),
so \(D=(d_G(u,v))_{u,v\in V(G)}\), and \(\lambda_{\min}(D)\) is its least
eigenvalue.
Let \(e_v\) denote the standard coordinate vector at \(v\). For a square
matrix \(M\), write \(\chi_M(t)=\det(tI-M)\) for its characteristic
polynomial and \(E_\lambda(M)=\ker(M-\lambda I)\) for its
\(\lambda\)-eigenspace, with \(I\) of the matching order. We list spectra with
eigenvalue multiplicities as superscripts in parentheses, and order adjacency
eigenvalues decreasingly. Orthogonality refers to the Euclidean inner product;
the orthogonal complement of the constant vector is the zero-sum space.
The girth \(g(G)\) is the length of a
shortest cycle, with infinite girth when there is no cycle. The diameter is
the largest distance between two vertices. For a connected regular graph, we
call the adjacency eigenvalues other than its degree eigenvalue
nonprincipal.

Aouchiche and Hansen \cite[Conjecture~7.16]{AouchicheHansen2014} record the
following Graffiti conjecture and attribute it to Fajtlowicz's 1998
\emph{Written on the Wall} report \cite{Fajtlowicz1998}.

\begin{wowconjecture}
If \(G\) has order at least three and girth at least five, then
\[
 \delta^*(G)\le-\lambda_{\min}(D(G)).
\]
\end{wowconjecture}

We measure the failure of the conjectured inequality by a scalar score:
for every finite simple connected graph of order at least two, put
\[
 \Phi(G)=\delta^*(G)+\lambda_{\min}(D(G)).
\]
Among graphs satisfying the conjecture's hypotheses, \(G\) is a strict
counterexample precisely when \(\Phi(G)>0\).

\subsection{Main results}

The first counterexample comes from a degree-\(k\) Moore graph of diameter
two \cite{HoffmanSingleton1960}: it has \(\delta^*=k\) and
\(\lambda_{\min}(D)=-(3+\sqrt{4k-3})/2\)
(Theorem~\ref{thm:moore-threshold}).
The conjecture holds on these graphs exactly for \(k\le3\), with equality at
\(k=3\), and fails for every realizable \(k>3\).  The degree-seven
Hoffman--Singleton graph therefore gives a gap of three.
From the Hoffman--Singleton coordinates, Section~\ref{sec:examples}
extracts further strict counterexamples on \(38\), \(39\), \(40\) and \(42\)
vertices (Theorem~\ref{thm:explicit-examples}); the second subconstituent of
a degree-\(K\) Moore graph is strict exactly when \(K\ge6\)
(Theorem~\ref{thm:second-subconstituent}).

The paper follows two lines of argument. For regular graphs of diameter
three, the distance-polynomial identity turns strict failure of WOW-284 into
an interval condition on the nonprincipal adjacency eigenvalues
(Theorem~\ref{thm:diameter-three-score}). Taking the trace of a polynomial
in the adjacency matrix then gives \(n<B_k=(k+2)(k^2+3)/6\) under the closed
window alone when \(k\ge4\), and no polynomial of higher degree improves this
bound (Proposition~\ref{prop:moment-bound}, Theorem~\ref{thm:lp-ceiling}).
Evaluating the optimal polynomial at the adjacency
matrix, subtracting a constant multiple of the all-ones matrix and adjusting
the diagonal gives the integral excess matrix of
Theorem~\ref{thm:integral-slack}.
Its integer entries, together with the nonnegative principal minors of the
associated slack matrix, give stronger order and cycle constraints. These lead to the
three-to-one bound (Theorem~\ref{thm:three-to-one}) and to \(n\le50\) in
degree six (Theorem~\ref{thm:degree-six-fifty}), where the case \(n=50\)
remains unresolved. The polynomial bound in
Section~\ref{sec:lp} is optimal within the stated class of polynomials;
we do not claim that the resulting graph-order bound is best possible.

The degree and diameter restrictions narrow the search: every regular strict
counterexample has degree at least six and diameter at most four, and a
diameter-four example would have degree at least ten
(Theorems~\ref{thm:regular-degree-six}, \ref{thm:endpoint-diameter},
and~\ref{thm:diameter-four}).

The second argument concerns deletion. Deleting vertices can increase the
distances between those left behind, so taking a principal submatrix of the
old distance matrix is not enough. For small deletions from a Moore graph,
we identify exactly which distances increase. This correction gives general
stability bounds (Section~\ref{sec:punctures});
Appendix~\ref{app:small-deletion-spectra} gives the exact one- and
two-deletion spectra. For Hoffman--Singleton, the bounds prove
strictness through three deletions, and exact orbit certificates handle four
and five. A six-vertex deletion then shows that five is the universal
radius, the largest \(s\) such that deleting any \(s\) or fewer vertices
leaves a strict counterexample (Theorem~\ref{thm:hs-radius}).
For a hypothetical degree-\(57\) Moore graph, the same corrections give
strictness after any \(56\) deletions
(Proposition~\ref{prop:clique-deletion-bound}).
Theorem~\ref{thm:equality-boundary} characterizes equality on the window boundary;
J{\o}rgensen's \(96\)-vertex graph attains it.

We do not claim that our examples exhaust all counterexamples, nor that the
order-$38$ example is minimal.
Section~\ref{sec:small-order-bounds} proves that a strict counterexample has at
least \(15\) vertices, shows by exhaustive computation that none has at most
\(19\), classifies the strict induced subgraphs of the Hoffman--Singleton
graph of orders \(36\) to \(50\), and gives three further, irregular
counterexamples of order \(40\).

\subsection{Related work}

The Hoffman--Singleton counterexample is also recorded in a public Capy
Build transcript~\cite{CapyBuild2026} and in a verification note by Wilbanks
and Annie~\cite{WilbanksAnnie2026}.

Howlader and Panigrahi already give the distance spectra of Moore graphs,
including the Hoffman--Singleton graph
\cite[Theorem~2.5(1)]{HowladerPanigrahi2022}. The distance-polynomial
framework is developed in \cite{HowladerPanigrahi2022,Fiol2016}.
Nonbacktracking walks never immediately reverse an edge. Linear-programming
bounds based on their counts are due to Nozaki
\cite{Nozaki2015}; related spectral-Moore bounds appear in
\cite{CioabaEtAl2016}. We use these methods to derive order restrictions
from the WOW-284 spectral window and to study vertex deletion.

\subsection{Organization}

Section~\ref{sec:moore} computes the score of every Moore graph of
diameter two and gives the Hoffman--Singleton graph in coordinates.
Section~\ref{sec:score-calculus} proves the diameter-three identity and
identifies its equality case, and Section~\ref{sec:examples} uses it to
certify the smaller counterexamples. Section~\ref{sec:obstructions} bounds
the diameter of regular counterexamples, Section~\ref{sec:lp} proves the
polynomial order bounds, and Section~\ref{sec:degree-six} combines them into
the degree and order restrictions. Section~\ref{sec:punctures} treats vertex
deletion, Section~\ref{sec:small-order-bounds} treats small orders, and
Section~\ref{sec:scope} lists open questions.

The appendices contain the longer proofs and the finite certificates, in the
order in which the main text first needs them: the deletion spectra of Moore
graphs and the proof of Theorem~\ref{thm:second-subconstituent}
(Appendix~\ref{app:small-deletion-spectra}), the certificates for the
explicit examples (Appendix~\ref{app:small-example-certificates}), three
technical steps in the polynomial order bounds
(Appendix~\ref{app:lp-slack-positivity}), the proof of the three-to-one
bound (Appendix~\ref{app:three-to-one-proof}), the degree-five case of the
degree bound (Appendix~\ref{app:degree-five}), necessary conditions at
degree six and order fifty (Appendix~\ref{app:order50-details}), further
deletions inside a neighborhood (Appendix~\ref{app:neighborhood-deletions}),
the methods and further results for small orders
(Appendix~\ref{app:small-orders}), and two construction families excluded by
the spectral window (Appendix~\ref{app:excluded-constructions}).
Appendix~\ref{sec:verification} records the exact computations and the
scope of the Lean verification.
Identities and finite case checks inside the analytic proofs are also
reproduced in exact arithmetic (Appendix~\ref{sec:verification}); such checks
are mentioned below only where a proof depends on them.

\section{Dual degree, Moore graphs, and the Hoffman--Singleton graph}
\label{sec:moore}

This section computes the score of every Moore graph of diameter two
(Theorem~\ref{thm:moore-threshold}). WOW-284 holds on these graphs exactly
for degree at most three, so the Hoffman--Singleton graph is a strict
counterexample on \(50\) vertices. We also give this graph in coordinates
(Proposition~\ref{prop:hs-coordinate}); the smaller counterexamples of
Section~\ref{sec:examples} are its induced subgraphs.

\subsection{Dual degree and the size of a radius-two ball}

When triangles and four-cycles are absent, dual degree measures the size
of a radius-two ball relative to the degree of its center. The same count
explains why a regular graph of girth at least five and diameter two is a
Moore graph.

A distance sphere consists of the vertices at a specified distance from
its center. For \(v\in V(G)\), write \(\Gamma_i(v)\) for the distance-\(i\) sphere and
\(B_2(v)=\Gamma_0(v)\cup\Gamma_1(v)\cup\Gamma_2(v)\) for the radius-two ball.

\begin{proposition}[Second-degree identity]\label{prop:radius-two}
If \(G\) contains no triangle and no \(4\)-cycle, then
\[
 |B_2(v)|=1+\sum_{u\in N(v)}d(u),
 \qquad
 d^*(v)=\frac{|B_2(v)|-1}{d(v)}.
\]
\end{proposition}

\begin{proof}
For distinct neighbors \(u,w\) of \(v\), the sets
\(N(u)\setminus\{v\}\) and \(N(w)\setminus\{v\}\) are disjoint; an
intersection would form a \(4\)-cycle, and a member in \(N(v)\) would form a
triangle.  These sets partition \(\Gamma_2(v)\), so
\[
 |\Gamma_2(v)|=\sum_{u\in N(v)}(d(u)-1).
\]
Adding the center and its neighbors proves the first identity, which is
Backelin's Lemma~2.1 \cite{Backelin2015}. Dividing by \(d(v)\) gives the
second identity.
\end{proof}

For a \(k\)-regular graph of girth at least five,
Proposition~\ref{prop:radius-two} gives \(|B_2(v)|=k^2+1\) for every \(v\);
hence \(n\ge k^2+1\), with equality exactly when the diameter is at most two
(the \emph{Moore bound}). A degree-\(k\) \emph{Moore graph of diameter two}
is a \(k\)-regular graph of diameter two on \(k^2+1\) vertices. At most
\(1+k+k(k-1)\) vertices lie within distance two of a vertex, with equality
only when no triangle or \(4\)-cycle is present; so in such a graph adjacent
vertices have no common neighbor and nonadjacent vertices have exactly one.

\subsection{When Moore graphs violate WOW-284}

\begin{theorem}\label{thm:moore-threshold}
Let \(M\) be a degree-\(k\) Moore graph of diameter two, \(k\ge2\).  Then
\[
 |V(M)|=k^2+1,
 \qquad g(M)=5,
 \qquad \delta^*(M)=k,
\]
\[
 \lambda_{\min}(D(M))=-\frac{3+\sqrt{4k-3}}2,
\]
and
\[
 \Phi(M)=k-\frac{3+\sqrt{4k-3}}2.
\]
Thus \(M\) satisfies WOW-284 exactly for \(k\le3\), with equality exactly at
\(k=3\).
\end{theorem}

\begin{proof}
Equality in the Moore bound fixes the common-neighbor counts: adjacent
vertices have no common neighbor, and nonadjacent vertices have exactly
one. These counts also give a five-cycle through every edge. Given an edge
\(uv\), choose \(x\in N(u)\setminus\{v\}\) and
\(y\in N(v)\setminus\{u\}\).  The vertices \(x,y\) cannot be adjacent,
since that would create a four-cycle; their unique common neighbor
therefore completes a five-cycle through \(uv\).  The common-neighbor
counts give the matrix identity
\[
 A^2=(k-1)I-A+J.
\]
On \(\one^\perp\), the nonprincipal adjacency eigenvalues are the roots
\[
 r,s=\frac{-1\pm\sqrt{4k-3}}2.
\]
If \(m_r,m_s\) are their multiplicities, then
\[
 m_r+m_s=k^2,\qquad k+m_rr+m_ss=0.
\]
Consequently
\[
 m_r=\frac{k(k\sqrt{4k-3}+k-2)}{2\sqrt{4k-3}},
 \qquad
 m_s=\frac{k(k\sqrt{4k-3}-k+2)}{2\sqrt{4k-3}},
\]
and both roots occur. To obtain the distance eigenvalues, use the fact
that every nonedge has distance two. Thus \(D=2J-2I-A\), and the
least distance eigenvalue is \(-2-r=-(3+\sqrt{4k-3})/2\).
Since regularity gives \(\delta^*=k\), the sign of \(\Phi(M)\) is that of
\(2k-3-\sqrt{4k-3}\). For \(k\ge2\) both terms are positive, so it is the
sign of \((2k-3)^2-(4k-3)\), and comparison with the dual degree
reduces to
\[
 (2k-3)^2-(4k-3)=4(k-1)(k-3)
\]
This gives the stated threshold.
\end{proof}

\subsection{The Hoffman--Singleton graph in coordinates}

We use two families of vertices, with coordinates taken modulo five.
All subscripts below lie in \(\F=\mathbb Z/5\mathbb Z\). Let
\[
 V(M)=\{P_{i,j}:i,j\in\F\}\mathbin{\dot\cup}
      \{Q_{k,\ell}:k,\ell\in\F\},
\]
with edges
\begin{align*}
 P_{i,j}&\sim P_{i,j\pm1},\\
 Q_{k,\ell}&\sim Q_{k,\ell\pm2},\\
 P_{i,j}&\sim Q_{k,ik+j}.
\end{align*}
This is Hafner's affine-coordinate form of the Hoffman--Singleton graph after a
minor reindexing \cite{Hafner2003}.

\begin{proposition}\label{prop:hs-coordinate}
The coordinate construction is a simple connected \(7\)-regular graph on
\(50\) vertices.  Adjacent pairs have no common neighbor and nonadjacent pairs
have exactly one.  Consequently it has girth five, diameter two, and
\[
 \Spec D(M)=\{91^{(1)},1^{(21)},(-4)^{(28)}\}.
\]
In particular, \(\delta^*(M)=7\) and \(\Phi(M)=3\).
\end{proposition}

\begin{proof}
The neighborhoods are
\begin{align*}
N(P_{i,j})&=\{P_{i,j-1},P_{i,j+1}\}
 \cup\{Q_{k,ik+j}:k\in\F\},\\
N(Q_{k,\ell})&=\{Q_{k,\ell-2},Q_{k,\ell+2}\}
 \cup\{P_{i,\ell-ik}:i\in\F\}.
\end{align*}
Each neighborhood contains seven distinct vertices. For the
common-neighbor count, first consider two \(P\)-vertices in the same layer, meaning that their
first coordinates agree.
The \(5\)-cycle contributes no common neighbor for an adjacent pair and
one for a nonadjacent pair; a common \(Q\)-neighbor would force their
second coordinates to agree.  For vertices in distinct \(P\)-layers, the
equation \((i-i')k=j'-j\) determines a unique common \(Q\)-neighbor.
The \(Q\)-vertices satisfy the same rule: within a layer the
\(5\)-cycle uses steps \(\pm2\), and between layers there is a unique
common \(P\)-neighbor.  Finally, for a cross pair
\(P_{i,j},Q_{k,\ell}\), put \(r=\ell-(ik+j)\).  The pair is adjacent
when \(r=0\); it has one common \(P\)-neighbor when
\(r\in\{\pm1\}\), and one common \(Q\)-neighbor when
\(r\in\{\pm2\}\).  These cases exhaust the five residues.
Since every nonadjacent pair has a common neighbor, the graph is connected
of diameter two; it is \(7\)-regular on \(50=7^2+1\) vertices, hence a
degree-\(7\) Moore graph of diameter two. The girth and diameter assertions
now follow from Theorem~\ref{thm:moore-threshold}.  For degree \(7\), the
multiplicity equations in its proof give adjacency multiplicities \(28\) at
\(2\) and \(21\) at \(-3\), and hence the displayed distance spectrum.
\end{proof}

\section{Distance spectra from adjacency eigenvalues}\label{sec:score-calculus}

We now express the distance matrix in terms of the adjacency matrix.
Let \(A_i\) have entry one on pairs of vertices at distance \(i\), and
zero elsewhere. In diameter three, the girth assumption expresses the
distance-two matrix as a quadratic in the adjacency matrix. The
distance-three matrix is what remains of the all-ones matrix after the
closer pairs are removed. The nonprincipal distance eigenvalues therefore
follow from the adjacency eigenvalues. We use this reformulation to
derive spectral obstructions to WOW-284: it gives the distance spectra of
the regular examples in Section~\ref{sec:examples} and underlies the order
bounds of Section~\ref{sec:lp}. The end of this section identifies its
equality case.

\begin{theorem}[Diameter-three score formula]\label{thm:diameter-three-score}
Let \(G\) be connected, \(k\)-regular, of girth at least five and diameter
three, with adjacency matrix \(A\) and order \(n\).  Then
\[
 D=3J+(k-3)I-2A-A^2,
\]
\[
 D+kI=3J+(2k-2)I-(A+I)^2.
\]
The principal distance eigenvalue is \(3n-k^2-k-3\), and a nonprincipal
adjacency eigenvalue \(\theta\) gives the distance eigenvalue
\[
 \mu(\theta)=k-2-(\theta+1)^2.
\]
Consequently
\[
 \resultbox{
 \Phi(G)=2k-2-\max_{\theta\ne k}(\theta+1)^2.
 }
\]
Thus \(G\) is a strict counterexample exactly when every nonprincipal
adjacency eigenvalue \(\theta\) satisfies
\[
 |\theta+1|<\sqrt{2k-2}
\]
\end{theorem}

\begin{proof}
We first count pairs at distance two.  Girth at least five means that
distinct vertices have at most one common neighbor, and adjacent vertices
have none.  The diagonal of the squared adjacency matrix counts the degree,
so \(A_2=A^2-kI\).  Because the diameter is three, the other distinct
nonadjacent pairs are at distance three; thus \(A_3=J-I-A-A_2\).
Putting these expressions into \(D=A+2A_2+3A_3\) proves both identities.

To compute the spectrum, apply these identities to adjacency eigenvectors.
Regularity and the identities make the constant vector an eigenvector of \(D\). Its
eigenvalue is positive and is the Perron root, because all off-diagonal
entries of the distance matrix are positive.  On \(\one^\perp\), \(J\) acts as zero, so each
nonprincipal adjacency eigenvalue gives the stated distance eigenvalue.
Taking the least of these values and using \(\delta^*=k\) gives the score formula.
\end{proof}

\begin{corollary}[Bipartite obstruction]\label{cor:bipartite}
A connected \(k\)-regular bipartite graph of girth at least five and diameter
three is not a strict counterexample for \(k\ge3\).
\end{corollary}

\begin{proof}
The nonprincipal eigenvalue \(-k\) satisfies
\[
 (-k+1)^2-(2k-2)=(k-1)(k-3)\ge0
 \qquad(k\ge3).
\]
Thus the eigenvalue \(-k\) lies on the boundary of the shifted window
when \(k=3\), and outside it when \(k>3\). In either case the strict
window condition fails.
\end{proof}

\subsection{The equality boundary and a graph attaining it}
\label{sec:controls}

The boundary of the shifted adjacency window corresponds to equality in
WOW-284. The following criterion identifies this boundary, and an explicit
graph shows that it can be attained.

\begin{theorem}[Equality boundary]\label{thm:equality-boundary}
Let \(G\) be connected, \(k\)-regular, of girth at least five and diameter
three.  Then
\[
 \Phi(G)=0
 \quad\Longleftrightarrow\quad
 \max_{\theta\ne k}|\theta+1|=\sqrt{2k-2}.
\]
Equivalently, \(D+kI\) is positive semidefinite and singular. When \(\Phi(G)=0\) and \(2k-2\) is
not a square, the two boundary adjacency eigenvalues occur with equal
multiplicity, so the distance eigenvalue \(-k\) has even multiplicity.  If
\(2k-2\) is a square, then \(k=2r^2+1\) for some integer \(r\).
\end{theorem}

\begin{proof}
This is immediate from Theorem~\ref{thm:diameter-three-score}.  In the
nonsquare case, the two boundary values are algebraic conjugates, so their
multiplicities in the integral characteristic polynomial agree.  In the square
case, \(2k-2=s^2\) forces \(s\) even.
\end{proof}

J{\o}rgensen's \(9\)-regular order-\(96\) graph of girth five and diameter
three \cite{Jorgensen2005} is an exact equality case: \(\delta^*=9\) and
\(\lambda_{\min}(D)=-9\) with multiplicity eight. Here \(2k-2=16\) is a
square, and the boundary adjacency eigenvalues \(3\) and \(-5\) occur with
multiplicities \(7\) and \(1\), so the equal-multiplicity conclusion of
Theorem~\ref{thm:equality-boundary} needs the nonsquare hypothesis.

\section{Smaller counterexamples from the Hoffman--Singleton graph}
\label{sec:examples}

We obtain smaller counterexamples from these coordinates in two ways.
First we delete a Petersen subgraph, followed by one or two further
vertices. We also take the graph induced by a distance-two sphere; this
second subconstituent gives another regular example.

Theorem~\ref{thm:explicit-examples} shows that these graphs, on \(38\),
\(39\), \(40\) and \(42\) vertices, are strict counterexamples. The regular
ones, of orders \(40\) and \(42\), have diameter three, so
Theorem~\ref{thm:diameter-three-score} computes their distance spectra from
their adjacency spectra; the irregular ones are certified by exact
computation. We first treat the second subconstituent of an arbitrary Moore
graph (Theorem~\ref{thm:second-subconstituent}).

Let
\[
 \mathcal P=\{P_{0,j},Q_{0,j}:j\in\F\}.
\]
The induced graph \(M[\mathcal P]\) is a Petersen graph.  Put
\[
 R=M-\mathcal P,
 \quad H_{39}=R-P_{1,0},
 \quad H_{38}=R-\{P_{1,0},P_{1,1}\},
\]
and let \(X_{42}\) be the second subconstituent of \(P_{0,0}\), namely the
graph induced by the vertices at distance two from it.

\subsection{Distance-two spheres in Moore graphs}

\begin{theorem}[Moore second subconstituents]\label{thm:second-subconstituent}
Let \(M\) be a degree-\(K\) Moore graph of diameter two, \(K\ge3\), and let
\(X\) be the graph induced by \(\Gamma_2(v)\) for a fixed vertex \(v\).  Then
\(X\) has order \(K(K-1)\), degree \(K-1\), girth at least five, diameter
three, and
\[
 \lambda_{\min}(D(X))=-\frac{5+\sqrt{4K-3}}2.
\]
Consequently
\[
 \Phi(X)=K-1-\frac{5+\sqrt{4K-3}}2,
\]
which is positive exactly for integers \(K\ge6\).
\end{theorem}

Each vertex in the second sphere loses its unique neighbor in the
first sphere. The order, degree and girth assertions follow directly.
For the spectrum, we group second-sphere vertices by their neighbors in
the first sphere and split the adjacency operator accordingly. We then
check that the induced graph has diameter three.  Appendix~\ref{app:small-deletion-spectra} gives the complete proof and
the multiplicity calculation.

\subsection{The explicit counterexamples}

\begin{theorem}\label{thm:explicit-examples}
The following are strict counterexamples.
\[
\begin{array}{c@{\quad}c@{\quad}c@{\quad}c}
\toprule
G&|V(G)|&\delta^*(G)&\lambda_{\min}(D(G))\\
\midrule
H_{38}&38&17/3&-3-\sqrt7\\
H_{39}&39&35/6&>-35/6\\
R&40&6&-5\\
X_{42}&42&6&-5\\
M&50&7&-4\\
\bottomrule
\end{array}
\]
The entry for \(H_{39}\) records an exact strict lower bound obtained from
positive definiteness of \(6D+35I\); it is not a decimal approximation to the
least eigenvalue.  Moreover, all \(40\) labeled singleton deletions of \(R\)
and all \(120\) labeled deletions of the endpoints of an edge of \(R\) are
strict counterexamples; within each family, the distance characteristic
polynomial is constant.
\end{theorem}

\begin{proof}
For the regular examples, we obtain the distance spectra from their
adjacency spectra. The graph \(R\) is the Anstee--Robertson graph, the unique \((6,5)\)-cage.
Here a cage is a regular graph of the prescribed degree and girth with the
fewest possible vertices.
O'Keefe and Wong proved minimality, and Wong proved uniqueness
\cite{OKeefeWong1979,Wong1979}; its realization as a Petersen deletion of the
Hoffman--Singleton graph also appears in
\cite[pp.~262--263]{KlinMuzychukZivAv2009}.
We derive its spectrum by separating the deleted and surviving vertices.
Write \(B\) for the adjacency matrix of the deleted Petersen graph,
\(A_R=A(R)\), and \(C\) for the \(10\times40\) matrix recording edges between
\(\mathcal P\) and \(R\).
A vertex of \(R\) has at most one neighbor in \(\mathcal P\): two adjacent
neighbors would make a triangle, while two nonadjacent Petersen vertices
already have their unique common neighbor in \(\mathcal P\).
There are \(10(7-3)=40\) edges between the two parts. Thus each vertex of
\(R\) has exactly one neighbor in \(\mathcal P\), and \(R\) is \(6\)-regular.

The Petersen identity \(B^2=2I-B+J\) and the blocks of the Moore identity give
\[
 CC^{\mathsf T}=4I,\qquad
 A_RC^{\mathsf T}=J_{40,10}-C^{\mathsf T}-C^{\mathsf T}B.
\]
The first identity shows that the image of \(C^{\mathsf T}\) has dimension ten.
The second transfers Petersen eigenvectors to this image and determines the
adjacency action there. Indeed, if \(Bz=\beta z\) and \(z\perp\one_{10}\),
then \(J_{40,10}z=0\), so \(A_RC^{\mathsf T}z=-(1+\beta)C^{\mathsf T}z\), and
\(C^{\mathsf T}z\ne0\) because \(CC^{\mathsf T}=4I\); since
\(C^{\mathsf T}\one_{10}=\one_{40}\), the same identity gives
\(A_R\one_{40}=(10-1-3)\one_{40}\). The Petersen eigenvalues \(1^{(5)}\) and \((-2)^{(4)}\) give eigenvalues
\((-2)^{(5)}\) and \(1^{(4)}\) on this image; the constant vector gives
the eigenvalue \(6\). Its orthogonal complement is \(\ker C\).
Since \(\one_{40}=C^{\mathsf T}\one_{10}\), \(J\) vanishes on \(\ker C\), and
\(\ker C\) is \(A_R\)-invariant because the symmetric matrix \(A_R\) preserves
the image of \(C^{\mathsf T}\).
On this \(30\)-dimensional space, the block
\(A_R^2+C^{\mathsf T}C=6I-A_R+J\) gives
\[
 A_R^2+A_R-6I=0.
\]
Hence its eigenvalues are \(2\) and \(-3\). The dimension and trace
equations give multiplicities \(18\) and \(12\), respectively: if \(a,b\) are these
multiplicities, then \(a+b=30\) and
\(6+4\cdot1+5\cdot(-2)+2a-3b=\tr A_R=0\). We obtain
\[
 \Spec A_R=\{6^{(1)},2^{(18)},1^{(4)},(-2)^{(5)},(-3)^{(12)}\}.
\]

To apply the distance polynomial, we must also establish the diameter.
Let \(x,y\in V(R)\) be nonadjacent. If their common neighbor in \(M\)
survives, they remain at distance two. Otherwise they share a deleted
Petersen neighbor. Choose \(a\in N_R(x)\). Girth five makes \(a,y\)
nonadjacent, and makes the Petersen neighbor of \(a\) different from that
of \(y\). Their unique common neighbor therefore lies in \(R\), giving a
path of length three from \(x\) to \(y\).
Each deleted vertex has four neighbors in \(R\), so pairs whose common
neighbor was deleted do occur. Thus \(R\) is connected with diameter three.
Theorem~\ref{thm:diameter-three-score} now gives
\[
 \Spec D(R)=\{75^{(1)},3^{(5)},0^{(16)},(-5)^{(18)}\},
 \qquad \Phi(R)=1.
\]
For \(X_{42}\), the proof of Theorem~\ref{thm:second-subconstituent}
(Appendix~\ref{app:small-deletion-spectra}) gives diameter three and
\(\Spec A(X_{42})=\{6^{(1)},2^{(21)},(-1)^{(6)},(-3)^{(14)}\}\), so
Theorem~\ref{thm:diameter-three-score} gives
\[
 \Spec D(X_{42})=\{81^{(1)},4^{(6)},0^{(14)},(-5)^{(21)}\}.
\]
The classical second-subconstituent identification and adjacency spectrum are
recorded in \cite[Table~3, p.~265]{vanDamHaemers2003}.
For the irregular examples, we check strictness directly. A degree count
for \(H_{38}\) gives \(\delta^*=17/3\). The characteristic-polynomial factor
\(x^2+6x+2\) has the root \(-3-\sqrt7\), and exact Sturm isolation shows
that no smaller root occurs. For \(H_{39}\), the exact matrix \(6D+35I\)
is positive definite, giving the required strict eigenvalue bound.  The exact checks verify the graph hypotheses, distances and dual degrees.
Sturm counts isolate the real roots, while rational \(LDL^{\mathsf T}\)
factorizations determine definiteness by an invertible triangular change of
coordinates. These certificates are checked by exact computation.  The characteristic polynomials for the two labeled deletion families are
recorded in Appendix~\ref{app:small-example-certificates}.
The labeled deletion families are enumerated exhaustively by computer.
No transitivity or numerical root ordering is used.  Fixed graph6,
adjacency-list, and edge-list records of these graphs are part of the data
described in Appendix~\ref{sec:verification}.
\end{proof}

\Needspace{6\baselineskip}
\section{The diameter of regular counterexamples}\label{sec:obstructions}

Theorem~\ref{thm:diameter-three-score} applies only in diameter three. This
section shows that diameter three is the main case: a regular strict
counterexample has diameter two, three or four, and diameter four requires
degree at least ten (Corollary~\ref{cor:trichotomy}). Since
Theorem~\ref{thm:moore-threshold} settles Moore graphs, the order bounds of
Section~\ref{sec:lp} concentrate on diameter three.
A geodesic, or shortest path, gives an upper bound on the least distance
eigenvalue. We begin with the geodesic bound.

\begin{lemma}\label{lem:diam-rayleigh}
For every connected graph,
\[
 \lambda_{\min}(D(G))\le-\diam(G).
\]
\end{lemma}

\begin{proof}
For a diametral pair \(u,v\), the Rayleigh quotient of \(e_u-e_v\) is
\(-d_G(u,v)\).
\end{proof}

\Needspace{8\baselineskip}
\begin{theorem}[Endpoint-neighborhood obstruction]\label{thm:endpoint-diameter}
Let \(G\) be any connected finite simple graph, and let \(u,v\) be vertices at
distance \(\ell=d_G(u,v)\ge5\).  Put \(p=d(u)\) and \(q=d(v)\).  Then
\[
 \resultbox{
 \lambda_{\min}(D(G))
 \le p+q-2-\sqrt{(p-q)^2+pq(\ell-2)^2}.
 }
\]
If \(\delta\) is the ordinary minimum degree, then
\[
 \resultbox{
 \lambda_{\min}(D(G))\le-\delta(\ell-4)-2.
 }
\]
Consequently every strict WOW-284 counterexample satisfies
\[
 \Delta>\delta(\ell-4)+2,
\]
where \(\Delta\) is the ordinary maximum degree.
In particular, every regular strict counterexample has diameter at most four.
\end{theorem}

\begin{proof}
Choose a test vector \(x\) with opposite signs on the two endpoint
neighborhoods, which are disjoint. Give weight \(a>0\) to \(N(u)\),
weight \(-b<0\) to \(N(v)\), and zero elsewhere.  Distances within a neighborhood are at
most two; distances between the neighborhoods are at least \(\ell-2\).
The cross-neighborhood terms have negative weights. Thus both distance
estimates give an upper bound on the Rayleigh quotient:
\[
 x^{\mathsf T}D(G)x\le2p(p-1)a^2+2q(q-1)b^2-2(\ell-2)pq\,ab,\qquad
 \|x\|^2=pa^2+qb^2,
\]
and \((y_1,y_2)=(a\sqrt p,b\sqrt q)\) gives a Rayleigh quotient for
\[
 \begin{pmatrix}
 2(p-1)&-(\ell-2)\sqrt{pq}\\
 -(\ell-2)\sqrt{pq}&2(q-1)
 \end{pmatrix}.
\]
The least eigenvalue gives the first bound.  Its eigenvector can be chosen
with both coordinates positive, as required by our sign assignment.
To replace the endpoint degrees by the minimum degree, write
\(p=\delta+\alpha\), \(q=\delta+\beta\), with
\(\alpha,\beta\ge0\), and put \(t=\ell-2\).  The right-hand side of the identity
\[
 (p-q)^2+pqt^2-(p+q+\delta(t-2))^2
 =(t-2)\{\delta t(\alpha+\beta)+(t+2)\alpha\beta\}
\]
is nonnegative.  Since \(p+q+\delta(t-2)>0\), comparison of the nonnegative
square roots gives the second bound.  Finally \(\delta^*(G)\le\Delta\).
For a \(k\)-regular graph \(\Delta=\delta=k\), and \(k>k(\ell-4)+2\) fails
for every \(\ell\ge5\); hence no two vertices of a regular strict
counterexample are at distance five or more.
\end{proof}

\begin{theorem}[Diameter four]\label{thm:diameter-four}
Let \(G\) be connected, \(k\)-regular, of girth at least five and diameter
four.  Then
\[
 \resultbox{
 \lambda_{\min}(D(G))\le-\frac{7+\sqrt{16k+1}}2.
 }
\]
Consequently no such graph of degree \(2\le k\le9\) is a strict
counterexample.
\end{theorem}

\begin{proof}
Choose \(u,v\) at distance four and put \(U=N(u)\), \(V=N(v)\).  For fixed
\(a\in U\), every \(b\in V\) with \(d_G(a,b)=2\) has a common neighbor in
\(N(a)\setminus\{u\}\).  Distinct such vertices \(b,b'\) require distinct
common neighbors, since otherwise \(b-c-b'-v-b\) is a \(4\)-cycle.  Thus at
most \(k-1\) vertices of \(V\) are at distance two from each \(a\).  No pair
in \(U\times V\) is adjacent, and every remaining pair has distance at least
three.  Therefore
\[
 \sum_{a\in U,b\in V}d_G(a,b)\ge2k^2+k.
\]
Use this cross-distance estimate in a test vector.
For \(\alpha,\beta>0\), assign weights \(\alpha,\beta,-\alpha,-\beta\)
to \(u,U,v,V\), respectively. Count each unordered pair with the product
of its two weights, then double the sum.  The result is
\[
 \frac{x^{\mathsf T}D(G)x}{x^{\mathsf T}x}
 \le
 \frac{-4\alpha^2-4k\alpha\beta-3k\beta^2}
 {\alpha^2+k\beta^2}.
\]
After setting \(y_1=\alpha\) and \(y_2=\sqrt{k}\,\beta\), the
right-hand side is the Rayleigh quotient of
\[
 \begin{pmatrix}-4&-2\sqrt{k}\\-2\sqrt{k}&-3\end{pmatrix},
\]
whose least eigenvalue is the displayed value. A minimizing vector can be chosen with
positive coordinates.  The strict comparison with \(-k\) holds for
\(2\le k\le9\); for \(4\le k\le9\), where \(2k-7>0\), it is
\(k^2-11k+12<0\).
\end{proof}

\begin{corollary}[Regular trichotomy]\label{cor:trichotomy}
Every regular strict counterexample has one of the following forms:
\begin{enumerate}
\item diameter two, hence a Moore graph;
\item diameter three, with
\(\lvert\theta+1\rvert<\sqrt{2k-2}\) for every nonprincipal adjacency
eigenvalue \(\theta\);
\item diameter four, with degree at least ten.
\end{enumerate}
There are no regular strict counterexamples of diameter at least five.
\end{corollary}

\begin{proof}
Girth excludes diameter one; in diameter two,
Proposition~\ref{prop:radius-two} gives \(n=k^2+1\); diameters three, four
and at least five are Theorems~\ref{thm:diameter-three-score},
\ref{thm:diameter-four} and~\ref{thm:endpoint-diameter}.
\end{proof}

\section{Order bounds from the spectral window}\label{sec:lp}

For a regular graph of girth at least five whose nonprincipal adjacency
eigenvalues lie in the window of Theorem~\ref{thm:diameter-three-score},
traces of polynomials in the adjacency matrix bound the order. This section
proves four such bounds. The fourth spectral moment gives \(n<B_k\)
(Proposition~\ref{prop:moment-bound}). No polynomial of higher degree in the
nonbacktracking basis improves this bound, which holds under the closed
window without any diameter assumption (Theorem~\ref{thm:lp-ceiling}).
Evaluating the optimal polynomial at the adjacency matrix gives an integral
matrix and a sharper bound in diameter three
(Theorem~\ref{thm:integral-slack}); for degree at least six, a case analysis
of this matrix gives the three-to-one bound
(Theorem~\ref{thm:three-to-one}). Section~\ref{sec:degree-six} applies these
bounds in low degree.

\subsection{The fourth-moment bound}

We first derive an order bound from the fourth spectral moment.
Let \(G\) satisfy the hypotheses of Theorem~\ref{thm:diameter-three-score}, and
write its nonprincipal adjacency eigenvalues as \(\theta_i\).  Put
\(y_i=\theta_i+1\).

\begin{proposition}[Fourth-moment identity]\label{prop:moment-bound}
One has
\[
 \sum_{i=1}^{n-1}(2k-2-y_i^2)(y_i+1)^2
 =(k+2)\bigl((k+2)(k^2+3)-6n\bigr).
\]
Every strict counterexample therefore satisfies
\[
 \resultbox{
 n< B_k:=\frac{(k+2)(k^2+3)}6.
 }
\]
\end{proposition}

\begin{proof}
The shifted WOW window makes every summand nonnegative, so the identity
will bound the order. With \(y=\theta+1\), each summand is a quartic in
\(\theta\):
\[
 (2k-2-y^2)(y+1)^2=-\theta^4-6\theta^3+(2k-15)\theta^2+(8k-20)\theta+8k-12.
\]
Girth at least five gives \(\tr A=\tr A^3=0\), \(\tr A^2=nk\) and
\(\tr A^4=nk(2k-1)\), since a closed \(4\)-walk must retrace an edge.
Explicitly, a closed \(4\)-walk at \(u\) either returns to \(u\) after two
steps (\(k^2\) walks) or has the form \(u\,a\,c\,a\,u\) with \(c\ne u\)
(\(k(k-1)\) walks); any other closed \(4\)-walk would traverse a
\(4\)-cycle. Summed over all \(n\) eigenvalues, the quartic gives
\[
 -nk(2k-1)+(2k-15)nk+(8k-12)n=-6n(k+2);
\]
the principal eigenvalue \(\theta=k\), that is \(y=k+1\), contributes
\(-(k+2)^2(k^2+3)\), and subtracting it gives the identity.
In a strict counterexample, each factor \(2k-2-y_i^2\) is positive.
To obtain a strict order bound, we must also exclude a zero sum: otherwise every nonprincipal adjacency eigenvalue would equal
\(-2\), and \(\tr A=0\) would give
\(k-2(n-1)=0\), or \(n=(k+2)/2\), incompatible with the elementary
bound \(n\ge k+1\) for a simple \(k\)-regular graph.
\end{proof}

\subsection{Nonbacktracking polynomials and the optimal bound}

The summand in the proof of Proposition~\ref{prop:moment-bound} is \(-6(k+2)f_*(\theta)\) for the quartic \(f_*\) of
Theorem~\ref{thm:lp-ceiling}. We now ask whether a higher-degree polynomial
can improve the fourth-moment bound. We minimize the order
bound subject to the coefficient and sign conditions below; this is the
one-variable nonbacktracking linear program (LP). To prove its optimum, we give
an attaining polynomial and a positive weighted sum of point evaluations,
the dual measure, that bounds every admissible polynomial. The same
measure determines when equality can hold.

The linear program uses the following basis of nonbacktracking
polynomials; they also express the distance matrix in higher diameter.

\begin{proposition}[Higher-diameter transfer]\label{prop:higher-transfer}
Let \(G\) be connected and \(k\)-regular, with diameter \(d\) and girth at
least \(2d-1\).  Let \(A_i\) denote the distance-\(i\) matrix, and define
\[
 F_0=1,
 \quad F_1=x,
 \quad F_2=x^2-k,
 \quad F_i=xF_{i-1}-(k-1)F_{i-2}\quad(i\ge3).
\]
Then \(A_i=F_i(A)\) for \(0\le i\le d-1\), and
\[
 D=dJ+q_d(A),
 \qquad
 q_d(x)=\sum_{i=0}^{d-1}(i-d)F_i(x).
\]
In particular,
\[
 q_3(x)=k-3-2x-x^2,
\]
\[
 q_4(x)=-x^3-2x^2+(2k-4)x+2k-4.
\]
\end{proposition}

\begin{proof}
The polynomial matrices count nonbacktracking walks: walks that never immediately
reverse an edge at an internal step \cite[Theorem~1]{Nozaki2015}.
Under the girth condition, these counts agree with the distance matrices. For
\(1\le i\le d-1\), a repeated vertex in such a walk would produce a
cycle of length at most \(i\), so the walk is a path. A shorter path
between the same endpoints would produce a cycle of length at most
\(2i-1\); two distinct paths of length \(i\) would produce one of length
at most \(2i\). These lengths are below \(2d-1\), so neither is
possible. Hence there is exactly one such walk when the endpoints have
distance \(i\), and none otherwise. The case \(i=0\) is immediate,
giving \(A_i=F_i(A)\) for \(0\le i\le d-1\).
In the sum \(D=\sum_{i=0}^d iA_i\), the identity \(A_d=F_d(A)\) need not hold.
Eliminate \(A_d\) using \(J=\sum_{i=0}^d A_i\).
This proves the formula within the established distance-polynomial
framework \cite{HowladerPanigrahi2022,Fiol2016}.
\end{proof}

\begin{definition}
Let \(F_i=F_i^{(k)}\) be the nonbacktracking polynomials from
Proposition~\ref{prop:higher-transfer}, and set
\[
 I_k=[-1-\sqrt{2k-2},-1+\sqrt{2k-2}].
\]
A finite polynomial \(f=\sum_i f_iF_i\) is \emph{admissible} if
\[
 f_0>0,
 \qquad f_i\ge0\quad(i\ge5),
 \qquad f(x)\le0\quad(x\in I_k).
\]
\end{definition}

For a connected \(k\)-regular graph of girth at least five whose nonprincipal spectrum
lies in \(I_k\), one has \(\tr F_i(A)=0\) for \(1\le i\le4\), while
\(\tr F_i(A)\ge0\) for \(i\ge5\), since these traces count closed
nonbacktracking walks \cite[Theorem~1 and Corollary~1]{Nozaki2015}.
The coefficient conditions therefore give
\(nf_0\le\tr f(A)\).  On the other hand, \(f(\theta)\le0\) for every
nonprincipal eigenvalue, so \(\tr f(A)\le f(k)\).  Thus
\[
 nf_0\le\tr f(A)\le f(k),
 \qquad n\le\frac{f(k)}{f_0}.
\]

\begin{theorem}[Exact LP ceiling and rigidity]\label{thm:lp-ceiling}
For every integer \(k\ge4\) and every admissible \(f\),
\[
 \resultbox{
 \frac{f(k)}{f_0}\ge B_k=\frac{(k+2)(k^2+3)}6.
 }
\]
Equality holds if and only if \(f\) is a positive scalar multiple of
\[
 \resultbox{
 f_*(x)=\frac{(x+2)^2(x^2+2x-(2k-3))}{6(k+2)}.
 }
\]
Thus increasing the polynomial degree cannot improve this one-variable LP
bound.  Consequently, any connected \(k\)-regular graph of girth at least
five whose nonprincipal spectrum lies in \(I_k\) satisfies \(n<B_k\).
\end{theorem}

\begin{proof}
First expand the stated quartic in the nonbacktracking basis to check
that it attains the bound:
\[
\begin{aligned}
6(k+2)f_*(x)={}&6(k+2)F_0(x)+2(2k+7)F_1(x)\\
&+(k+13)F_2(x)+6F_3(x)+F_4(x).
\end{aligned}
\]
The constant coefficient is positive, and there are no coefficients above
degree four.  On \(I_k\), the factor \((x+1)^2-(2k-2)\) is nonpositive, while
\((x+2)^2\ge0\); hence \(f_*\) is admissible and \(f_*(k)=B_k\).

To prove that no admissible polynomial does better, we use a positive
measure supported on the two window endpoints and the interior zero of the
quartic.  Put \(\omega=\sqrt{2k-2}\),
\(\xi_\pm=-1\pm\omega\), and \(\xi_0=-2\), and define
\begin{align*}
w_-&=\frac{k(k+2)(2k^2-6-3(k-1)\omega)}{24(2k-3)},\\
w_0&=\frac{k(k-1)(k^2+3)}{6(2k-3)},\\
w_+&=\frac{k(k+2)(2k^2-6+3(k-1)\omega)}{24(2k-3)}.
\end{align*}
\(w_-,w_0,w_+>0\).  For the only nontrivial inequality, this follows from
\[
 (2k^2-6)^2-18(k-1)^3
 =2(k-3)(2k-3)(k^2+3)>0.
\]
Each delta denotes a unit point mass at its subscript. Integration against
the measure
\(\mu=w_-\delta_{\xi_-}+w_0\delta_{\xi_0}+w_+\delta_{\xi_+}\) is therefore
a weighted sum of three evaluations. This measure satisfies
\[
 \mu(1)=B_k-1,
 \qquad \mu(F_i)=-F_i(k)\quad(1\le i\le4).
\]
To check these identities, put \(y=x+1\). The support becomes
\(\{-\omega,-1,\omega\}\), so
\[
\begin{gathered}
 \int y^j\,d\mu=(-1)^jw_0+\bigl(w_++(-1)^jw_-\bigr)\omega^j,\\
 w_++w_-=\frac{k(k+2)(k^2-3)}{6(2k-3)},\qquad
 w_+-w_-=\frac{k(k-1)(k+2)}{4(2k-3)}\,\omega.
\end{gathered}
\]
Using \(\omega^2=2k-2\), one finds \(\int d\mu=k(k^2+2k+3)/6=B_k-1\) and
\[
\begin{aligned}
 \int y\,d\mu&=\tfrac16k(k-1)(k+3), &
 \int y^2\,d\mu&=\tfrac16k(k-1)(k+1)(k+3),\\
 \int y^3\,d\mu&=\tfrac16k(k-1)(3k^2+4k-3), &
 \int y^4\,d\mu&=\tfrac16k(k-1)(2k^3+5k^2-2k-9).
\end{aligned}
\]
Writing \(F_1,\dots,F_4\) in powers of \(y\) and using \(F_i(k)=k(k-1)^{i-1}\)
gives \(\mu(F_i)=-F_i(k)\).
The measure cancels the first four nonconstant basis terms. For the remaining
terms, define the slack \(a_i=\mu(F_i)+F_i(k)\). We need this slack to be
strictly positive for \(i\ge5\). Appendix~\ref{app:lp-slack-positivity}
verifies this positivity: the first
five values follow from exact polynomial expansions, and a Chebyshev bound
handles every later term. For an arbitrary admissible polynomial, expansion
in the nonbacktracking basis gives
\[
 \int f\,d\mu
 =B_kf_0-f(k)+\sum_{i\ge5}f_i a_i
 \ge B_kf_0-f(k).
\]
Here \(\int f\,d\mu=\sum_if_i\mu(F_i)\); by the identities above,
\(F_0(k)=1\) and \(\mu(F_i)=a_i-F_i(k)\) for \(i\ge5\), the terms
\(-f_iF_i(k)\), \(i\ge0\), add up to \(-f(k)\).
The coefficient condition and the slack positivity in
Appendix~\ref{app:lp-slack-positivity} make the last sum nonnegative.
Also,
\(f\le0\) on \(\operatorname{supp}\mu\), so \(\int f\,d\mu\le0\).
These inequalities prove the lower bound for every admissible polynomial.
Equality must hold in both inequalities.  Since every high-degree slack
is strictly positive, \(f_i=0\) for \(i\ge5\).  The measure has positive
weight at each support point, so \(f\) must vanish at
\(\xi_-,-2,\xi_+\).  Its interior zero \(-2\) has even multiplicity,
because \(f\le0\) on \(I_k\).  A polynomial of degree at most four with
these zeros is a scalar multiple of \(f_*\), and \(f_0>0\) makes that
scalar positive.

The polynomial bound includes the endpoints of \(I_k\), so the trace bound
gives only \(n\le B_k\). Equality \(n=B_k\) would force every nonprincipal
eigenvalue to be a zero of \(f_*\), that is, to lie in
\(\{-2,-1\pm\omega\}\); the trace and second-moment equations then give a
negative or nonintegral multiplicity. Appendix~\ref{app:lp-slack-positivity}
carries out this computation, so \(n<B_k\) throughout the closed window.
\end{proof}

\subsection{The integral excess matrix}

\begin{theorem}[Integral optimal-slack bound]\label{thm:integral-slack}
Let \(G\) be connected and \(k\)-regular with \(k\ge4\), of girth at least
five and diameter three, and suppose
\[
 \lvert\theta+1\rvert<\sqrt{2k-2}
 \qquad(\theta\ne k)
\]
for every adjacency eigenvalue \(\theta\).  Define
\[
 g_k(x)=(x+2)^2\bigl((x+1)^2-(2k-2)\bigr),\quad
 C_k=(k+2)^2(k^2+3),\quad h_k=6(k+2),
\]
\[
 \mathcal S_k=-g_k(A)+\frac{C_k}{n}J.
\]
Then
\[
 \mathcal S_k\succeq0,\qquad
 \mathcal S_k\one=0,\qquad
 \tr\mathcal S_k=h_k(B_k-n).
\]
Moreover,
\[
 \mathcal E_k=g_k(A)-(h_k+1)J+I
\]
is a nonzero, symmetric, entrywise nonnegative integral matrix with zero
diagonal and constant row sum
\[
 \varepsilon_{k,n}=C_k-(h_k+1)n+1.
\]
On \(\one^\perp\), one has \(\mathcal S_k=I-\mathcal E_k\).  In particular,
\[
 \resultbox{
 n\le
 \left\lfloor
 \frac{(k+2)^2(k^2+3)}{6(k+2)+1}
 \right\rfloor.
 }
\]
\end{theorem}

\begin{proof}
Apply the scaled optimal LP polynomial \(g_k=6(k+2)f_*\) from
Theorem~\ref{thm:lp-ceiling} to the adjacency matrix. The correction by the
all-ones matrix makes \(\mathcal S_k\) vanish on \(\langle\one\rangle\).
On a nonprincipal
\(\theta\)-eigenspace, its eigenvalue is
\[
 \bigl(2k-2-(\theta+1)^2\bigr)(\theta+2)^2\ge0.
\]
To compute the diagonal, use the nonbacktracking expansion of \(g_k\),
whose constant coefficient is \(h_k\).  Girth at least five implies
\((F_i(A))_{uu}=0\) for every \(u\) and \(1\le i\le4\).  Thus
\[
 \tr\mathcal S_k=C_k-h_kn=h_k(B_k-n),
 \qquad (g_k(A))_{uu}=h_k.
\]

We obtain a stronger order restriction from the integer off-diagonal
entries. Put
\(z_{uv}=(g_k(A))_{uv}\in\mathbb Z\) for \(u\ne v\), and set
\(a=C_k/n-h_k=(\mathcal S_k)_{uu}\).  Nonnegativity of the \(2\times2\)
principal minor on \(\{u,v\}\) gives
\[
 \left|\frac{C_k}{n}-z_{uv}\right|\le a.
\]
We must exclude \(C_k/n-z_{uv}=a\), equivalently \(z_{uv}=h_k\). This would imply
\((e_u-e_v)^{\mathsf T}\mathcal S_k(e_u-e_v)=0\), and positive semidefiniteness would then
put \(e_u-e_v\in\ker\mathcal S_k\). The strict spectral window gives
\[
 \ker\mathcal S_k=\langle\one\rangle\oplus E_{-2}(A).
\]
Since \(e_u-e_v\perp\one\), it would be a \((-2)\)-eigenvector of \(A\).
Yet the \(u\)-coordinate of \(A(e_u-e_v)\) is \(-1\) if \(u\sim v\) and
\(0\) otherwise, never \(-2\).  Therefore
\[
 z_{uv}\ge h_k+1.
\]
This proves the stated entrywise properties of \(\mathcal E_k\).
Its row sum follows from \(g_k(A)\one=C_k\one\); on \(\one^\perp\),
the term \(J\) vanishes and gives the claimed identity.

It remains to prove \(\mathcal E_k\ne0\). Otherwise the adjacency operator
on the rational space \(\one^\perp\) is annihilated by the quartic
\(g_k(x)+1\), which is irreducible over \(\mathbb Q\); this forces
\(4\mid n-1\), and \(\tr A=0\) then gives \(n-1=2k/3<k\), which is
impossible. Appendix~\ref{app:lp-slack-positivity} gives the details. Thus
\(\mathcal E_k\ne0\).  Its constant row sum is consequently a positive
integer, so \(\varepsilon_{k,n}\ge1\), which is equivalent to the displayed
order bound.
\end{proof}

\subsection{The three-to-one excess bound}

Up to a positive factor, the trace of \(\mathcal S_k\) is the gap between
the graph order and the scalar LP bound. Its entries contain further
information: \(2\times2\) principal minors constrain individual entries,
and larger minors relate several of them. In the integral excess matrix,
these constraints become arithmetic restrictions. For degree at least six
they give the following bound.

\begin{theorem}[Three-to-one excess bound]\label{thm:three-to-one}
Under the hypotheses of Theorem~\ref{thm:integral-slack}, assume \(k\ge6\)
and write \(\varepsilon=\varepsilon_{k,n}\).  The integral parameter
\[
 r=2\varepsilon-n-2
   =2(k+2)^2(k^2+3)-(12k+27)n
\]
satisfies
\[
 r>0,\qquad n\le3r.
\]
Consequently
\[
 \resultbox{
 n\le\left\lfloor
 \frac{3(k+2)^2(k^2+3)}{18k+41}
 \right\rfloor.
 }
\]
\end{theorem}

\paragraph{Outline of the proof.}
Appendix~\ref{app:three-to-one-proof} gives the complete case analysis. The
\(2\times2\) principal minors of \(\mathcal S_k=I-\mathcal E_k+\rho J\), where
\(\rho=(\varepsilon_{k,n}-1)/n\), bound the off-diagonal entries of
\(\mathcal E_k\). If \(r\le0\), or if \(n>3r\), these bounds, together with a
kernel argument or with a Gram-determinant estimate and a divisor check, make
\(\mathcal E_k\) the adjacency matrix of a simple graph whose complement is
regular with least eigenvalue at least \(-2\). Since \(n\ge k^2+2\ge38\), the
classification of such graphs \cite{CameronEtAl1976}, together with counting
arguments, reduces the complement and its components to complete graphs,
cocktail-party graphs and line graphs. The remaining parameter families are
excluded by the radius-two bound, by divisibility conditions on \(4k+9\) and
\(18k+41\), by the parity of \(kn\), or because a rational invariant space on
which the adjacency operator satisfies an irreducible quartic would have
dimension not divisible by four.

\section{Degree at least six and order at most fifty}
\label{sec:degree-six}

We now combine the diameter restrictions of Section~\ref{sec:obstructions}
with the order bounds of Section~\ref{sec:lp}. Small degree gives another
obstruction, because too few vertices are available to satisfy the spectral
window: every regular strict counterexample has degree at least six
(Theorem~\ref{thm:regular-degree-six}), and in degree six it has at most
\(50\) vertices (Theorem~\ref{thm:degree-six-fifty}). For the low-degree
argument we also use the fourth-moment bound of
Proposition~\ref{prop:moment-bound}.

\subsection{Degree at least six}

\begin{theorem}\label{thm:regular-degree-six}
Every connected regular strict counterexample to WOW-284 has degree at least
six.
\end{theorem}

\begin{proof}
Let the degree of \(G\) be \(k\). Lemma~\ref{lem:diam-rayleigh} and
strictness first give
\(\diam(G)<k\).  This already excludes \(k\le2\): the order satisfies
\(|V(G)|\ge3\), and the girth hypothesis excludes completeness, so
\(\diam(G)\ge2\ge k\).  For \(k=3\), the radius-two lower bound and
\(\diam(G)\le2\) force equality in the Moore bound.
Theorem~\ref{thm:moore-threshold} then gives \(\Phi(G)=0\), contradicting
the required \(\Phi(G)>0\).

For \(k=4\), the two possible diameters lead to different obstructions.
Diameter two would require a degree-four Moore graph, but its adjacency
multiplicities are nonintegral.  For diameter three,
Proposition~\ref{prop:moment-bound} gives \(n<B_4=19\).  The radius-two ball has
\(17\) vertices, and the third layer is nonempty, so \(n=18\).
Every vertex then has exactly one vertex at distance three: the
distance-three matrix is a perfect matching.
Its \(-1\)-eigenspace \(W=\ker(A_3+I)\) is a nine-dimensional rational
subspace of \(\one^\perp\), spanned by the nine vectors \(e_u-e_{u'}\), one
for each pair \(\{u,u'\}\) at distance three. By the proof of
Theorem~\ref{thm:diameter-three-score}, \(A_3\) is a polynomial in \(A\) and
\(J\), and \(AJ=JA\).  Since \(A\) commutes with \(A_3\), the space
\(W\) is \(A\)-invariant, and
\[
 A_3=J+3I-A-A^2
 \quad\Longrightarrow\quad
 (A^2+A-4I)|_W=0.
\]
The polynomial \(x^2+x-4\) is irreducible over \(\mathbb Q\), so the
characteristic polynomial of the restricted operator must be a power of this
quadratic. The dimension of the rational space must therefore be even, a
contradiction.

For \(k=5\), diameter two again fails the Moore multiplicity condition.
Diameter four can be excluded directly: a diametral geodesic gives the
principal submatrix \(D(P_5)\). Its characteristic-polynomial factor \(x^2+6x+4\) supplies
the eigenvalue \(-3-\sqrt5<-5\).  Cauchy interlacing therefore contradicts
strictness.  In diameter three, Meringer's enumeration
\cite[p.~142]{Meringer1999} (the \((5,5)\)-cages, of order \(30\), are the
smallest \(5\)-regular graphs of girth five) and Proposition~\ref{prop:moment-bound}
(\(B_5=98/3\)) leave
\(n\in\{30,31,32\}\).  The handshake identity \(5n=2|E(G)|\) gives
\(n\equiv0\pmod2\), leaving \(n\in\{30,32\}\). Both orders are excluded in Appendix~\ref{app:degree-five}: at \(n=32\)
by compressing the adjacency matrix to the distance layers of a vertex,
which forces \(\theta_2(A)>-1+\sqrt8\) for the second-largest adjacency
eigenvalue \(\theta_2(A)\), contrary to
Theorem~\ref{thm:diameter-three-score}; and at \(n=30\) by the four
\((5,5)\)-cages, each of which has a distance eigenvalue at most \(-5\).
\end{proof}

\subsection{Order bounds in degrees six through nine}

\begin{theorem}\label{thm:degree-six-fifty}
Every connected \(6\)-regular strict counterexample to WOW-284 has order at
most \(50\).
\end{theorem}

\begin{proof}
First establish diameter three, as required by the order bound.
Theorem~\ref{thm:endpoint-diameter} excludes diameter at least five, and
Theorem~\ref{thm:diameter-four} excludes diameter four. Diameter two would
attain the Moore bound, giving order \(37\) and adjacency characteristic polynomial
\((x-6)(x^2+x-5)^{18}\); its root sum is \(-12\), contradicting
\(\tr A=0\).  Diameter one is excluded by the girth hypothesis.  Thus the
diameter is three. Strictness and Theorem~\ref{thm:diameter-three-score}
give \(|\theta+1|<\sqrt{10}\) for every nonprincipal adjacency eigenvalue,
so the hypotheses of Theorem~\ref{thm:integral-slack} hold with \(k=6\),
\(C_6=8^2\cdot39=2496\) and \(h_6+1=49\); already that theorem gives
\[
 n\le\left\lfloor\frac{8^2\cdot39}{49}\right\rfloor=50;
\]
Theorem~\ref{thm:three-to-one} gives the same value, so this bound does not
depend on the line-graph classification.
\end{proof}

\begin{corollary}[Low-degree order windows]\label{cor:low-degree-windows}
There is no regular strict counterexample of degree at most five.  In degrees
six through nine, every regular strict counterexample satisfies
\begin{align*}
k=6&:\quad n\le50,\\
k=7&:\quad n=50\text{ in diameter two, or }n\le74\text{ in diameter three},\\
k=8&:\quad n\le108,\\
k=9&:\quad n\le150.
\end{align*}
\end{corollary}

\begin{proof}
The degree exclusion is Theorem~\ref{thm:regular-degree-six},
Theorem~\ref{thm:endpoint-diameter} excludes diameter at least five, and
Theorem~\ref{thm:diameter-four} removes diameter four for \(k\le9\).
 For diameter three, Theorem~\ref{thm:three-to-one} gives, for
\(k=6,7,8,9\), respectively,
\[
 n\le50,\ 75,\ 108,\ 150.
\]
Since \(7n=2|E(G)|\), the degree-seven bound improves to \(74\).  The
diameter-two alternatives are determined by the Moore multiplicities: only
the degree-seven, order-\(50\) Hoffman--Singleton case survives.
(For \(k\in\{6,8,9\}\), \(4k-3\in\{21,29,33\}\) is not a square, so the
multiplicities in the proof of Theorem~\ref{thm:moore-threshold} are
irrational.)
\end{proof}

Appendix~\ref{app:three-to-one-proof} lists the corresponding order bounds
through degree twenty, including the parity adjustment for odd degrees.

The degree-six bound still permits \(n=50\).
Appendix~\ref{app:order50-details} restricts the adjacency multiplicities
and cycle counts at this order, and studies the signed graph obtained
from the excess matrix. It also bounds the number of five-cycles through an
edge, which gives a second, local proof that a \(6\)-regular strict
counterexample cannot have \(51\) vertices. These restrictions are
necessary conditions;
they leave the existence of a strict counterexample unresolved.

\section{Vertex deletion from Moore graphs}\label{sec:punctures}

Deleting vertices can increase the distances between those that remain.
We identify these increases and use them to bound the change in the least
distance eigenvalue. With these bounds we prove that the Hoffman--Singleton
graph remains a strict counterexample after the deletion of any five
vertices, while some six-vertex deletion is not strict
(Theorem~\ref{thm:hs-radius}), and that a degree-\(57\) Moore graph, if one
exists, remains one after the deletion of any \(56\) vertices
(Proposition~\ref{prop:clique-deletion-bound}). The full one- and two-deletion spectra are given in
Appendix~\ref{app:small-deletion-spectra}.

Let \(M\) be a degree-\(k\) Moore graph of diameter two.

\subsection{Changed distances and a stability bound}

\begin{theorem}[Small-puncture normal form]\label{thm:small-puncture}
Let \(M\) be a degree-\(k\) Moore graph of diameter two, let
\(S\subseteq V(M)\) have size \(s\le k-1\), and put \(H=M-S\).  Then \(H\)
is connected, has diameter at most three, and
\[
 \resultbox{\delta^*(H)=k-\frac{s}{k}.}
\]
Let \(B\) be the surviving-vertex by deleted-vertex incidence matrix, with
\(B_{xz}=1\) exactly when \(x\sim_M z\).  Then
\[
 \resultbox{
 D(H)=2(J-I)-A(H)+BB^{\mathsf T}
 -\operatorname{diag}(BB^{\mathsf T}).
 }
\]
\end{theorem}

\begin{proof}
The case \(s=0\) is immediate, so assume \(1\le s\le k-1\).  Let
\(x,y\in V(H)\) be nonadjacent in \(M\), and suppose their unique common
neighbor \(z\) lies in \(S\).  For every
\(a\in N_M(x)\setminus\{z\}\), the vertices \(a,y\) are nonadjacent and
have a unique common neighbor \(b_a\).  Thus
\[
 x-a-b_a-y
\]
is a length-three path.  The \(k-1\) paths are internally vertex-disjoint:
if \(b_a=b_{a'}\), there would be a four-cycle, and if \(b_a=a'\),
there would be a triangle through \(x\); also \(b_a\ne z\), since
\(z\sim x\sim a\).  Apart from \(z\), only
\(s-1\le k-2\) vertices have been deleted. Each can meet at most one of the
internally disjoint paths, so at least one path survives.  Thus every pair that loses its length-two path is
at distance exactly three in the surviving graph, giving the matrix formula.

For \(x\in V(H)\), let \(t_x=|N_M(x)\cap S|\).  A deleted neighbor of
\(x\) is adjacent to no surviving neighbor of \(x\), by triangle-freeness;
each other deleted vertex has at most one common neighbor with \(x\).  Thus
\[
 \sum_{y\in N_H(x)}t_y\le s-t_x.
\]
Since \(d_H(y)=k-t_y\),
\(d_H^*(x)=k-(k-t_x)^{-1}\sum_{y\in N_H(x)}t_y\).
Consequently,
\[
 d_H^*(x)
 \ge k-\frac{s-t_x}{k-t_x}
 \ge k-\frac{s}{k},
\]
where the last inequality follows from
\[
 \left(k-\frac{s-t}{k-t}\right)-\left(k-\frac{s}{k}\right)
 =\frac{t(k-s)}{k(k-t)}\ge0.
\]
The intersection bound
\[
 \left|\bigcap_{z\in S}\Gamma_2(z)\right|
 \ge k^2+1-s(k+1)\ge2
\]
provides a surviving vertex \(x\) at distance two from every deleted
vertex.  This vertex attains the lower bound.  For \(z\in S\), let
\(y_z\) be the unique common neighbor of \(x,z\).  It cannot be deleted:
\(y_z\in S\) would, by the choice of \(x\), require \(d_M(x,y_z)=2\),
contrary to \(x\sim y_z\).  All these common neighbors therefore survive. They need not be distinct:
each deleted vertex contributes once to the sum of the neighbor-degree
deficits. The lower bound is consequently attained.
\end{proof}

\begin{proposition}[Deletion stability]\label{prop:deletion-stability}
Let \(G\) be connected with at least two vertices, let
\(S\subseteq V(G)\), and suppose \(H=G-S\) is connected with at least two
vertices.  Put
\[
 D_0=D(G)[V(H)],
 \qquad E_S=D(H)-D_0,
\]
\[
 a=\delta^*(G),
 \quad b=\delta^*(H),
 \quad \gamma=\Phi(G).
\]
Then
\[
 \resultbox{
 \Phi(H)\ge\gamma-(a-b)+\lambda_{\min}(E_S).
 }
\]
\end{proposition}

\begin{proof}
The score \(\Phi(H)\) is the least eigenvalue of \(D(H)+bI\).
We write this matrix as
\[
 D(H)+bI=(D_0+aI)+E_S-(a-b)I.
\]
The first term is a principal submatrix of \(D(G)+aI\), whose least
eigenvalue is \(\gamma\).  Principal-submatrix interlacing
\cite[Corollary~2.2]{Haemers1995} therefore gives
\(\lambda_{\min}(D_0+aI)\ge\gamma\).
Weyl's inequality bounds the least eigenvalue after adding \(E_S\),
and subtracting \((a-b)I\) shifts every eigenvalue by \(-(a-b)\).
This gives the stated bound.  For Moore punctures, \(E_S\) is the
explicit distance-increase matrix specified in Theorem~\ref{thm:small-puncture}.
\end{proof}

\begin{corollary}[Uniform deletion stability]
\label{cor:uniform-deletion}
Under the hypotheses of Theorem~\ref{thm:small-puncture}, put
\[
 t_x=|N_M(x)\cap S|,
 \qquad
 \tau(S)=\max_{x\in V(M-S)}t_x.
\]
Then
\[
 \resultbox{
 \Phi(M-S)\ge
 k-\frac{3+\sqrt{4k-3}}2-\frac{s}{k}-\tau(S).
 }
\]
In particular, the deletion is a strict counterexample whenever the
right-hand side is positive.  Since \(\tau(S)\le s\), every deletion of
\(s\) vertices is strict whenever
\[
 \resultbox{
 s\left(1+\frac1k\right)
 <k-\frac{3+\sqrt{4k-3}}2.
 }
\]
For \(k>3\), equivalently, every deletion of at most
\[
 r_k=\min\left\{
 k-1,\
 \left\lceil
 \frac{k}{k+1}\left(k-\frac{3+\sqrt{4k-3}}2\right)
 \right\rceil-1
 \right\}
\]
vertices is strict.  In particular, \(r_7=2\); for a hypothetical
degree-\(57\) Moore graph, whose existence remains open
\cite{SmithMontemanni2026}, this estimate gives \(r_{57}=47\).
Proposition~\ref{prop:clique-deletion-bound} strengthens that
guarantee to every deletion of at most \(56\) vertices.
\end{corollary}

\begin{proof}
By Theorem~\ref{thm:small-puncture}, the distance-increase matrix is
\[
 E_S=BB^{\mathsf T}-\operatorname{diag}(t_x).
\]
Since \(BB^{\mathsf T}\succeq0\) and
\(\operatorname{diag}(t_x)\preceq\tau(S)I\), one has
\(\lambda_{\min}(E_S)\ge-\tau(S)\).  The parent Moore graph has score
\[
 \Phi(M)=k-\frac{3+\sqrt{4k-3}}2,
\]
and Theorem~\ref{thm:small-puncture} shows that deletion lowers the minimum
dual degree by \(s/k\).  Proposition~\ref{prop:deletion-stability} now gives
the first bound. The second follows from \(\tau(S)\le s\). The largest integer \(s\) within the allowed deletion range that lies
strictly below the endpoint of this inequality is
\(r_k\).
\end{proof}

\subsection{Deletions inside a neighborhood}

The stability bound of Corollary~\ref{cor:uniform-deletion} handles many deletion
sets without computing their full distance spectra. When the deleted
vertices share a surviving neighbor, the distance correction has a more
specific form, which gives a stronger bound. This proves stability through
three deletions for Hoffman--Singleton. Exact orbit certificates then cover
sizes four and five and determine the universal radius.

Each deleted vertex contributes a correction clique on its surviving
neighbors. If the deleted vertices have a common surviving neighbor,
these cliques meet only at that vertex. This intersection pattern lets us
compute the least eigenvalue of the correction explicitly.

\begin{proposition}[Neighborhood punctures]
\label{prop:neighborhood-puncture}
Let \(M\) be a degree-\(k\) Moore graph of diameter two, \(k\ge2\).
Let \(w\in V(M)\), let \(\varnothing\ne S\subseteq N_M(w)\), and put
\(s=|S|\le k-1\) and \(H=M-S\).  The distance-increase matrix
\[
 E_S=D(H)-D(M)[V(H)]
\]
is the adjacency matrix of \(s\) copies of \(K_k\) meeting only at \(w\),
together with \(k(k-s)\) isolated vertices.  In particular,
\[
 \lambda_{\min}(E_S)
 =\frac{k-2-\sqrt{(k-2)^2+4s(k-1)}}2,
\]
and
\[
 \Phi(H)\ge
 k-\frac{3+\sqrt{4k-3}}2-\frac{s}{k}
 +\frac{k-2-\sqrt{(k-2)^2+4s(k-1)}}2.
\]
\end{proposition}

\begin{proof}
The deleted set is independent, since a pair of its vertices joined by an
edge would form a triangle with \(w\).  Thus each deleted vertex has all
\(k\) of its neighbors in \(H\).  For distinct \(z,z'\in S\), the Moore
property gives
\[
 N_M(z)\cap N_M(z')=\{w\}.
\]
By Theorem~\ref{thm:small-puncture}, every destroyed length-two path is
replaced by a length-three path.  Hence \((E_S)_{xy}=1\), for distinct
surviving vertices \(x,y\), precisely when they belong to \(N_M(z)\) for
some \(z\in S\).  There is at most one such \(z\).  These neighborhoods
therefore give the asserted cliques, with \(k(k-s)\) remaining vertices.

We take the coordinate vector at the center and the indicator vector of
all noncentral clique vertices.  Their span is invariant under the
correction matrix.  After normalizing both vectors to have length one,
the matrix on this span is the normalized symmetric quotient
\[
 \begin{pmatrix}
 0&\sqrt{s(k-1)}\\
 \sqrt{s(k-1)}&k-2
 \end{pmatrix}.
\]
The other eigenvalues are \(-1\) on the zero-sum space of the \(k-1\)
noncentral vertices of each clique, \(k-2\) on the differences of their
constant sums, and \(0\) on the isolated vertices.  These spaces have
dimensions \(s(k-2)\), \(s-1\), and \(k(k-s)\), respectively, so the
decomposition is complete.  The smaller quotient root is at most \(-1\), because its
characteristic polynomial has value \((k-1)(1-s)\le0\) at \(-1\).
It is therefore the least eigenvalue of \(E_S\).
The final inequality follows from Proposition~\ref{prop:deletion-stability}
and \(\delta^*(H)=k-s/k\).
\end{proof}

\subsection{The Hoffman--Singleton radius}
\label{sec:robustness}

\begin{corollary}[Analytic three-vertex stability]
\label{cor:analytic-three-puncture}
Let \(M\) be the Hoffman--Singleton graph.  For every
\(S\subseteq V(M)\) with \(|S|\le3\),
\[
 \Phi(M-S)>\frac17.
\]
In particular, every such puncture is a strict counterexample to WOW-284.
\end{corollary}

\begin{proof}
The graph \(M\) has score \(3\).  If \(|S|\le2\),
Corollary~\ref{cor:uniform-deletion} gives
\(\Phi(M-S)\ge3-|S|/7-|S|\ge5/7\).
Now let \(|S|=3\), and put
\[
 \tau=\max_{x\in V(M-S)}|N_M(x)\cap S|.
\]
If \(\tau\le2\), the same corollary gives
\(\Phi(M-S)\ge3-3/7-2=4/7\).
If \(\tau=3\), some surviving vertex \(w\) is adjacent to all three
deleted vertices.  Proposition~\ref{prop:neighborhood-puncture} gives
\[
 \Phi(M-S)\ge\frac{71-7\sqrt{97}}{14}>\frac17,
\]
where the strict comparison follows from
\(49\cdot97=4753<4761=69^2\).
\end{proof}

\begin{theorem}[Hoffman--Singleton robustness radius]\label{thm:hs-radius}
Let \(M\) be the Hoffman--Singleton graph.  Every induced graph \(M-S\) with
\(|S|\le5\) is a strict counterexample to WOW-284.  This is sharp in the
universal sense: there exists a six-vertex set whose deletion is not strict.
Hence the universal vertex-deletion robustness radius is exactly five.
\end{theorem}

\begin{proof}
Theorem~\ref{thm:small-puncture} gives
\[
 \delta^*(M-S)=\frac{49-|S|}{7}.
\]
Corollary~\ref{cor:analytic-three-puncture} proves strictness analytically
for every set with \(|S|\le3\).  Only sets of sizes four and five remain
to be certified. We reduce the finite check using automorphisms, which are
vertex permutations preserving adjacency. Deletion sets related by these
permutations give isomorphic surviving graphs. Two explicitly stored
permutations are verified edge by edge on the coordinate graph, and they
generate a group of order \(252000\).  The full orbit
inventory on deletion sets of sizes \(0,1,2,3,4,5\) has counts
\[
 1,1,2,4,11,33.
\]
Thus the remaining sizes require \(44\) representatives.  For each one,
exact fraction arithmetic gives
\[
 7D(M-S)+(49-|S|)I\succ0.
\]

To show that the radius cannot be increased, delete
\[
 \{P_{2,4},P_{3,1},P_{3,4},Q_{2,1},Q_{3,4},Q_{4,4}\}.
\]
The resulting graph has \(\delta^*=43/7\), while an exact
\(LDL^{\mathsf T}\) decomposition of \(7D(M-S)+43I\) has exactly one negative
and no zero pivot.  Since \(L\) is invertible, Sylvester's law of inertia makes
\(7D(M-S)+43I\) indefinite, so the graph is not strict.
\end{proof}

Some six-vertex deletions remain strict: by
Proposition~\ref{prop:hs-neighborhood-six} in
Appendix~\ref{app:neighborhood-deletions}, deleting any six neighbors of a
vertex leaves a strict counterexample. Such deletions do not increase the
universal radius, since the deletion in Theorem~\ref{thm:hs-radius} fails.

\subsection{Arbitrary deletions from a degree-\texorpdfstring{\(57\)}{57} Moore graph}

For a hypothetical degree-\(57\) Moore graph,
Corollary~\ref{cor:degree57-neighborhood} in
Appendix~\ref{app:neighborhood-deletions} uses
Proposition~\ref{prop:neighborhood-puncture} to show \(\Phi(M-S)>685/57\)
whenever \(S\) lies in the neighborhood of a vertex and \(|S|\le56\).
For an arbitrary deletion set, we group the correction cliques that pass
through one surviving vertex. The remaining cliques can be bounded
separately, giving the following improvement.

\begin{proposition}[Splitting the distance correction at a vertex]
\label{prop:clique-deletion-bound}
Let \(M\) be a degree-\(k\) Moore graph of diameter two, \(k\ge2\),
and let \(S\subseteq V(M)\) have size \(s\le k-1\).  Put \(H=M-S\),
\[
 E_S=D(H)-D(M)[V(H)],
 \qquad
 \tau=\max_{x\in V(H)}|N_M(x)\cap S|,
\]
and, for \(0\le r\le k-1\), define
\[
 f_k(r)=\frac{k-2-\sqrt{(k-2)^2+4r(k-1)}}2.
\]
Then
\[
 \lambda_{\min}(E_S)
 \ge \max\{-\tau,\,f_k(\tau)-s+\tau\},
\]
and
\[
 \Phi(H)\ge k-\frac{3+\sqrt{4k-3}}2-\frac{s}{k}
 +\max\{-\tau,\,f_k(\tau)-s+\tau\}.
\]
In particular, if \(M\) has degree \(57\), then every
\(S\subseteq V(M)\) with \(|S|\le56\) satisfies
\[
 \Phi(M-S)\ge\frac{286}{57}>0.
\]
Thus every such puncture is a strict counterexample to WOW-284.
\end{proposition}

\begin{proof}
For \(s=0\), one has \(\tau=0\), \(f_k(0)=0\), and \(E_S=0\), so the
claims follow from Theorem~\ref{thm:moore-threshold}.  Assume \(s\ge1\).
For \(z\in S\), let \(b_z\) be the indicator vector of
\(N_M(z)\cap V(H)\), and put
\[
 C_z=b_zb_z^{\mathsf T}-\operatorname{diag}(b_z).
\]
The matrix \(C_z\) records the distance increases caused by deleting
\(z\): it is the adjacency matrix of a clique on the surviving neighbors
of \(z\), with all other vertices isolated.
Theorem~\ref{thm:small-puncture} gives \(E_S=\sum_{z\in S}C_z\).
Each clique matrix has least eigenvalue at least \(-1\), since
\[
 C_z+I=b_zb_z^{\mathsf T}
       +\operatorname{diag}(\mathbf1-b_z)\succeq0.
\]
The Gram representation also gives the uniform bound:
\(E_S\succeq-\tau I\).

To obtain the second bound, choose \(w\in V(H)\) attaining \(\tau\),
and put \(Z=N_M(w)\cap S\).  Every deleted vertex has a surviving
neighbor, since \(s<k\), so \(1\le|Z|=\tau\le k-1\).
We group the \(\tau\) clique matrices indexed by \(Z\) and estimate
the remaining \(s-\tau\) separately.
The correction matrix for deleting \(Z\) alone is
\[
 F_Z=D(M-Z)-D(M)[V(M-Z)].
\]
By the small-puncture formula, its principal submatrix on \(V(H)\)
is exactly \(\sum_{z\in Z}C_z\).
Proposition~\ref{prop:neighborhood-puncture} gives
\(\lambda_{\min}(F_Z)=f_k(\tau)\), so interlacing bounds the least
eigenvalue of this grouped sum below by \(f_k(\tau)\).
Each of the remaining clique matrices has least eigenvalue at least
\(-1\).  Weyl's inequality therefore yields
\[
 \lambda_{\min}(E_S)
 \ge f_k(\tau)-(s-\tau).
\]
Together with \(E_S\succeq-\tau I\), this proves the first inequality.
The score bound follows from Proposition~\ref{prop:deletion-stability}
and \(\delta^*(H)=k-s/k\).

For \(k=57\), the original Moore graph has score \(48\).  If \(\tau\le42\), the first
bound gives \(\lambda_{\min}(E_S)\ge-42\).  If \(\tau\ge43\), then
\[
 (27+2\tau)^2-(55^2+224\tau)
 =4\bigl((\tau-43)(\tau+14)+28\bigr)>0.
\]
Since \(27+2\tau>0\), this implies
\(f_{57}(\tau)>14-\tau\), and hence
\[
 \lambda_{\min}(E_S)>14-s\ge-42.
\]
Consequently, for every \(s\le56\),
\[
 \Phi(M-S)\ge48-\frac{s}{57}-42
 \ge48-\frac{56}{57}-42=\frac{286}{57}>0.
\]
\end{proof}

\subsection{One and two vertex deletions}

For one deleted vertex, and for the endpoints of an edge when \(k\ge3\),
the least distance eigenvalue is \(-2-\sqrt{k}\). Combining this with the
dual-degree formula gives, respectively,
\[
 \Phi(M-v)=k-\frac1k-2-\sqrt{k},
 \qquad
 \Phi(M-\{u,v\})=k-\frac2k-2-\sqrt{k}\quad(uv\in E(M)).
\]
Both scores are positive for every integer \(k\ge5\).
Theorems~\ref{thm:one-puncture} and~\ref{thm:edge-puncture} give the
complete spectra in Appendix~\ref{app:small-deletion-spectra}.
For two nonadjacent deletions, Theorem~\ref{thm:nonadjacent-puncture}
also gives a full factorization and proves strictness for every realizable
integer \(k\ge6\). The stability bound above avoids these diagonalizations
when only strictness is needed.

\section{Small orders and induced subgraphs of the Hoffman--Singleton graph}
\label{sec:small-order-bounds}

This section treats graphs that need not be regular. Every strict
counterexample has at least \(15\) vertices
(Theorem~\ref{thm:fifteen-vertices}), an exhaustive search excludes orders up
to \(19\) (Proposition~\ref{prop:nineteen-vertices}), and the strict induced
subgraphs of the Hoffman--Singleton graph of order at least \(36\) are
classified up to automorphism (Proposition~\ref{prop:hs-classification});
among them are three further irregular counterexamples of order \(40\)
(Proposition~\ref{prop:order40-irregular}).

Here and in Appendix~\ref{app:small-orders}, each statement
carries one of four labels: \emph{proved} (a written
proof is given in the text), \emph{computer-assisted} (an exhaustive computation
reproduced by the code described in Appendix~\ref{sec:verification}),
\emph{Lean-checked} (with the Lean names and the scope stated in
Appendix~\ref{sec:verification}), or \emph{conjectured}. In the
computer-assisted results, floating-point computations only propose
candidates and integer test vectors; every decision rests on an exact
certificate. For a connected graph with \(\delta^*(G)=p/q\), an integer
vector \(x\) with \(x^{\mathsf T}(qD+pI)x<0\) certifies \(\Phi(G)<0\), and
every graph without such a vector is decided by exact fraction-free
elimination of \(qD+pI\). None of these
results decides whether \(38\) is the minimum order of a strict
counterexample, and no infinite family of strict counterexamples is claimed.

Throughout, \(M\) is the coordinate Hoffman--Singleton graph of
Proposition~\ref{prop:hs-coordinate}, and \(R\), \(H_{38}\), \(H_{39}\) and
\(X_{42}\) are the graphs defined before Theorem~\ref{thm:explicit-examples}.

\subsection{A lower bound of fifteen vertices}

We use the radius-two count of Proposition~\ref{prop:radius-two},
\(|B_2(v)|=1+d(v)\,d^*(v)\), together with the following path constants,
obtained from the exact characteristic polynomials
\begin{align*}
 \chi_{D(P_4)}&=(x^2-4x-6)(x^2+4x+2),&
 \lambda_{\min}(D(P_4))&=-2-\sqrt2,\\
 \chi_{D(P_6)}&=(x+1)(x^2+8x+4)(x^3-9x^2-36x-20),&
 \lambda_{\min}(D(P_6))&=-4-2\sqrt3,
\end{align*}
and \(\lambda_{\min}(D(P_3))=-2\). A shortest path is an isometric subgraph,
so its distance matrix is a principal submatrix of \(D\); Cauchy interlacing
then bounds \(\lambda_{\min}(D)\) from above.

\begin{lemma}[proved]\label{lem:small-diam-two}
Let \(G\) be connected, of order at least three and girth at least five. If
\(\diam(G)\le2\) and \(\Phi(G)>0\), then \(G\) is a Moore graph of degree
\(7\) or \(57\); in particular \(|V(G)|\ge50\).
\end{lemma}

\begin{proof}
Diameter one would make \(G\) complete, hence not triangle-free, so
\(\diam(G)=2\). If \(G\) is a tree, it is a star \(K_{1,m}\) with \(m\ge2\):
its center has \(d^*=1\), and a leaf, the center and another leaf form a
shortest path, so \(\Phi(G)\le1-2<0\).

Suppose that \(G\) contains a cycle. Let \(u,w\) be nonadjacent. They have
a unique common neighbor \(c\). Each \(x\in N(u)\setminus\{c\}\) is not
adjacent to \(w\), since otherwise \(u,x,w,c\) is a four-cycle, so \(x\) and
\(w\) have a unique common neighbor \(f(x)\in N(w)\). We have \(f(x)\ne c\),
since otherwise \(u,x,c\) is a triangle, and \(f\) is injective, since
\(f(x)=f(x')\) with \(x\ne x'\) gives the four-cycle \(u,x,f(x),x'\). Hence
\(d(u)\le d(w)\), and by symmetry nonadjacent vertices have equal degree.
If \(w\) were a leaf with neighbor \(u\), every vertex would be adjacent to
\(u\) and \(G\) would be a star; so every degree is at least two. Let
\(u\sim w\), and choose \(x\in N(u)\setminus\{w\}\) and
\(y\in N(w)\setminus\{u\}\); then \(x\ne y\) and \(x\not\sim y\). If every
vertex were adjacent to \(u\) or to \(w\), a common neighbor \(z\) of \(x\)
and \(y\) would satisfy \(z\notin\{u,w\}\) and would be adjacent to \(u\) or
\(w\), closing a triangle. So some vertex \(z\) is adjacent to neither, and
\(d(u)=d(z)=d(w)\). Thus \(G\) is \(k\)-regular, every ball \(B_2(v)\) is all
of \(V(G)\), and Proposition~\ref{prop:radius-two} gives \(|V(G)|=k^2+1\): \(G\)
is a Moore graph of diameter two. Theorem~\ref{thm:moore-threshold} and the
Hoffman--Singleton theorem \cite{HoffmanSingleton1960} then force
\(k\in\{7,57\}\).
\end{proof}

\begin{theorem}[proved]\label{thm:fifteen-vertices}
Every strict counterexample to WOW-284 has at least \(15\) vertices. If
\(G\) is a strict counterexample of order \(n<50\), then \(\diam(G)\ge3\),
\[
 \delta^*(G)>2+\sqrt2,
 \qquad
 \Delta(G)<\frac{n-1}{2+\sqrt2},
\]
where \(\Delta(G)\) is the maximum degree.
\end{theorem}

\begin{proof}
By Lemma~\ref{lem:small-diam-two}, a strict counterexample of diameter at most
two has at least \(50\) vertices. Otherwise \(\diam(G)\ge3\), and the first
four vertices of a diametral shortest path span an isometric \(P_4\), so
\(\lambda_{\min}(D)\le-2-\sqrt2\) and
\(\delta^*(G)>-\lambda_{\min}(D)\ge2+\sqrt2\). Let \(w\) have degree
\(\Delta(G)\). Since \(d^*(w)\ge\delta^*(G)\),
\[
 n\ge|B_2(w)|=1+\Delta(G)\,d^*(w)>1+\Delta(G)(2+\sqrt2),
\]
which is the degree bound. Since \(\Delta(G)\ge d^*(v)\ge\delta^*(G)>3\) for
every \(v\), we have \(\Delta(G)\ge4\) and \(n>9+4\sqrt2>14.6\).
\end{proof}

For orders at most \(37\), the same count gives
\(\delta^*(G)\le\min\{\Delta(G),36/\Delta(G)\}\le6\), with equality only for a
\(6\)-regular graph on \(37\) vertices, that is, a Moore graph of degree six,
which does not exist \cite{HoffmanSingleton1960}. Hence \(\delta^*(G)<6\), an
isometric \(P_6\) is excluded because \(-4-2\sqrt3<-6\), and a strict
counterexample of order at most \(37\) has diameter three or four.

\subsection{Exhaustive search on at most nineteen vertices}

\begin{proposition}[computer-assisted]\label{prop:nineteen-vertices}
No connected graph of girth at least five on at most \(19\) vertices is a
strict counterexample.
\end{proposition}

The graphs were generated, one per isomorphism class, by \texttt{geng} from
nauty \cite{McKayPiperno2014}. For \(15\le n\le19\), the maximum degree was
capped by Theorem~\ref{thm:fifteen-vertices}, and only graphs with
\(\delta^*>2+\sqrt2\) were kept; every remaining graph was decided exactly as
described above. Appendix~\ref{app:small-orders} gives the options and the
counts.

\Needspace{14\baselineskip}
\subsection{Induced subgraphs of the Hoffman--Singleton graph}

\begin{proposition}[computer-assisted]\label{prop:hs-classification}
For \(0\le s\le14\), the \(\operatorname{Aut}(M)\)-orbits of \(s\)-subsets
\(S\subseteq V(M)\) with \(M-S\) connected, and those with \(\Phi(M-S)>0\),
number as follows (\(|S|=50-|V(M-S)|\)).
\[
\begin{array}{rrr@{\qquad}rrr}
\toprule
|V(M-S)|&\text{connected}&\text{strict}&|V(M-S)|&\text{connected}&\text{strict}\\
\midrule
50&1&1&42&2{,}521&108\\
49&1&1&41&10{,}799&22\\
48&2&2&40&42{,}575&4\\
47&4&4&39&151{,}802&1\\
46&11&11&38&488{,}298&1\\
45&33&33&37&1{,}419{,}003&0\\
44&128&110&36&3{,}736{,}168&0\\
43&561&185&&&\\
\bottomrule
\end{array}
\]
Every other connected orbit has \(\Phi<0\), except one orbit of order \(43\)
with \(\Phi=0\) (Proposition~\ref{prop:order43-equality}). In particular,
\(M\) has no strict induced subgraph of order \(36\) or \(37\), and the strict
induced subgraphs of orders \(38\) and \(39\) are, up to automorphisms of
\(M\), exactly \(H_{38}\) and \(H_{39}\).
\end{proposition}

The orbits were enumerated by orderly generation over the \(252{,}000\)
automorphisms of \(M\), and each connected representative was decided in the
same way. Appendix~\ref{app:small-orders} describes the enumeration and the
certificate for the equality case.

\begin{proposition}[computer-assisted, Lean-checked]\label{prop:order40-irregular}
The graphs
\begin{align*}
 X_{40}&=M-\{P_{0,j},Q_{k,0}:j,k\in\F\}=X_{42}-\{P_{0,2},P_{0,3}\},\\
 Y_{40}&=M-\{P_{0,j}:j\in\F\}-\{P_{1,0},Q_{0,0},Q_{0,2},Q_{2,2},Q_{4,4}\},\\
 Z_{40}&=M-\{P_{0,j}:j\in\F\}-\{P_{1,0},Q_{0,0},Q_{0,2},Q_{1,1},Q_{1,4}\}
\end{align*}
are strict counterexamples of order \(40\), of girth five and diameter three:
\[
\begin{array}{lccc}
\toprule
G&\text{edges}&\text{degree counts}&\text{certificate}\\
\midrule
X_{40}&115&5^{10}\,6^{30}&\delta^*=17/3,\ 3D+17I\succ0\\
Y_{40}&118&5^{9}\,6^{26}\,7^{5}&\delta^*=39/7,\ 7D+39I\succ0\\
Z_{40}&118&5^{10}\,6^{24}\,7^{6}&\delta^*=39/7,\ 7D+39I\succ0\\
\bottomrule
\end{array}
\]
Together with \(R\), they represent the four strict orbits of order \(40\) in
Proposition~\ref{prop:hs-classification}.
\end{proposition}

The four graphs are pairwise nonisomorphic: \(R\) is regular, \(X_{40}\) has
\(115\) edges, and \(Y_{40}\), \(Z_{40}\) have different degree counts. The
positive definiteness was checked twice, by fraction-free elimination and by
an independent rational \(LDL^{\mathsf T}\) decomposition with all pivots
positive. Lean checks the same data and the displayed positive-definite
shifts (Appendix~\ref{sec:verification}).

\section{Open questions on spectral constraints and stability}\label{sec:scope}

We list the questions that the results above leave open.
The necessary conditions at degree six and order $50$ do not decide whether
a strict counterexample with these parameters exists. The minimum order of a
strict counterexample is also unknown; by the computations of
Section~\ref{sec:small-order-bounds} it is at least \(20\), and no induced subgraph
of the Hoffman--Singleton graph of order \(36\) or \(37\) is strict. A further direction is to extend the
order bounds for regular graphs to irregular graphs.

Appendix~\ref{app:excluded-constructions} defines a balanced prime-field
family of girth-five graphs and rules out strict counterexamples in that
family when the diameter is three. It also rules out every layer-respecting
perfect-matching deletion of Hoffman--Singleton considered there.  These calculations test the
window criterion but are not needed for the counterexamples or the deletion
bounds above.

\paragraph{Disclosure and reproducibility.}
OpenAI ChatGPT-5.6 Sol Pro assisted with adversarial proof checking, proof
exploration, and Lean formalization. Codex also assisted with supplementary
proof arguments.  The author assumes full responsibility
for the mathematics, attribution, and conclusions.  The source, exact
certificates, and build instructions are available at
\href{https://github.com/SamPetkov/wow284}{\texttt{github.com/SamPetkov/wow284}};
Appendix~\ref{sec:verification} describes the code and data.

\begingroup
\footnotesize
\setlength{\bibsep}{0pt}
\renewcommand{\bibsection}{\section*{References}}
\bibliographystyle{unsrtnat}
\bibliography{references}
\endgroup

\appendix

\section{Distance spectra of Moore graphs after vertex deletion}
\label{app:small-deletion-spectra}

This appendix proves the exact spectra quoted in Section~\ref{sec:punctures}
and, using the one-vertex decomposition, the second-subconstituent spectrum
of Theorem~\ref{thm:second-subconstituent}.
Here \(M\) is any degree-\(k\) Moore graph of diameter two, not only the
coordinate graph of Proposition~\ref{prop:hs-coordinate}, and
\(\Delta=\sqrt{4k-3}\).

Every distance matrix below uses shortest-path distances in the surviving
graph. The correction to the original distances is given by
Theorem~\ref{thm:small-puncture}:
a surviving nonadjacent pair moves from distance two to distance three
exactly when its unique common neighbor was deleted.

\subsection{One deleted vertex and the endpoints of an edge}

\begin{theorem}[One deleted vertex]\label{thm:one-puncture}
For \(v\in V(M)\), put \(H=M-v\).  Then
\[
 |V(H)|=k^2,
 \qquad \delta^*(H)=k-\frac1k,
 \qquad \lambda_{\min}(D(H))=-2-\sqrt{k}.
\]
The complete distance spectrum is
\begin{align*}
\Spec D(H)={}&\{\rho_+,\rho_-\}
 \cup\{(-2+\sqrt{k})^{(k-1)},(-2-\sqrt{k})^{(k-1)}\}\\
&\cup\left\{\left(-\frac{\Delta+3}{2}\right)^{(m_+)},
\left(\frac{\Delta-3}{2}\right)^{(m_-)}\right\},
\end{align*}
where
\[
 \rho_\pm=k^2-2\pm\sqrt{k(k^3-2k^2+3k-1)},
\]
\[
 m_\pm=\frac{k(k-2)(\Delta\pm1)}{2\Delta}.
\]
Thus
\[
 \Phi(H)=k-\frac1k-2-\sqrt{k},
\]
which is positive exactly for integers \(k\ge5\).
\end{theorem}

\begin{theorem}[Endpoints of an edge]\label{thm:edge-puncture}
Let \(uv\in E(M)\) and \(H=M-\{u,v\}\).  For \(k\ge3\),
\[
 |V(H)|=k^2-1,
 \qquad \delta^*(H)=k-\frac2k,
 \qquad \lambda_{\min}(D(H))=-2-\sqrt{k}.
\]
The complete distance spectrum is
\begin{align*}
\Spec D(H)={}&\{\sigma_+,\sigma_-,k-4\}\\
&\cup\{(-2+\sqrt{k})^{(2k-4)},(-2-\sqrt{k})^{(2k-4)}\}\\
&\cup\left\{\left(-\frac{\Delta+3}{2}\right)^{(a_+)},
\left(\frac{\Delta-3}{2}\right)^{(a_-)}\right\},
\end{align*}
where
\[
 \sigma_\pm=k^2-3\pm\sqrt{k^4-2k^3+3k^2-8k+7},
\]
\[
 a_+=\frac{(k-2)(k+(k-2)\Delta)}{2\Delta},
 \quad
 a_-=\frac{(k-2)((k-2)\Delta-k)}{2\Delta}.
\]
Thus
\[
 \Phi(H)=k-\frac2k-2-\sqrt{k},
\]
which is positive exactly for integers \(k\ge5\).
\end{theorem}

\begin{proof}[Proof of Theorems~\ref{thm:one-puncture}
and~\ref{thm:edge-puncture}]
The dual-degree formulas are the cases \(s=1,2\) of
Theorem~\ref{thm:small-puncture}. In each case we split
\(\mathbb R^{V(H)}\) into the cell-constant space, invariant planes pairing a
zero-sum vector on a neighbor cell with its vector of neighbor sums, and a
residual space on which \(D(H)\) is \(-2I\) minus an adjacency block
satisfying the Moore relation.

For one deleted vertex, put \(A=N(v)\),
\(B=\Gamma_2(v)\), let \(C\) be
the \(A\)-by-\(B\) incidence matrix, and let \(B_0\) be the adjacency
matrix on \(B\).  Only pairs inside \(A\) lose their unique length-two path.
By Theorem~\ref{thm:small-puncture}, their new distance is three.
All other distances are unchanged, so
\[
 D(H)=\begin{pmatrix}3(J-I)&2J-C\\2J-C^{\mathsf T}&2(J-I)-B_0\end{pmatrix}.
\]
\Needspace{6\baselineskip}
The normalized constant quotient is
\[
 Q_1=\begin{pmatrix}
 3(k-1)&(2k-1)\sqrt{k-1}\\
 (2k-1)\sqrt{k-1}&2k^2-3k-1
 \end{pmatrix},
\]
whose eigenvalues are \(\rho_\pm\).  Block comparison in the Moore identity
gives
\[
 CC^{\mathsf T}=(k-1)I,\qquad
 CB_0=J-C,\qquad
 C^{\mathsf T}\one_A=\one_B,
\]
\[
 B_0^2+C^{\mathsf T}C=(k-1)I-B_0+J.
\]
Thus \(C^{\mathsf T}\) is injective and
\[
 \mathbb R^B=
 \langle\one_B\rangle\perp
 C^{\mathsf T}(\one_A^\perp)\perp K,
 \qquad K=\ker C,\qquad \dim K=k(k-2).
\]
Choose an orthogonal basis \(\mathcal B\) of \(\one_A^\perp\), and for
\(z\in\mathcal B\) put
\[
 W_z=\operatorname{span}\{(z,0),(0,C^{\mathsf T}z)\}.
\]
The quotient acts on the cell-constant space \(W_{\mathrm{const}}\).
The remaining directions split into the two-dimensional spaces \(W_z\)
for \(z\in\mathcal B\) and the residual space \(0\oplus K\).
The identities above show that these spaces are invariant and mutually
orthogonal.  Put
\[
 W_{\mathrm{inc}}=\bigoplus_{z\in\mathcal B}W_z.
\]
Consequently
\[
\mathbb R^{V(H)}=W_{\mathrm{const}}\perp W_{\mathrm{inc}}\perp(0\oplus K),
 \qquad 2+2(k-1)+k(k-2)=k^2.
\]
For \(z\in\mathcal B\), \(D(H)(z,0)=-(3z,C^{\mathsf T}z)\) and, since
\(B_0C^{\mathsf T}z=(CB_0)^{\mathsf T}z=-C^{\mathsf T}z\),
\(D(H)(0,C^{\mathsf T}z)=-((k-1)z,C^{\mathsf T}z)\).
On each zero-sum incidence space, the matrix of the distance operator in the chosen generators is
\(\bigl(\begin{smallmatrix}-3&-(k-1)\\-1&-1\end{smallmatrix}\bigr)\),
with eigenvalues \(-2\pm\sqrt{k}\).  The restriction to \(K\) satisfies
\(B_0^2+B_0-(k-1)I=0\) and \(D=-2I-B_0\), so the remaining distance eigenvalues
are obtained by negating the two Moore adjacency roots and subtracting two.  Their multiplicities are determined by the trace of \(B_0\):
after removing its principal eigenvalue \(k-1\) and its \(k-1\) copies
of \(-1\) on \(C^{\mathsf T}(\one_A^\perp)\), the remaining trace is
zero.  This trace and the residual dimension give \(m_\pm\).

For an adjacent deleted pair, put
\[
 A=N(u)\setminus\{v\},\qquad B=N(v)\setminus\{u\},\qquad
 C=V(M)\setminus\bigl(\{u,v\}\cup A\cup B\bigr).
\]
Only pairs inside \(A\) or inside \(B\) lose a length-two path.  Since \(k\ge3\), Theorem~\ref{thm:small-puncture} applies to the
two deleted vertices and makes the new distance of each such pair three.
Vectors with opposite constant values on the two neighbor cells and zero
on the third cell form an invariant line with eigenvalue
\(k-4\). On the remaining cell-constant directions, the two neighbor-cell
values agree, and the normalized symmetric quotient is
\[
 Q_2=\begin{pmatrix}
 5k-8&\sqrt{(k-1)(2k-3)(4k-6)}\\
 \sqrt{(k-1)(2k-3)(4k-6)}&2k^2-5k+2
 \end{pmatrix},
\]
with eigenvalues \(\sigma_\pm\).  Let \(R_A,R_B\) be the incidence matrices
from \(A,B\) to \(C\), and let \(T\) be the adjacency matrix on \(C\).  The
Moore identity gives
\[
 R_AR_A^{\mathsf T}=R_BR_B^{\mathsf T}=(k-1)I,\qquad
 R_AR_B^{\mathsf T}=J,
\]
\[
 R_AT=J-R_A,\qquad R_BT=J-R_B,
\]
\[
 R_A^{\mathsf T}R_A+R_B^{\mathsf T}R_B+T^2
 =(k-1)I-T+J.
\]
The transpose maps of \(R_A\) and \(R_B\), restricted to their cell
zero-sum spaces, are injective and have orthogonal images.  Pair each zero-sum vector in its original cell with the vector obtained
from its neighbor sums. Each pair spans an invariant two-dimensional
space carrying
\(\bigl(\begin{smallmatrix}-3&-(k-1)\\-1&-1\end{smallmatrix}\bigr)\),
as in the one-vertex calculation.  The residual space
\[
 K=\ker R_A\cap\ker R_B,\qquad \dim K=(k-2)^2,
\]
is invariant and satisfies
\[
 T^2+T-(k-1)I=0,\qquad D=-2I-T.
\]
These spaces and the three-dimensional cell-constant space are mutually
orthogonal and complete, since
\[
 3+2(2k-4)+(k-2)^2=k^2-1.
\]
The constant direction of \(T\) has eigenvalue \(k-2\), while its two
zero-sum incidence images contribute trace \(-2(k-2)\).  Since
\(\tr T=0\), one has \(\tr(T|_K)=k-2\), and dimension and trace give
\(a_\pm\).  The residual negative root is greater than
\(-2-\sqrt{k}\).  For \(k\ge3\), the antisymmetric constant root also satisfies
\[
 (k-4)-(-2-\sqrt{k})=k-2+\sqrt{k}>0.
\]
It remains to compare the incidence root with the smaller roots of the
constant quotients. We shift each quotient by subtracting the candidate
root from its diagonal. If the shifted matrices are positive definite,
then every quotient eigenvalue lies strictly above that root.  In each case the leading
diagonal entry is positive, and
\[
 \det(Q_1+(2+\sqrt{k})I)
 =k(2k^2+2k^{3/2}-3k+2)>0,
\]
\[
 \det(Q_2+(2+\sqrt{k})I)
 =(\sqrt{k}-1)(\sqrt{k}+1)
 (2k^2+2k^{3/2}-3k+2\sqrt{k}+6)>0.
\]
The shifted symmetric quotients are therefore positive definite.  The
constant-quotient eigenvalues lie strictly above the incidence root,
completing the comparison.
Exact quotient normalization, trace-to-multiplicity equations, replacement
paths, and all least-root comparisons are independently checked by
exact computation.
\end{proof}

\subsection{The distance spectrum of a Moore second sphere}
\label{app:second-sphere-proof}

To prove Theorem~\ref{thm:second-subconstituent}, we use the decomposition
of the second-sphere adjacency matrix from the proof of
Theorem~\ref{thm:one-puncture}. It gives the eigenvalues and their
multiplicities, after which we recover the distance spectrum.

\begin{proof}[Proof of Theorem~\ref{thm:second-subconstituent}]
Each vertex of the second sphere loses exactly one neighbor when we
pass to its induced graph.  Put \(m=K(K-1)\); the Moore property gives
\(\lvert\Gamma_2(v)\rvert=m\).  Each vertex of \(\Gamma_2(v)\) has its
single parent in \(N(v)\), while each vertex of \(N(v)\) has \(K-1\)
neighbors in \(\Gamma_2(v)\).  The induced graph \(X\) consequently has
order \(m\) and degree \((K-1)\).  It contains no triangle or four-cycle,
since \(M\) contains neither.  To obtain the distance spectrum we now
calculate the adjacency spectrum and then prove that the diameter is three.

The graph \(X\) is the graph on \(\Gamma_2(v)\) in the proof of
Theorem~\ref{thm:one-puncture} with \(k=K\); its adjacency matrix is \(B_0\)
there. That proof shows that \(B_0\) has eigenvalue \(K-1\) on the constant
vector, \(-1\) with multiplicity \(K-1\) on
\(C^{\mathsf T}(\one_A^\perp)\), and on \(\ker C\), of dimension
\(K(K-2)\), only \(r,s=(-1\pm\sqrt{4K-3})/2\), with the multiplicities
\(m_+,m_->0\) of Theorem~\ref{thm:one-puncture}; both roots occur.
It remains to check that \(X\) has diameter three.
If the unique common neighbor in \(M\) of nonadjacent vertices \(x,y\in X\)
lies in \(X\), their distance in \(X\) is two.  The only other
possibility is their common parent in \(N(v)\).  Choose
\(b\in N_X(x)\), and let \(p\) be the common parent of \(x,y\).
The parent of \(b\) is not \(p\), since otherwise \(p,x,b\) would form
a triangle.
Moreover \(b\not\sim y\), and the unique common neighbor
\(c\) of \(b,y\) belongs to \(X\): it cannot be \(v\), and it cannot
lie in \(N(v)\), since \(b\) and \(y\) have different parents there.
Thus \(x-b-c-y\) is a path in \(X\).  Each parent has \(K-1\ge2\)
children, and pairs of distinct children have no length-two path in \(X\).
Thus \(X\) has diameter three and
Theorem~\ref{thm:diameter-three-score} applies.  Among the nonprincipal
adjacency eigenvalues, \((-1+\sqrt{4K-3})/2\) maximizes
\(|\theta+1|\); substitution gives
\[
 \lambda_{\min}(D(X))=-\frac{5+\sqrt{4K-3}}2.
\]
For \(K=3,4,5\), direct comparison gives a nonpositive score.  For
\(K\ge6\), both sides of the relevant comparison are positive, and the
threshold reduces to
\[
 (2K-7)^2-(4K-3)=4(K^2-8K+13)>0
\]
because \(K>4+\sqrt3\).  The finite \(K=7\) instance is also checked by exact computation.
\end{proof}

\subsection{Two nonadjacent deleted vertices}

\begin{theorem}[Two nonadjacent deleted vertices]\label{thm:nonadjacent-puncture}
Let \(u,v\) be nonadjacent vertices of \(M\), \(k\ge5\), and put
\(H=M-\{u,v\}\).  Define
\begin{align*}
R_k(x)={}&x^4+(10-2k^2)x^3+(2k^3-17k^2-2k+36)x^2\\
&+(12k^3-49k^2-4k+53)x\\
&-2k^4+17k^3-38k^2+5k+20,
\end{align*}
\[
 M_- =\frac{k(k-2)+(k^2-4k+2)\Delta}{2\Delta},
 \quad
 M_+ =\frac{-k(k-2)+(k^2-4k+2)\Delta}{2\Delta}.
\]
Then
\begin{align*}
\chi_{D(H)}(x)={}&(x-k+3)R_k(x)
 (x^2+4x-k+3)^{k-2}\\
&\cdot(x^2+4x-k+5)^{k-2}
 \left(x+\frac{\Delta+3}{2}\right)^{M_-}\\
&\cdot\left(x-\frac{\Delta-3}{2}\right)^{M_+}.
\end{align*}
Moreover,
\[
 \delta^*(H)=k-\frac2k,
\]
and \(H\) is a strict counterexample for every realizable integer \(k\ge6\).
\end{theorem}

\begin{proof}
The deleted vertices have a unique common neighbor \(w\).  With
\(A=N(u)\setminus\{w\}\), \(B=N(v)\setminus\{w\}\),
\(C=N(w)\setminus\{u,v\}\), and
\[
 Z=V(M)\setminus\bigl(\{u,v,w\}\cup A\cup B\cup C\bigr),
\]
the surviving graph has cell sizes, in the order \(\{w\},A,B,C,Z\),
\[
 1,\quad k-1,\quad k-1,\quad k-2,\quad (k-1)(k-2).
\]
The Moore common-neighbor rule gives a perfect matching between
\(A\) and \(B\), and the five-cell partition is equitable.
The distance rule from Theorem~\ref{thm:small-puncture}, stated at the
start of this section, applies because \(2\le k-1\).
We now restrict the distance matrix to vectors constant on each cell.
The row-sum quotient is
\[
\begin{pmatrix}
0&3k-3&3k-3&k-2&2(k-1)(k-2)\\
3&3k-6&2k-3&2k-4&2k^2-7k+6\\
3&2k-3&3k-6&2k-4&2k^2-7k+6\\
1&2k-2&2k-2&2k-6&2k^2-7k+5\\
2&2k-3&2k-3&2k-5&2k^2-7k+5
\end{pmatrix}.
\]
Its characteristic polynomial is \((x-k+3)R_k(x)\).

Identify \(\mathbb R^A\) with \(\mathbb R^B\) through the perfect matching,
and let \(R_A,R_B,R_C\) be the incidence matrices from the three
neighbor cells to \(Z\).  Let \(T\) be the adjacency matrix on \(Z\).
The Moore identity gives
\[
\begin{gathered}
R_AR_A^{\mathsf T}=R_BR_B^{\mathsf T}=(k-2)I,\quad
R_CR_C^{\mathsf T}=(k-1)I,\\
R_AR_B^{\mathsf T}=J-I,\quad
R_AR_C^{\mathsf T}=R_BR_C^{\mathsf T}=J,\\
R_AT+R_B=J-R_A,\quad
R_BT+R_A=J-R_B,\quad
R_CT=J-R_C,\\
R_A^{\mathsf T}R_A+R_B^{\mathsf T}R_B+R_C^{\mathsf T}R_C+T^2
=(k-1)I-T+J.
\end{gathered}
\]
Moreover, every vertex of \(Z\) has one neighbor in each of \(A,B,C\),
so
\[
 R_A^{\mathsf T}\one_A
 =R_B^{\mathsf T}\one_B
 =R_C^{\mathsf T}\one_C
 =\one_Z,
 \qquad T\one_Z=(k-3)\one_Z.
\]

Use the coordinate order \((\{w\},A,B,C,Z)\).  Choose orthogonal bases
\(\mathcal B_A\) of \(\one_A^\perp\) and
\(\mathcal B_C\) of \(\one_C^\perp\).  For
\(\xi\in\mathcal B_A\) and \(\eta\in\mathcal B_C\), put
\[
\begin{aligned}
u_\xi^\pm&=(0,\xi,\pm\xi,0,0),&
v_\xi^\pm&=(0,0,0,0,
 (R_A^{\mathsf T}\pm R_B^{\mathsf T})\xi),\\
u_\eta^C&=(0,0,0,\eta,0),&
v_\eta^C&=(0,0,0,0,R_C^{\mathsf T}\eta),
\end{aligned}
\]
and define
\[
 W_\xi^\pm=\operatorname{span}\{u_\xi^\pm,v_\xi^\pm\},
 \qquad
 W_\eta^C=\operatorname{span}\{u_\eta^C,v_\eta^C\}.
\]
The incidence identities make the relevant maps injective, so each displayed
span is two-dimensional. They also show that these spaces are mutually
orthogonal and orthogonal to the five-dimensional cell-constant space
\(W_{\mathrm{const}}\).
For the residual space
\[
 K=\ker R_A\cap\ker R_B\cap\ker R_C
\]
one has
\[
 K\subseteq\one_Z^\perp,\qquad
 \dim K=(k-2)(k-4),\qquad T(K)\subseteq K.
\]
Thus \(W_{\mathrm{const}}\), the spaces \(W_\xi^\pm,W_\eta^C\), and
\(0\oplus0\oplus0\oplus0\oplus K\) form an orthogonal direct sum. Their
dimensions account for every surviving vertex, since
\[
 5+4(k-2)+2(k-3)+(k-2)(k-4)=k^2-1.
\]

In the ordered generator pairs
\((u_\xi^+,v_\xi^+)\), \((u_\xi^-,v_\xi^-)\), and
\((u_\eta^C,v_\eta^C)\), respectively, the distance operator has matrices
\[
 \begin{pmatrix}-4&-(k-3)\\-1&0\end{pmatrix},\qquad
 \begin{pmatrix}-2&-(k-1)\\-1&-2\end{pmatrix},\qquad
 \begin{pmatrix}-2&-(k-1)\\-1&-1\end{pmatrix}.
\]
These give the two displayed quadratic factors and \(k-3\) copies of
each Moore linear factor.  On \(K\), the \(Z\)-block identity gives
\[
 T^2+T-(k-1)I=0,
 \qquad D(H)|_K=-2I-T.
\]
The constant direction of \(Z\) has \(T\)-eigenvalue \(k-3\). The
\(Z\)-images of the symmetric, antisymmetric and common-neighbor spaces
have eigenvalues \(-2,0,-1\), with dimensions
\(k-2,k-2,k-3\), respectively.  Since \(\tr T=0\),
\[
 \tr(T|_K)=2(k-2).
\]
The residual adjacency operator has just two possible eigenvalues, so its
dimension and trace determine both multiplicities.  We add the Moore-root
copies from the common-neighbor spaces to obtain the full exponents
\(M_-\) and \(M_+\).

This completes the spectral factorization.  The value
\(\delta^*(H)=k-\frac2k\) again follows from the \(s=2\) case of
Theorem~\ref{thm:small-puncture}. For strictness we do not need the least
root of the quartic. Since \(u,v\in N_M(w)\),
Proposition~\ref{prop:neighborhood-puncture} with \(s=2\) gives
\[
 \Phi(H)\ge
 \frac{3k-5-\Delta-\sqrt{k^2+4k-4}}2-\frac2k.
\]
For \(k\ge7\), one has
\[
 \sqrt{4k-3}<k-1,
 \qquad
 \sqrt{k^2+4k-4}<k+2.
\]
Indeed, the corresponding differences of squares are
\(k(k-6)+4>0\) and \(8>0\).  The preceding lower bound is therefore
greater than
\[
 \frac{k-6}{2}-\frac2k
 \ge\frac12-\frac27=\frac3{14}>0.
\]
For \(k=6\), use \(\sqrt{21}<23/5\) and \(\sqrt{56}<15/2\), which
follow from \(529>525\) and \(225>224\).  The same lower bound is then
greater than
\[
 \frac{13-23/5-15/2}{2}-\frac13=\frac7{60}>0.
\]
Thus \(\Phi(H)>0\) for every realizable integer \(k\ge6\).
The direct-sum decomposition, injectivity, orthogonality, all recomputed
distances, and the strict comparison
are independently checked by two exact computations.
\end{proof}

\section{Finite spectral computations and deletion checks}
\label{app:small-example-certificates}

We record the finite certificates behind the examples and deletion families of
Theorem~\ref{thm:explicit-examples}, and the orbit check of
Theorem~\ref{thm:hs-radius}.
The graph encodings are part of the data described in
Appendix~\ref{sec:verification}; the coordinate definitions of the small counterexamples
are given above Theorem~\ref{thm:explicit-examples}.

For \(H_{38}\), the factor \(x^2+6x+2\), together with the exact Sturm
count, identifies the least distance root as \(-3-\sqrt7\).
An independent rational \(LDL^{\mathsf T}\) decomposition of
\(3D(H_{38})+17I\) has inertia \((38,0,0)\), where the entries count
positive, negative and zero eigenvalues.
For \(H_{39}\), the decomposition of \(6D(H_{39})+35I\) has inertia
\((39,0,0)\). Thus each shift is positive definite and proves strictness.
For \(R\) and \(X_{42}\), the exact spectra in
Theorem~\ref{thm:explicit-examples} identify the least eigenvalue \(-5\).
These graph, metric, dual-degree and matrix checks are reproduced by
exact computation.

The two labeled deletion families of \(R\) have the following common
distance characteristic polynomials.
For every
\(v\in V(R)\),
\begin{align*}
\det(xI-D(R-v))=P_{39}(x):={}&x^9(x+5)^{12}(x^2+6x+3)\\
&\cdot(x^3+3x^2-15x-7)^2(x^3+3x^2-15x-3)^2\\
&\cdot(x^4-78x^3+303x^2-70x-450),
\end{align*}
whereas for every \(uv\in E(R)\),
\begin{align*}
\det(xI-D(R-\{u,v\}))=P_{38}(x):={}&x^4(x-2)(x+5)^8(x^2+6x+2)^2\\
&\cdot(x^3+3x^2-15x-3)\\
&\cdot(x^4+5x^3-7x^2-23x-6)\\
&\cdot(x^5+9x^4+7x^3-77x^2-54x-4)\\
&\cdot(x^9-67x^8-404x^7+1772x^6+7205x^5\\
&\hspace{3em}-18489x^4-17018x^3+20288x^2+16824x+1680).
\end{align*}
The least root of \(P_{38}\) is \(-3-\sqrt7\); that of \(P_{39}\) is the least
root of \(x^3+3x^2-15x-3\), about \(-5.588\), above \(-35/6\).

All \(40\) singleton and \(120\) edge-endpoint deletions are enumerated and checked by computer; the root and
positive-definiteness checks are recorded with each family.

For the orbit check in Theorem~\ref{thm:hs-radius}, the orbit cardinalities
on deletion sets of each size \(s\le5\) sum to \(\binom{50}{s}\), so the
\(52\) representatives cover every deletion set. Each representative with
\(s=4,5\) gives
\[
 \operatorname{inertia}\bigl(7D(M-S)+(49-s)I\bigr)=(50-s,0,0);
\]
the six-vertex set of that theorem gives \((43,1,0)\) for \(7D(M-S)+43I\).

\section{Technical steps in the polynomial order bounds}
\label{app:lp-slack-positivity}

This appendix completes three steps in the proofs of
Theorems~\ref{thm:lp-ceiling} and~\ref{thm:integral-slack}: the positivity
of the dual slack, the exclusion of the graph order \(n=B_k\), and the
nonvanishing of the excess matrix.

\subsection{Positivity of the dual LP slack}

Use the measure, support points, and slack notation from the proof of
Theorem~\ref{thm:lp-ceiling}.  We verify the positivity of every slack beyond
degree four, as used in its optimality and rigidity argument.

We check the first five values exactly and then
bound the tail.  For \(5\le i\le9\), removing the common positive factor
\(k(k-1)(k+2)(k^2+3)/6\) leaves, respectively,
\[
 2,\quad 5k-13,\quad 2(3k^2-17k+25),
\]
\[
 6k^3-47k^2+139k-150,
 \quad
 2(3k^4-27k^3+106k^2-219k+194),
\]
all positive for \(k\ge4\); for the nontrivial residual factors this follows
after writing \(k=m+4\), when all coefficients are nonnegative and the
constant terms are positive.  For \(i\ge10\), put \(r=k-1\ge3\).
The support lies in \([-2\sqrt r,2\sqrt r]\), because
\(1+\sqrt{2r}\le2\sqrt r\).  For \(|z|\le1\), the recurrence gives
\[
 F_i(2\sqrt r\,z)
 =r^{i/2}U_i(z)-r^{(i-2)/2}U_{i-2}(z),
\]
where \(U_j\) is the Chebyshev polynomial of the second kind.
Using \(|U_j(z)|\le j+1\), the total mass \(B_k-1\), and
\(F_i(k)=kr^{i-1}\), we obtain
\[
\begin{aligned}
 \frac{|\mu(F_i)|}{F_i(k)}
 &\le \frac{r^2+4r+6}{6}\,r^{-i/2}\bigl((i+1)r+i-1\bigr)\\
 &\le \frac{2i+1}{3}\,r^{3-i/2}
 \le \frac{2i+1}{3}\,3^{3-i/2}.
\end{aligned}
\]
The second inequality is
\((r^2+4r+6)\bigl((i+1)r+i-1\bigr)\le2(2i+1)r^3\); the difference of the two
sides is \((r-3)\bigl((3i+1)r^2+4ir+2i-2\bigr)\ge0\). The third uses
\(r\ge3\) and \(3-i/2<0\).
At \(i=10\) the right-hand side is \(7/9\), and it decreases thereafter
because \(3(2i+1)^2-(2i+3)^2=8i^2-6>0\).  Hence
\[
 \frac{|\mu(F_i)|}{F_i(k)}<1\qquad(i\ge10).
\]
The principal contribution is positive and exceeds the absolute value of
the measure's contribution. Their sum is therefore positive, proving the
required positivity for every remaining slack.

\subsection{Equality in the closed-window order bound}

\begin{proof}[Completion of the proof of Theorem~\ref{thm:lp-ceiling}]
We use the notation of that proof and of the trace bound before it.
The polynomial bound includes the endpoints of \(I_k\).
We must therefore rule out graph equality separately. Suppose \(n=B_k\).
Equality in the trace bound for the optimal quartic forces every
nonprincipal eigenvalue to belong to
\[
 \{-2,-1-\omega,-1+\omega\},\qquad \omega=\sqrt{2k-2}.
\]
Let \(m_0,m_-,m_+\) be their respective multiplicities.  If
\(\omega\) is irrational, the trace equation
\[
 0=k-(m_-+m_+)-2m_0+\omega(m_+-m_-)
\]
forces \(m_+=m_-\).  Counting the nonprincipal eigenvalues then gives
\(m_0=k-n+1<0\), since
\[
 B_k-(k+1)=\frac{k(k-1)(k+3)}6>0.
\]
This is impossible.  Thus \(\omega\) is rational and, as the square root
of an integer, must be an integer.

The second moment of \(A+I\) now determines \(m_0\).  The principal
eigenvalue contributes \((k+1)^2\).  Each nonprincipal adjacency
eigenvalue \(-1\pm\omega\) contributes \(2k-2\) to this moment,
while each copy of \(-2\) contributes \(1\).  Hence
\[
 (2k-2)(n-1-m_0)+m_0=n(k+1)-(k+1)^2,
\]
and substituting \(n=B_k\) gives
\[
 m_0=\frac{k(k-1)(k^2+3)}{6(2k-3)}.
\]
Since \(m_0\) is an integer, \(2k-3\) divides
\(k(k-1)(k^2+3)\).  The identity
\[
 16k(k-1)(k^2+3)
 =(2k-3)(8k^3+4k^2+30k+21)+63
\]
therefore implies \(2k-3\mid63\).  With \(k\ge4\), this leaves
\(k\in\{5,6,12,33\}\).  Only \(k=33\) has integral
\(\sqrt{2k-2}\), but then \(m_0=9152/3\), again impossible.
Consequently \(n<B_k\) throughout the closed window.
\end{proof}

\subsection{The excess matrix is nonzero}

\begin{proof}[Completion of the proof of Theorem~\ref{thm:integral-slack}]
We use the notation of that theorem and its proof.
We prove \(\mathcal E_k\ne0\). Suppose instead that the excess
matrix is zero. Then
\[
 g_k(A)=(h_k+1)J-I,
\]
so \(p_k(A)=0\) on \(\one^\perp\), where \(p_k(x)=g_k(x)+1\).  After
\(y=x+2\),
\[
 p_k(y-2)=y^4-2y^3+(3-2k)y^2+1.
\]
There are two possible ways to factor this polynomial over \(\mathbb Q\).
A rational linear factor would give a root among \(\pm1\), whose values are
\(3-2k\) and \(7-2k\).  For a factorization into monic integer quadratics,
the constant terms would be both \(1\) or both \(-1\).  The cubic and
linear coefficients would then be equal or opposite, respectively; their
actual values are \(-2\) and \(0\).  Neither factorization is possible.

On the rational zero-sum space, substituting the adjacency operator into
this irreducible quartic gives zero. Its characteristic polynomial must
therefore be a power of the quartic:
\[
 \chi_{A|_{\one^\perp}}(x)=p_k(x)^m,\qquad n-1=4m.
\]
The four roots of \(p_k\) sum to \(-6\), so \(\tr A=0\) gives
\[
 0=k-6m.
\]
Hence \(n-1=4m=2k/3\), contradicting \(n\ge k+1\).  Thus
\(\mathcal E_k\ne0\).
\end{proof}

\section{Proof of the three-to-one bound}
\label{app:three-to-one-proof}

This appendix proves Theorem~\ref{thm:three-to-one}, determines the
parameters at which its bound can be attained, and tabulates the bound.

\begin{proof}[Proof of Theorem~\ref{thm:three-to-one}]
Put \(\mathcal E=\mathcal E_k\), \(C=C_k\), \(h=h_k\), and
\[
 \rho=\frac{\varepsilon-1}{n}
      =\frac{1+r/n}{2}.
\]
The optimal slack matrix has the form
\[
 \mathcal S_k=I-\mathcal E+\rho J\succeq0.
\]
We use two standard graph constructions in the proof. The line graph has
one vertex for each edge of the original graph, with two vertices adjacent
when those edges meet. The original graph is called a root of the line
graph. A cocktail-party graph is a complete graph with a perfect matching
removed.

First suppose \(r\le0\). We show that \(\mathcal E\) is the adjacency matrix
of a simple graph, then exclude its complement by classifying graphs with
bounded least eigenvalue and checking rational eigenspace dimensions.
After proving \(r>0\), we suppose \(n>3r\) and exclude the following four
cases:
\begin{enumerate}
\item an entry \(\mathcal E_{uv}=2\), which forces \(n=3r+2\)
and is excluded by divisibility and parity;
\item a simple \(\mathcal E\) whose complement has at least three components,
which forces \(n\le3r+6\) and is excluded by fixed remainders;
\item a connected complement, whose cocktail-party case contradicts
\(r>0\) and whose line-graph roots are excluded by size and divisibility;
\item a complement with two components, excluded by their classified types
and the remaining fixed remainders.
\end{enumerate}
Here diameter three and the radius-two count give
\(n\ge k^2+2\ge38\).  The classification therefore applies whenever the
complement is connected.  In the final two-component case we first exclude
\(k=6,7,8\); for \(k\ge9\), each component has order greater than \(28\).
The cases also use two divisibility conditions. The first follows from
the identity
\[
 128(2C-r)
 =(4k+9)(64k^3+112k^2+196k+327)+(129-128r).
\]
Since \(n=(2C-r)/(12k+27)\) is integral and \(4k+9\) is odd,
\begin{equation}\label{eq:r-divisibility}
 4k+9\mid129-128r.
\end{equation}
The second condition will be used when the order is close to the desired
bound.  If \(n=3r+t\), the defining equation for \(r\) gives
\(2(18k+41)\,r=2C-(12k+27)t\). Since \(12k+27\) is odd, \(t\) is even, and
\[
 18k+41\mid 2C-(12k+27)t.
\]
Since \(\gcd(18,18k+41)=1\), Euclidean division after multiplication by
\(18^4\) yields the necessary condition
\begin{equation}\label{eq:fixed-remainder}
 18k+41\mid 132650+34992t.
\end{equation}

We first prove positivity of \(r\).  Suppose \(r\le0\).  A
\(2\times2\) principal minor of \(\mathcal S_k\) gives
\[
 \mathcal E_{uv}\le2+\frac rn\qquad(u\ne v).
\]
Thus \(\mathcal E\) is the adjacency matrix of a simple graph when \(r<0\).
If \(r=0\) and \(\mathcal E_{uv}=2\), positivity puts \(e_u+e_v\) in
\(\ker\mathcal S_k\).  After subtracting its projection onto \(\one\), the
strict spectral window gives a \((-2)\)-eigenvector of \(A\).  Its
\(u\)-coordinate would require
\[
 A_{uv}=-2+\frac{2k+4}{n},
\]
which is impossible because \(n\ge k^2+2\) and \(A_{uv}\in\{0,1\}\).
Hence \(\mathcal E\) is simple also when \(r=0\).

Let \(X\) be the graph with adjacency matrix \(J-I-\mathcal E\).
It is regular of degree
\[
 d=n-1-\varepsilon=\frac{n-r-4}{2},
\]
and \(A(X)+2I=I-\mathcal E+J=\mathcal S_k+\frac{n-r}{2n}J\succeq0\) because
\(r\le0\). Each component has at least \(d+1\) vertices, and \(n\ge3(d+1)\)
would give \(n\le3r+6\le6\); so \(X\) has at most two components.
We exclude the remaining possibilities by examining eigenspaces of
\(\mathcal E\). This matrix is a polynomial in \(A\) and \(J\), so it
commutes with \(A\). Each rational eigenspace of \(\mathcal E\) is therefore
\(A\)-invariant.

The graph \(X\) is now simple and regular, and its order is
\(n\ge38\).  If it is connected, the classification of connected regular
graphs of order greater than \(28\) and least eigenvalue at least
\(-2\) makes it a line graph or a cocktail-party graph
\cite{CameronEtAl1976}; see also
\cite[Theorems~4.1.1 and~4.1.5]{CvetkovicEtAl2004} and
\cite[Introduction]{KoolenEtAl2025}.

Suppose first that \(X\) is connected.  The complete case would give
\(\mathcal E=0\), excluded in Theorem~\ref{thm:integral-slack}.  In the
cocktail-party case, \(\mathcal E\) is a perfect matching and \(r=-n\).
The identity
\[
 1296C=(6k+13)(216k^3+396k^2+654k+1175)+277
\]
forces \(k=44\) and \(n=14812\).  On the \(7406\)-dimensional
 \((-1)\)-eigenspace of \(\mathcal E\), the rational operator \(A\) is
 annihilated by \(g_{44}(x)+2\).

If \(X=L(Y)\), choose a connected root \(Y\). Since \(X\) is
\(d\)-regular, every edge \(ab\) of \(Y\) satisfies
\(\deg_Y(a)+\deg_Y(b)=d+2\). Degrees alternate along paths. An odd
cycle forces them to coincide; otherwise \(Y\) is bipartite, with a
constant degree on each part. Thus \(Y\) is regular or semiregular
bipartite. Here semiregular means that the degree is constant within each
part, though the two degrees may differ.

If \(Y\) is \(q\)-regular, then \(q\ge n/4\) and
\(|V(Y)|\le8\), contradicting \(q\le |V(Y)|-1\) because \(n\ge38\).
If \(Y\) is semiregular bipartite with part sizes \(a\ge b\ge2\), then
\[
 \frac1a+\frac1b=\frac{n-r}{2n}\ge\frac12,\qquad n\le ab.
\]
The cases \(b\ge3\) have \(ab<38\).  For \(b=2\), connectedness forces
\(Y=K_{a,2}\), hence \(r=-4\).  Equation~\eqref{eq:r-divisibility} leaves
\(k=158\) and \(n=664748\).  The resulting \((-1)\)-eigenspace of \(\mathcal E\)
 has dimension \(332373\).

If \(X\) has two components, write their orders as
\(d+1+a_1\) and \(d+1+a_2\).  Then
\[
 a_1+a_2=r+2,
\]
so \(r\in\{-2,-1,0\}\).  Checking the divisors in
\eqref{eq:r-divisibility}, together with \(n\ge k^2+2\) and \(kn\) even,
leaves only
\[
 (r,k,n)=(-1,62,40875).
\]
The components then have orders \(20437\) and \(20438\).  Their common
degree is one less than the smaller order, so the smaller component is
complete.  In the larger component, each vertex has exactly one
nonneighbor; these missing edges form a perfect matching, giving the
cocktail-party graph.  Thus \(\mathcal E\) has eigenvalue \(-1\) with
 multiplicity \(10219\).

In each of the three exceptional cases, the eigenspace just found is
rational, orthogonal to the constant vector and preserved by the adjacency
operator. The excess matrix acts there as minus the identity. Its defining
relation therefore makes \(g_k(x)+2\) vanish on that space.  This quartic is irreducible over
\(\mathbb Q\) for every integer \(k\ge6\), except \(k=7\).  To check
irreducibility, substitute \(y=x+2\), giving
 \[
  y^4-2y^3+(3-2k)y^2+2.
 \]
 Gauss's lemma reduces a quadratic factorization to monic integer quadratics
 whose constant terms multiply to \(2\); coefficient comparison leaves only
 \(k=4\) or \(k=7\), while the rational-root alternatives give \(k=2\) or
 \(k=4\).  The quartic is therefore irreducible at degrees \(44,62,158\), but the
 corresponding dimensions \(7406,10219,332373\) are not divisible by four.
 Since an invariant rational space annihilated by an irreducible quartic is a
 vector space over the degree-four field \(\mathbb Q[x]/(g_k(x)+2)\), its
 rational dimension must be divisible by four.
 This contradiction proves \(r>0\).

Now suppose, for a contradiction, that \(n>3r\), and put
\(x=r/n\in(0,1/3)\). To apply the graph classification again, we first
show that the excess matrix is the adjacency matrix of a simple graph.  The \(2\times2\) minors give
\(\mathcal E_{uv}\le2\).  Suppose \(\mathcal E_{uv}=2\), and for
\(w\notin\{u,v\}\) set \(s_w=\mathcal E_{uw}+\mathcal E_{vw}\).
Cauchy--Schwarz, applied to \(e_u+e_v\) and \(e_w\) in the Gram matrix
\(\mathcal S_k\), gives
\[
 (1+x-s_w)^2\le x(3+x),
\]
Thus the nonnegative integer \(s_w\) belongs to \(\{1,2\}\).  Summing
over the rows shows that exactly \(r\) vertices satisfy \(s_w=2\).
We collect these vertices in \(W\), let \(y=\sum_{w\in W}e_w\), and put
\[
  \eta_W=\sum_{\{w,z\}\subseteq W}\mathcal E_{wz}.
\]
 Then
\[
 \begin{aligned}
  (e_u+e_v)^{\mathsf T}\mathcal S_k(e_u+e_v)&=2x,\\
  (e_u+e_v)^{\mathsf T}\mathcal S_k y&=r(x-1),\\
  y^{\mathsf T}\mathcal S_k y&=r+\frac{r^2(1+x)}2-2\eta_W.
 \end{aligned}
\]
 Positivity of this \(2\times2\) Gram determinant yields
 \[
  2xr+r^2(3x-1)-4x\eta_W\ge0.
 \]
Since \(\eta_W\ge0\), this forces \(n\le3r+2\); as \(t\) is even,
\(n=3r+2\). By \eqref{eq:fixed-remainder}, \(18k+41\) divides
\(202634=2\cdot71\cdot1427\), and the only such divisor with \(k\ge6\) is
\(1427\), giving \((k,n,r)=(77,77831,25943)\); but \(kn\) is odd, contrary to
the handshake lemma.  Therefore \(\mathcal E\) is simple.

It follows that \(X\), with adjacency matrix \(J-I-\mathcal E\), is again a
simple graph. It is \(d\)-regular and has least eigenvalue at least \(-2\).
If \(X\) had three or more components, then \(n\ge3(d+1)\), or
\(n\le3r+6\).  This leaves \(t\in\{4,6\}\), since \(t=2\) was excluded
above; the remainders \(272618\) and \(342602\) have no divisor \(18k+41\)
with \(k\ge6\).  Hence \(X\) has at most two components.

For connected \(X\), the cocktail-party case contradicts \(r>0\).
We are left with \(X=L(Y)\), with the same regular/semiregular
dichotomy as above. If \(Y\) is \(q\)-regular on \(v\) vertices, then
 \[
  q=\frac{n-r}{4}>\frac n6,\qquad v=\frac{2n}{q}<12.
 \]
 Simplicity and \(n\ge38\) leave only
 \[
 \begin{aligned}
  (q,v;n,r)\in\{&(8,10;40,8),(9,10;45,9),\\
                 &(8,11;44,12),(10,11;55,15)\}.
 \end{aligned}
 \]
 The radius-two lower bound forces \(k\le7\), and direct substitution in the
 defining formula for \(r\) excludes all four cases.

 For a semiregular bipartite root with part sizes \(a\ge b\ge2\),
 write \(p=n/a\le q=n/b\) for its two degrees.  Then
 \[
  p+q=\frac{n-r}{2}>\frac n3,
  \qquad (b-2)n=b(r+2p),
 \]
 so \(b<6\).  Connectedness excludes \(p=1\).  For \(b=2\), the
 left side of the displayed identity is zero while the right side is
 positive.  For \(b=5\), the degree-sum inequality makes the sum of the
 reciprocals of the part sizes greater than one third.  The larger part
 therefore has at most seven vertices, and the root has at most thirty-five
 edges, below the required line-graph order thirty-eight.  The case \(b=4\) leaves only
 \[
  (p,q,a,b;n,r)=(4,10,10,4;40,12),\ (4,11,11,4;44,14),
 \]
 both excluded by the radius-two bound and the formula for \(r\).
 For \(b=3\), one has \(p=2\) or \(3\), giving \(n=3r+12\) or
 \(n=3r+18\).  Equation~\eqref{eq:fixed-remainder} gives the remainders
 \(552554=2\cdot276277\) and \(762506=2\cdot381253\).  Their odd cofactors
 are prime and congruent to \(13\pmod {18}\), whereas \(18k+41\equiv5
 \pmod {18}\).  Thus \(X\) cannot be connected.

It remains to exclude two components. The conditions \(r>0\) and \(n>3r\)
mean \(3C/(18k+41)<n<2C/(12k+27)\). For \(k=6,7,8\) these intervals,
\((50.26,50.42)\), \((75.66,75.89)\) and \((108.65,108.94)\), contain no
integer; for \(k\ge9\) the left endpoint exceeds \(150\). Both components
then have order greater than \(28\), so
we can apply the same classification to each component.  Suppose a component
of order \(N\) is the line graph of a simple \(q\)-regular graph \(Y\)
on \(v\) vertices.  Its degree is \(d=2q-2\), and the number of edges
of the root is \(N=vq/2\).  Hence
\[
 q=\frac{d+2}{2}=\frac{n-r}{4}>\frac n6>25,
 \qquad
 v=\frac{2N}{q}\le\frac{2n}{q}<12.
\]
Simplicity gives \(q\le v-1\le10\), a contradiction.  If a component
of order \(N\) has a semiregular bipartite root, its part sizes satisfy
 \[
  \frac1a+\frac1b=\frac{d+2}{N}
  \ge\frac{n-r}{n+r+2}>\frac{49}{100}.
 \]
 If \(b\ge3\), the five possible pairs \((a,b)\) have \(ab\le18<N\).
 If \(b=2\), connectedness forces the root \(K_{d,2}\), whose line graph has
 order \(2d\).  Thus every component is \(K_{d+1}\), a cocktail-party graph
 of order \(d+2\), or \(L(K_{d,2})\).  The first two types alone force
 \(r\le0\).  If exactly one component has order \(2d\), then
\(n=3r+8\) or \(n=3r+10\). If both do, then \(n=2r+8\). Together with \(n>3r\), this gives \(r<8\), hence \(n<24\),
  contrary to \(n>150\). For the two remaining mixed cases,
  the divisor checks at the end of this appendix exclude
  \(n=3r+8\) and leave only \(k=123\) for \(n=3r+10\).
  In that case \(kn\) is odd. This final contradiction proves \(n\le3r\).

All cases under the contrary assumption have been excluded. Substitute
the definition of \(r\) into the three-to-one inequality to obtain
 \[
  (18k+41)n\le3(k+2)^2(k^2+3),
 \]
 which is the displayed bound.
\end{proof}

\begin{corollary}[Rigidity at equality]\label{cor:three-to-one-equality}
Under the hypotheses and notation of Theorem~\ref{thm:three-to-one}, equality
in \(n\le3r\) can occur only for
\[
 \resultbox{(k,n,r)=(103,185220,61740).}
\]
This is an arithmetic parameter classification; it does not assert the
existence of a graph with these parameters.
\end{corollary}

\begin{proof}
Equality gives
\[
 (18k+41)r=(k+2)^2(k^2+3)=:C_k.
\]
Exact Euclidean division yields
\[
\begin{aligned}
18^4C_k={}&(18k+41)
 (5832k^3+10044k^2+17946k+29107)\\
&+66325,
\end{aligned}
\]
where \(66325=5^2\cdot7\cdot379\).  Hence \(18k+41\mid66325\).
Since \(k\ge6\), this divisor is at least \(149\) and is congruent to \(5\)
modulo \(18\).  Among the positive divisors of \(66325\), only
\(1895\) has these properties.  Thus \(k=103\), after which direct
substitution gives \(r=61740\) and \(n=3r=185220\).
\end{proof}

\subsection{Order comparisons and divisor checks}
\label{app:three-to-one-arithmetic}

This subsection tabulates the bound of Theorem~\ref{thm:three-to-one} and
completes the two-component case of its proof.
For comparison across degrees, the diameter-three bounds before parity adjustment in degrees \(6\) through
\(20\) are
\[
\begin{array}{c|rrrrrrrrrrrrrrr}
 k&6&7&8&9&10&11&12&13&14&15&16&17&18&19&20\\
 \hline
 n&50&75&108&150&201&263&336&422&521&635&765&911&1075&1257&1459.
\end{array}
\]
When \(k\) is odd, \(kn=2|E(G)|\) requires \(n\) to be even; in particular,
the entries for \(k=7,11,15,17,19\) improve to
\(74,262,634,910,1256\), respectively.

In the two-component case of that proof, \(n=2r+8\) is excluded directly,
leaving \(n=3r+t\) with \(t\in\{8,10\}\), and \eqref{eq:fixed-remainder}
requires \(18k+41\mid132650+34992t\). The middle column below gives this
remainder; enumerating its positive divisors of the form \(18k+41\), with
\(k\ge6\), gives the last column.
\[
\begin{array}{ccc}
\toprule
t&\text{remainder}&\text{candidate }k\\
\midrule
8&412586&\text{none}\\
10&482570&123\\
\bottomrule
\end{array}
\]
For \(t=10\), \(k=123\) gives \(r=104847\) and \(n=314551\), so \(kn\) is
odd, contradicting the handshake identity. These divisions and candidates are
checked over the integers by computer.

\section{The degree-five case}
\label{app:degree-five}

This appendix completes the proof of Theorem~\ref{thm:regular-degree-six}.
That proof reduced the case of a \(5\)-regular strict counterexample \(G\) to
diameter three and order \(n\in\{30,32\}\); we exclude both orders.

\begin{proof}[Completion of the proof of Theorem~\ref{thm:regular-degree-six} for \(k=5\)]
At \(n=32\), restrict the
adjacency quadratic form to vectors constant on each distance layer about
a vertex. The layer sizes are \(1,5,20,6\). Write \(a\) for the average
internal degree in the distance-two layer.  Each vertex of that layer has
one neighbor in the first layer, so the edges from the second to the third
layer satisfy \(20(4-a)\le6\cdot5\), giving \(a\ge5/2\).
We use the layer indicators normalized to have length one. Restricting the
adjacency quadratic form to their span gives the following symmetric
compression; the argument does not require an equitable layer partition:
\[
 Q(a)=
 \begin{pmatrix}
 0&\sqrt5&0&0\\
 \sqrt5&0&2&0\\
 0&2&a&(4-a)\sqrt{10/3}\\
 0&0&(4-a)\sqrt{10/3}&5-\frac{10}{3}(4-a)
 \end{pmatrix}.
\]
The derivative of its only \(a\)-dependent block is
\[
 \begin{pmatrix}1&-\sqrt{10/3}\\-\sqrt{10/3}&10/3\end{pmatrix}
 =
 \begin{pmatrix}1\\-\sqrt{10/3}\end{pmatrix}
 \begin{pmatrix}1&-\sqrt{10/3}\end{pmatrix}\succeq0.
\]
Because the derivative is positive semidefinite, no ordered eigenvalue of
\(Q(a)\) decreases as \(a\) increases. Write
\(5=\mu_1(a)\ge\mu_2(a)\ge\mu_3(a)\ge\mu_4(a)\).
Here \(\mu_1(a)=5\): \((1,\sqrt5,\sqrt{20},\sqrt6)\) is an eigenvector of
\(Q(a)\) for every \(a\), and a compression of \(A\) has no eigenvalue
above \(5\).
It is therefore enough to bound the second eigenvalue at the smallest
feasible value,
\[
 \chi_{Q(5/2)}(x)=\frac14(x-5)p_{5,6}(x),
 \qquad
 p_{5,6}(x)=4x^3+10x^2-16x-30,
 \qquad p_{5,6}(11/6)=-29/27<0.
\]
Since \(p_{5,6}(11/6)<0\) and \(p_{5,6}(x)\to+\infty\) as
\(x\to+\infty\), its largest root satisfies
\(\mu_2(5/2)>11/6>-1+\sqrt8\).  Monotonicity and Poincaré separation give
\[
 \theta_2(A)\ge\mu_2(a)\ge\mu_2(5/2),
\]
where \(\theta_2(A)\) is the second-largest adjacency eigenvalue; this
contradicts the bound \(\theta_2(A)<-1+\sqrt8\) given by
Theorem~\ref{thm:diameter-three-score}.  At \(n=30\),
Meringer's isomorph-free enumeration leaves exactly four
\((5,5)\)-cages \cite[p.~142]{Meringer1999}; each fixed record has an exact
distance eigenvalue at most \(-5\). The characteristic-polynomial factors and their
roots are listed below;
the classification ensures that the four records cover every case; the factors give the
distance-spectral obstruction.
The accompanying release contains an independent exact check of the complete
case split and of the four fixed graph6 records.
\end{proof}

For the four degree-five cages, the displayed factor divides
\(\det(xI-D)\) over \(\mathbb Z[x]\). Its negative root is an eigenvalue,
which suffices for \(\lambda_{\min}(D)<-5\); no assertion that it is the
least eigenvalue is needed.
\[
\begin{array}{lcc}
\toprule
\text{fixed graph6 record}&\text{distance factor}&\text{eigenvalue}\\
\midrule
\text{Foster}&x^2+6x-11&-3-2\sqrt5\\
\text{Meringer}&x+6&-6\\
\text{Robertson--Wegner}&x+6&-6\\
\text{Wong}&x^2+6x-11&-3-2\sqrt5\\
\bottomrule
\end{array}
\]
The records have order \(30\), degree five, girth five and diameter three,
and are pairwise nonisomorphic. Meringer's classification
\cite[p.~142]{Meringer1999} supplies completeness.
The integer characteristic polynomials, graph properties and
nonisomorphism checks are reproduced by exact computation.

\section{Local cycle counts and the degree-six, order-fifty case}
\label{app:order50-details}

We collect necessary conditions for the degree-six, order-fifty case left
open by the main order bound. The first subsection holds in every degree:
it applies the principal-minor method of Section~\ref{sec:lp} to the
number of five-cycles through an edge. The second bounds the multiplicity
of the adjacency eigenvalue minus two without assuming strictness. The
remaining subsections assume a strict counterexample and derive local
cycle and signed-graph restrictions.

\subsection{Five-cycle counts through an edge}

\begin{proposition}[Edge-local cycle bounds]\label{prop:edge-cycle}
Let \(G\) be connected and \(k\)-regular, where \(k\ge4\), of girth at least
five and diameter
three, and suppose
\[
 |\theta+1|\le\sqrt{2k-2}
 \qquad(\theta\ne k)
\]
for every adjacency eigenvalue \(\theta\).  For an edge \(uv\), let
\(\sigma_{uv}\) be the number of
\(5\)-cycles containing that edge.  Then
\[
 2k-2\le\sigma_{uv}\le
 \frac{2(k+2)^2(k^2+3)}n-10k-26.
\]
If \(n=k^2+1+c\), then also
\[
 \sigma_{uv}\ge(k-1)^2-c.
\]
\end{proposition}

\begin{proof}
Let \(\mathcal S_k=-g_k(A)+(C_k/n)J\) be as in
Theorem~\ref{thm:integral-slack}; its eigenvalues
\((2k-2-(\theta+1)^2)(\theta+2)^2\) and \(0\) show \(\mathcal S_k\succeq0\)
under the closed window. By the expansion in the proof of
Theorem~\ref{thm:lp-ceiling},
\(g_k=6(k+2)F_0+(4k+14)F_1+(k+13)F_2+6F_3+F_4\).
For an edge \(uv\), girth at least five gives
\((F_2(A))_{uv}=(F_3(A))_{uv}=0\), while \((F_4(A))_{uv}\) counts
nonbacktracking walks \(u\,a\,b\,c\,v\); by girth these are paths, which close
up with \(uv\) to the five-cycles through \(uv\).
Hence \((g_k(A))_{uu}=6(k+2)\) and \((g_k(A))_{uv}=4k+14+\sigma_{uv}\), so
the diagonal and edge entries of \(\mathcal S_k\) are
\[
 a=\frac{C_k}{n}-6(k+2),
 \qquad
 b=\frac{C_k}{n}-(4k+14)-\sigma_{uv}.
\]
The principal submatrix on \(\{u,v\}\) is
\(\bigl(\begin{smallmatrix}a&b\\b&a\end{smallmatrix}\bigr)\), so
\(a\ge0\) and \(-a\le b\le a\).  The inequality \(b\le a\) gives
\(\sigma_{uv}\ge2k-2\), while \(b\ge-a\) gives the stated upper bound.

For the final bound, every radius-two ball has size \(k^2+1\).  The set
\[
 \{u,v\}\cup(N(u)\setminus\{v\})\cup(N(v)\setminus\{u\})
\]
contains \(2k\) vertices and lies in \(B_2(u)\cap B_2(v)\).  Every further
intersection vertex is at distance two from both endpoints, so it is the
middle vertex \(b\) of exactly one path \(u\,a\,b\,c\,v\) above; conversely,
the middle vertex of each such path is such an intersection vertex.  Hence
\[
 |B_2(u)\cap B_2(v)|=2k+\sigma_{uv}.
\]
Inclusion--exclusion and \(n=k^2+1+c\) now give
\(\sigma_{uv}\ge(k-1)^2-c\).
\end{proof}

\paragraph{Strict form of Proposition~\ref{prop:edge-cycle}.}
Under the same hypotheses, with a strict spectral window, both endpoint
bounds improve to
\[
 2k-1\le\sigma_{uv}
 <\frac{2(k+2)^2(k^2+3)}n-10k-26.
\]
For the lower bound, the proof of
Theorem~\ref{thm:integral-slack} gives
\((g_k(A))_{uv}\ge h_k+1=6(k+2)+1\), whereas the edge calculation above
gives \((g_k(A))_{uv}=4k+14+\sigma_{uv}\).
For the upper bound, use \(\mathcal S_k,a,b\) from the proof above. For a
positive semidefinite matrix, a vector with zero quadratic form lies in the
kernel. Thus equality \(b=-a\) would put \(e_u+e_v\) in \(\ker\mathcal S_k\).
In the strict window,
\[
 \ker\mathcal S_k=\langle\one\rangle\oplus\ker(A+2I),
\]
so \(w=e_u+e_v-(2/n)\one\) would satisfy \(Aw=-2w\).
At the \(u\)-coordinate, this gives
\(1-2k/n=-2+4/n\), hence \(n=2(k+2)/3<k+1\), contradicting
the order bound for a simple \(k\)-regular graph. Thus \(b>-a\), proving
the strict upper bound.

For \(\ell\ge3\), let \(N_\ell\) denote the number of
\(\ell\)-cycles in \(G\).

\begin{corollary}[Edge--cycle divisibility sieve]
\label{cor:edge-cycle-sieve}
Under the hypotheses of Proposition~\ref{prop:edge-cycle}, put
\[
 L_{k,n}=\max\{2k-2,\;2k^2-2k+2-n\},
 \qquad
 U_{k,n}=\frac{2(k+2)^2(k^2+3)}{n}-10k-26.
\]
Necessarily
\[
 \resultbox{\lceil L_{k,n}\rceil\le\lfloor U_{k,n}\rfloor.}
\]
If both sides equal an integer \(s\), then every edge of \(G\) lies in
exactly \(s\) five-cycles and
\[
 \resultbox{5\mid \frac{skn}{2}.}
\]
Under the strict spectral window, put
\(L^{\mathrm{str}}_{k,n}=\max\{2k-1,\;2k^2-2k+2-n\}\).
Then the stronger necessary condition is
\[
 \lceil L^{\mathrm{str}}_{k,n}\rceil
 \le\lceil U_{k,n}\rceil-1.
\]
If these two integer bounds coincide at \(s\), the same divisibility
conclusion holds.
\end{corollary}

\begin{proof}
Writing \(n=k^2+1+c\), the two lower bounds in
Proposition~\ref{prop:edge-cycle} combine to
\[
 \sigma_{uv}\ge
 \max\{2k-2,(k-1)^2-c\}=L_{k,n},
\]
while its upper bound is \(U_{k,n}\).  Since \(\sigma_{uv}\) is an integer,
the first conclusion follows.  If the two integer bounds coincide at \(s\),
then \(\sigma_{uv}=s\) for every edge.  Counting edge--five-cycle incidences
gives
\[
 5N_5=\sum_{uv\in E(G)}\sigma_{uv}
 =s|E(G)|=\frac{skn}{2},
\]
which proves the divisibility condition.
The strict-window version follows from the strengthened edge bounds above:
for integral \(\sigma_{uv}\), the inequality \(\sigma_{uv}<U_{k,n}\) is
equivalent to \(\sigma_{uv}\le\lceil U_{k,n}\rceil-1\).
\end{proof}

\paragraph{The first excluded order in degree six.}
For a connected \(6\)-regular strict counterexample, the proof of
Theorem~\ref{thm:degree-six-fifty} shows that the diameter is three, and
Theorem~\ref{thm:diameter-three-score} gives the spectral window, so
Proposition~\ref{prop:edge-cycle} applies. This gives a short local
obstruction to the first order excluded by
Theorem~\ref{thm:degree-six-fifty}.
If \(n=51\), Corollary~\ref{cor:edge-cycle-sieve} gives
\(\lceil L_{6,51}\rceil=\lfloor U_{6,51}\rfloor=11\), forcing
\[
 5N_5=153\cdot11=1683,
 \qquad 5\nmid 1683,
\]
which is impossible.

\subsection{The adjacency multiplicity bound at order fifty}

\begin{proposition}[The order-\(50\) \((-2)\)-multiplicity bound]
\label{prop:order50-minus-two}
Let \(G\) be \(6\)-regular, of order \(50\), and of girth at least five.
The multiplicity \(m_{-2}(A)\) of the adjacency eigenvalue \(-2\) satisfies
\[
 m_{-2}(A)\le20.
\]
\end{proposition}

\begin{proof}
We isolate the eigenvalue whose multiplicity is to be bounded.  Put
\(m=m_{-2}(A)\), remove the principal eigenvalue \(6\) and the \(m\)
copies of \(-2\), and write \(\mu_0,\ldots,\mu_4\) for the moments of
the remaining eigenvalues.  Girth at least five gives
\[
 \tr A=0,\quad \tr A^2=300,\quad \tr A^3=0,\quad \tr A^4=3300,
\]
and therefore
\[
 (\mu_0,\mu_1,\mu_2,\mu_3,\mu_4)
 =(49-m,-6+2m,264-4m,-216+8m,2004-16m).
\]
The moment matrix
\[
 \mathsf H=\begin{pmatrix}
 \mu_0&\mu_1&\mu_2\\
 \mu_1&\mu_2&\mu_3\\
 \mu_2&\mu_3&\mu_4
 \end{pmatrix}
\]
is the Gram matrix obtained by evaluating the constant, linear, and
quadratic monomials at the remaining eigenvalues, counted with multiplicity.
It is therefore positive semidefinite, so its determinant is nonnegative.
Expanding that determinant gives
\[
 \det\mathsf H=3600(1625-81m)\ge0.
\]
Since \(m\) is an integer, \(m\le20\).
\end{proof}

\subsection{Cycle counts at order fifty}

The order bound still allows the degree-six boundary case. We examine its
local cycle counts and a representation of its signed matrix by inner
products of vectors. These necessary conditions restrict the possible
structure without settling existence.

\begin{proposition}[An integer-sensitive layer bound]\label{prop:order50-layer-vector}
Let $G$ be a finite simple connected $6$-regular graph of order $50$,
girth at least five and diameter three. If $\tau(v)$ counts five-cycles
through $v$, then, for every vertex $v$,
\[
 \Phi(G)\le\frac{795-21\tau(v)}{377}.
\]
In particular, $\tau(v)\ge38$ implies $\Phi(G)\le-3/377<0$.
\end{proposition}
\begin{proof}
Fix $v$ and put $L_i=\Gamma_i(v)$. Girth and diameter give the four layer
sizes $1,6,30,13$.
Every vertex of $L_2$ has one parent in $L_1$, and an edge inside $L_2$
corresponds to exactly one five-cycle through $v$. Thus, writing
$\tau=\tau(v)$,
\[
 e(L_2)=\tau,\qquad e(L_2,L_3)=150-2\tau.
\]
For $u\in L_2$, let $s_u$ count its neighbors in $L_2$; for $w\in L_3$,
let $e_w$ count its neighbors in $L_3$. Counting degrees within and between
these layers gives
\[
 \sum_{u\in L_2}s_u=2\tau,\qquad
 \sum_{w\in L_3}e_w=2\tau-72\ge0.
\]
Since these degrees are integers, the inequalities
$(s_u-2)(s_u-3)\ge0$ and $e_w(e_w-1)\ge0$ imply
\[
 \sum_{u\in L_2}s_u^2\ge10\tau-180,\qquad
 \sum_{w\in L_3}e_w^2\ge2\tau-72.
\]
To use these degree sums in a Rayleigh bound, take $f$ constant on
$L_0,L_1,L_2,L_3$ with values $28,9,-1,-4$.
Then $f\perp\one$, $\|f\|^2=1508$, and $(A+I)f$ has respective entries
$82$, $32$, $3s_u-12=-1+9-s_u-4(5-s_u)$ and $-10-3e_w=-4-(6-e_w)-4e_w$.
Since $\sum_u(3s_u-12)^2=9\sum_us_u^2-144\tau+4320$ and
$\sum_w(10+3e_w)^2=1300+60\sum_we_w+9\sum_we_w^2$, the degree-sum bounds
give
\[
 \begin{aligned}
 \|(A+I)f\|^2
 &=82^2+6\cdot32^2+\sum_{u\in L_2}(3s_u-12)^2
                       +\sum_{w\in L_3}(10+3e_w)^2\\
 &\ge12868+(2700-54\tau)+(138\tau-3668)
 =11900+84\tau.
 \end{aligned}
\]
The distance identity is $D+6I=3J+10I-(A+I)^2$.
Since $G$ is regular, $\Phi(G)=\lambda_{\min}(D+6I)$; the Rayleigh
principle applied to this nonzero $f$ gives
\[
 \Phi(G)\le\frac{10\|f\|^2-\|(A+I)f\|^2}{\|f\|^2}
 \le\frac{3180-84\tau}{1508}=\frac{795-21\tau}{377}.
\]
The calculation uses the fact that the degrees within each layer are integers; no
equitable-partition assumption is used.
\end{proof}

\begin{theorem}\label{thm:order50-feasibility}
Let \(G\) be a connected \(6\)-regular graph of order \(50\) and girth
at least five, and suppose \(\Phi(G)>0\).  Then \(G\) has diameter three.
For an edge \(e\), let \(\sigma_e\) be the number of five-cycles containing
\(e\).  Then \(\sigma_e\in\{12,13\}\).  Call \(e\) low or high according as
\(\sigma_e=12\) or \(13\), respectively, and let \(H\) be the spanning
subgraph of high edges.  Put \(m=|E(H)|\), and let \(\tau(v)\) be the number
of five-cycles through \(v\).  Then
\[
 \tau(v)\in\{36,37\},
 \qquad d_H(v)=2\tau(v)-72\in\{0,2\},
\]
\[
 m\equiv0\pmod5,
 \qquad N_5=360+\frac m5.
\]
For a two-edge path \(u-v-w\), put
\[
 R_{uvw}=6\alpha_{uvw}+\beta_{uvw},
\]
where \(\alpha\) and \(\beta\) count the five- and six-cycles containing the
path.  The allowed values are
\[
\begin{array}{c|c}
\text{types of the two incident edges}&R_{uvw}\\
\hline
\text{low--low}&30,31,32\\
\text{mixed}&30,31,32\\
\text{high--high}&30,31.
\end{array}
\]
Writing \(S_2=\sum_vd_H(v)^2\), one has
\[
 1950-m\le N_6
 \le2200-\frac{5m}{6}-\frac{S_2}{12},
\]
\[
 \frac{43m^2-70200m+119632500}{58500}
 \le N_6
 <\frac{4220000-2200m-7m^2}{2000}.
\]
Writing \(n_i\) for the number of vertices \(v\) with \(d_H(v)=i\), where
\(i\in\{0,2\}\), one has \(n_0+n_2=50\) and \(m=n_2\equiv0\pmod5\), so
\[
 (n_0,n_2)=(50-5j,5j)\qquad\text{for some }j\in\{0,1,\ldots,10\}.
\]
Each of these eleven pairs admits at least one integer \(N_6\) satisfying all
four displayed bounds, including the strict upper moment bound. Thus the
bounds restrict \(N_6\), but do not eliminate any further degree profiles.
\end{theorem}

\begin{proof}
Theorems~\ref{thm:endpoint-diameter} and~\ref{thm:diameter-four} exclude
diameter at least four. Diameter at most two is also impossible:
Proposition~\ref{prop:radius-two} would give \(50=1+6^2=37\).
Thus the diameter is three. Theorem~\ref{thm:diameter-three-score} now gives
\(|\theta+1|<\sqrt{10}\) for every nonprincipal adjacency eigenvalue.
Corollary~\ref{cor:edge-cycle-sieve} gives
\[
 L_{6,50}=12,\qquad U_{6,50}=\frac{346}{25}<14.
\]
Since \(\sigma_e\) is integral, \(\sigma_e\in\{12,13\}\) for every edge
\(e\).  For a fixed vertex, the layer degree sums in
Proposition~\ref{prop:order50-layer-vector} give $\tau\ge36$.
Its Rayleigh bound and $\Phi(G)>0$ give $\tau<265/7$, hence
$\tau\le37$. Thus $\tau\in\{36,37\}$. Finally,
\(\sum_{e\ni v}\sigma_e=2\tau(v)\), so the high-edge degree is
\(2\tau(v)-72\); counting edge--five-cycle incidences gives
\(5N_5=12\cdot150+m\).

The next restrictions come from a two-path \(u-v-w\). Girth at least
five gives \((A^3)_{uw}=\alpha_{uvw}\) and
\((A^4)_{uw}=16+\beta_{uvw}\). Here \(16=3\cdot6-2\) counts the
length-four walks with an immediate reversal: the three reversal
positions give six choices each, with two overlaps. The remaining walks
correspond to six-cycles containing \(u-v-w\).
Let \(h_e=0\) for a low edge and
\(h_e=1\) for a high edge.
Here \(g_6(x)=x^4+6x^3+3x^2-28x-36\) and \(C_6/50=2496/50\). Girth five gives
\((A^2)_{uu}=6\), \((A^4)_{uu}=66\) and, on an edge, \((A^3)_{uv}=11\),
\((A^4)_{uv}=\sigma_{uv}\). Hence \(g_6(A)\) has diagonal entries \(48\), edge
entries \(38+\sigma_{uv}=50+h_{uv}\), and entry \(19+R_{uvw}\) at \(\{u,w\}\).
Substitution into the centered slack matrix gives
\[
 25\mathcal S_6[\{u,v,w\}]=
 \begin{pmatrix}
 48&-2-25h_{uv}&773-25R_{uvw}\\
 -2-25h_{uv}&48&-2-25h_{vw}\\
 773-25R_{uvw}&-2-25h_{vw}&48
 \end{pmatrix}\succeq0.
\]
The \(2\times2\) minors give \(R_{uvw}\in\{29,30,31,32\}\).
At \(R_{uvw}=29\), the Gram norm of \(e_u-e_w\) vanishes, so this
vector lies in the kernel of \(\mathcal S_6\). On \(\one^\perp\),
\(\mathcal S_6=-g_6(A)\). The strict shifted window excludes the
endpoint zeros of \(g_6\), leaving only its double zero at \(-2\).
Thus \(A(e_u-e_w)=-2(e_u-e_w)\), which is impossible: its
\(u\)-coordinate would have to be both zero and \(-2\).
If both edges are high and \(R_{uvw}=32\), the displayed determinant
is \(-33750\), also impossible. This proves the local table.

To pass from the local table to global cycle bounds, sum over unoriented
two-edge paths.
Each five-cycle contains five such paths and each six-cycle contains six,
so
\[
 \sum_{u-v-w}R_{uvw}=30N_5+6N_6
                   =10800+6m+6N_6.
\]
There are \(50\binom62=750\) paths, of which
\[
 \sum_v\binom{d_H(v)}2=\frac{S_2-2m}{2}
\]
have two high edges. Every local value is at least \(30\) and at most
\(32\), with upper bound \(31\) for a high--high path. Hence
\[
 22500\le10800+6m+6N_6
 \le24000-\frac{S_2-2m}{2}.
\]
Rearranging gives the first pair of \(N_6\) bounds.
Since \(d_H(v)\in\{0,2\}\), \(S_2=4m\), and the first upper bound is
\(2200-7m/6\).
For the second pair, we evaluate polynomials on the shifted nonprincipal
spectrum. A positive semidefinite Gram matrix gives the lower bound. For the
strict upper bound, the spectral window gives positive weights, and diameter
three ensures that the weighted Gram matrix is positive definite.
The last subsection of this appendix gives this support argument, both
matrices and their determinant calculations.
For each of the eleven profiles the second pair is sharper; it leaves
\(2045\le N_6\le2109\) at \(m=0\) and \(1987\le N_6\le2046\) at \(m=50\).
The exact matrix calculations for the four $N_6$ bounds, the
layer-vector calculation and the eleven-profile check are reproduced by
exact computation.
These eleven profiles are necessary data, not constructions of graphs.
\end{proof}

\begin{corollary}[Eleven high-edge profiles]\label{cor:high-edge-profile-filter}
Under the hypotheses of Theorem~\ref{thm:order50-feasibility}, the high-edge
graph $H$ is a disjoint union of cycles of length at least five, covering
$5j$ vertices for some $j\in\{0,1,\ldots,10\}$, and $50-5j$ isolated
vertices.
\end{corollary}
\begin{proof}
All degrees of $H$ are zero or two, so every nontrivial component is a
cycle. As a subgraph of $G$, it has no triangle or four-cycle.
The cycles cover $n_2=m$ vertices, a multiple of five by
Theorem~\ref{thm:order50-feasibility}.

The intrinsic degree and girth conditions on $H$ alone cannot eliminate
any of these eleven values of $j$: for $j\ge1$, the auxiliary graph
$C_{5j}\sqcup(50-5j)K_1$ has exactly these data; for $j=0$, use the empty
graph on fifty vertices. These auxiliary graphs are not asserted to occur
as high-edge subgraphs of suitable graphs $G$. Further exclusion must use
their relation to $G$, beyond the abstract degree and girth conditions.
\end{proof}

\subsection{A signed graph from the excess matrix}

We next compare eigenspaces of the adjacency and excess matrices. In a
signed adjacency matrix, a nonzero off-diagonal entry records an edge and
gives its sign.

\begin{proposition}[Signed-complement bridge]\label{prop:signed-complement}
Under the hypotheses and notation of Theorem~\ref{thm:three-to-one}, assume
\(0<r<n\) and define
\[
 S=(6k+14)J-2I-g_k(A).
\]
Then \(S\) is a signed adjacency matrix:
\[
 S_{uu}=0,\qquad S_{uv}\in\{-1,0,1\}\quad(u\ne v).
\]
It has constant signed row sum \((n-r-4)/2\), and
\[
 S+2I\succeq0,\qquad E_{-2}(S)=E_{-2}(A).
\]
\end{proposition}

\begin{proof}
Write \(x=r/n\in(0,1)\), \(\rho=(1+x)/2\), and
\(\mathcal E=\mathcal E_k\).  Since
\[
 \mathcal S_k=I-\mathcal E+\rho J\succeq0,
\]
its \(2\times2\) principal minors give
\(-1\le\mathcal E_{uv}\le2+x\).  The entries of \(\mathcal E\) are nonnegative
integers, so \(\mathcal E_{uv}\in\{0,1,2\}\) for \(u\ne v\), and
\(S=J-I-\mathcal E\) has the asserted entries (so \(S=A(X)\), with \(X\) as in
the proof of Theorem~\ref{thm:three-to-one} in
Appendix~\ref{app:three-to-one-proof}, when \(\mathcal E_k\) is a
\(0/1\) matrix).  The row sum follows from
\(\varepsilon=(n+r+2)/2\).  Moreover,
\[
 S+2I
 =\mathcal S_k+\frac{n-r}{2n}J\succeq0.
\]
To identify the shared eigenspace, use the fact that \(A\) and \(S\)
commute. A nonprincipal adjacency eigenvector with eigenvalue \(\theta\)
is an eigenvector of \(S\) with eigenvalue \(-2-g_k(\theta)\).
Inside the open shifted WOW window, \(g_k\) vanishes only at
\(\theta=-2\), since its other zeros are the excluded endpoints.
The eigenvalue of \(S\) on the constant vector is
\((n-r-4)/2>-2\), since \(r<n\).
Thus the two \((-2)\)-eigenspaces coincide.
\end{proof}

\begin{remark}[Signed-root Gram formulation]\label{rem:signed-root}
Under the hypotheses of Theorem~\ref{thm:order50-feasibility}, put
\[
 T=50J-g_6(A)-2I.
\]
Then
\[
 T\one=2\one,\qquad T+2I\succeq0,
 \qquad 25\mathcal S_6=25(T+2I)-2J.
\]
This is the specialization of Proposition~\ref{prop:signed-complement}:
the excess parameter is \(r=42\), so the signed row sum is \(2\).
By Proposition~\ref{prop:signed-complement}, \(T_{uu}=0\) and
\(T_{uv}\in\{-1,0,1\}\) for \(u\ne v\). Explicitly, \(T_{uv}=-h_{uv}\) on
edges, \(T_{uw}=31-R_{uvw}\) at distance two, and \(T_{uz}=50-q_{uz}\) at
distance three, where \(q_{uz}=6(A^3)_{uz}+(A^4)_{uz}\in\{49,50,51\}\).
 Consequently \(T+2I\) is the Gram matrix of \(50\) vectors of norm
 \(\sqrt2\), with inner products in \(\{-1,0,1\}\) between distinct vectors.
 Its row sum is \(4\).  The kernel of \(T+2I\) is \(E_{-2}(A)\), so
 Proposition~\ref{prop:order50-minus-two} gives
 \[
  \operatorname{rank}(T+2I)\ge30.
 \]
Thus the unresolved degree-six boundary can be viewed as a
 root-type integral Gram problem constrained by the graph's distance classes
 and local cycle data.
\end{remark}

\subsection{Why the signed complement is disconnected}

Connectedness of a signed graph means connectedness after forgetting the
edge signs. We show that connectedness here would force a root
representation incompatible with the parity of the number of negative
edges.

\begin{theorem}[Disconnected signed complement]\label{thm:order50-disconnected}
Under the hypotheses of Theorem~\ref{thm:order50-feasibility}, the edge-signed
graph with adjacency matrix \(T\) from Remark~\ref{rem:signed-root} is
disconnected.
\end{theorem}

\begin{proof}
Let \(N_i\) denote the number of \(i\)-cycles in \(G\).  In terms of the degree-six
nonbacktracking polynomials,
\[
\begin{aligned}
(g_6+2)^2={}&28144F_0+18220F_1+8838F_2+3576F_3+1233F_4\\
&+352F_5+78F_6+12F_7+F_8.
\end{aligned}
\]
Girth at least five gives \(\tr F_i(A)=0\) for \(1\le i\le4\).  For
lengths five and six, every closed nonbacktracking walk is a directed cycle.
At lengths seven and eight there are also the walks obtained by attaching a
one-edge tail to a directed five- or six-cycle; a tail at a cycle vertex
avoids the two cycle edges there, so each directed five- or six-cycle
carries \(5\cdot4\) or \(6\cdot4\) such walks.  Hence
\[
\begin{aligned}
\tr F_5(A)&=10N_5,&
\tr F_6(A)&=12N_6,\\
\tr F_7(A)&=14N_7+40N_5,&
\tr F_8(A)&=16N_8+48N_6.
\end{aligned}
\]
These walk counts give the full trace of the squared polynomial.  On the
nonprincipal space, the signed matrix and the polynomial differ only by a
sign, so their squares have the same trace there.  We therefore subtract the
principal adjacency contribution \((g_6(6)+2)^2=2498^2\) and add the square
of the signed matrix's eigenvalue \(2\) on the constant vector, giving
\[
\tr T^2
=8(500N_5+123N_6+21N_7+2N_8-604100).
\]
If \(P_+\) and \(P_-\) are the numbers of positive and negative signed edges,
then the signed row sum and the zero diagonal give
\[
P_+-P_-=50,\qquad
\tr T^2=2(P_++P_-)=100+4P_-.
\]
The polynomial expression for the trace is a multiple of eight.
Comparing it with the signed-edge count shows that \(P_-\) is odd.

We show that connectedness is incompatible with this parity.  Suppose
\(T\) is connected.  Proposition~\ref{prop:order50-minus-two} gives
\(\operatorname{rank}(T+2I)\ge30\).  The representation theorem for
connected edge-signed graphs with smallest eigenvalue at least \(-2\)
\cite[Theorem~2]{GreavesEtAl2015} (see also \cite{CameronEtAl1976})
represents \(T+2I\) in a root system of
type \(D_\ell\) or \(E_8\).  The exceptional root system has rank eight. The rank bound therefore
rules out \(E_8\), leaving
\[
B^{\mathsf T}B=T+2I
\]
for an integral matrix \(B\) whose columns \(b_u\) are roots
\(\pm e_i\pm e_j\).  Put \(s=B\one\).  Since
\((T+2I)\one=4\one\),
\[
b_u\mathbin{\cdot}s=4\quad\text{for every }u,
\qquad
\lVert s\rVert^2=200.
\]
Changing signs of coordinate axes if necessary, assume \(s_i\ge0\).
Form a multigraph with the used coordinates as vertices and one edge
for each root column, joining its two nonzero coordinates.
This is the coordinate-support multigraph.
It is connected: otherwise the columns of \(B\) would split into
two families with zero inner products between them, and \(T\)
would be disconnected.  If \(v\) coordinates are used, then
\[
30\le v\le51:
\]
the lower bound is the Gram rank, and a connected support multigraph with
fifty root-edges has at most fifty-one vertices.

Call the nonnegative coordinate values the levels. Each root joins two
coordinates whose levels satisfy \(\pm s_i\pm s_j=4\).  Along such an edge, the levels either differ by
four or sum to four.  Starting from one coordinate and following
the connected support multigraph, all levels must therefore lie in one of
\[
0,4,8,\ldots;\qquad
2,6,10,\ldots;\qquad
1,3,5,\ldots.
\]
The first family would make \(\lVert s\rVert^2\) divisible by \(16\), contrary
to \(200\equiv8\pmod {16}\).

In the second family, a level at least \(10\), together with the other
\(v-1\ge29\) levels, would contribute at least
\(10^2+29\cdot2^2>200\).  Thus only levels \(2\) and \(6\) occur.  If their
multiplicities are \(\nu_2,\nu_6\), then
\[
\nu_2+9\nu_6=50,\qquad
v=\nu_2+\nu_6=50-8\nu_6\ge30,
\]
so
\[
(\nu_2,\nu_6)\in\{(50,0),(41,1),(32,2)\}.
\]
Every root is now either \(e_i+e_j\) with \(s_i=s_j=2\), or
\(e_i-e_j\) with \(s_i=6\) and \(s_j=2\), after ordering the two coordinates.
On a fixed support the sign pattern is unique, and duplicate columns
would have inner product \(2\).  Distinct roots in this family therefore
share at most one coordinate.  At a level-two coordinate, let \(p_i,m_i\)
count positive and negative incidences.  We have \(p_i-m_i=2\).
Each level-six coordinate has six incidences, all paired with negative
incidences at level-two coordinates.  A negative Gram product occurs when
two roots meet at a level-two coordinate with opposite signs.  The shared
coordinate is unique, so each such product is counted once in
\[
P_-=\sum_{i:s_i=2}p_im_i.
\]
Consequently
\[
P_-\equiv\sum_{i:s_i=2} p_im_i
\equiv\sum_{i:s_i=2} m_i
=6\nu_6
\equiv0\pmod2,
\]
contradicting the parity already proved.

For the odd levels, let \(a_t,c_t\) count coordinates at levels
\(4t+1,4t+3\), respectively. These form two chains, one in each odd residue
class modulo four. Sum the signed incidences along each chain. A difference
root has one positive and one negative coordinate, so its contributions
cancel within that chain.  The only roots
joining the chains are sum roots \(e_i+e_j\) between levels \(1\) and
\(3\), and they contribute equally to the two sums.  Hence
\[
\sum_{t\ge0}(4t+1)a_t
=\sum_{t\ge0}(4t+3)c_t.
\]
To see the cancellation in the squared levels, write
\[
\begin{aligned}
(4t+1)^2-3&=16t(t+1)-2(4t+1),\\
(4t+3)^2-3&=16t(t+1)+2(4t+3).
\end{aligned}
\]
After multiplying by \(a_t\) and \(c_t\) and summing, the linear terms
cancel by the signed-incidence identity above.  Since
\(\lVert s\rVert^2=200\), this gives
\[
\frac{200-3v}{32}
=\sum_{t\ge1}\binom{t+1}{2}(a_t+c_t).
\]
The right-hand side is an integer.  Hence \(v\equiv24\pmod {32}\), which is impossible for
\(30\le v\le51\).  All three level families are excluded, proving that \(T\)
is disconnected.
\end{proof}

\subsection{Six-cycle bounds from shifted spectral moments}
\label{app:order50-moment-bounds}

We give the determinant calculation for the second pair of six-cycle
bounds in Theorem~\ref{thm:order50-feasibility}.
Work under its hypotheses.  The earlier part of its proof has established
diameter three and \(N_5=360+m/5\), where \(m\) counts the high edges.
Let \(\theta_1,\ldots,\theta_{49}\) be the nonprincipal adjacency
eigenvalues, and put
\[
 y_i=\theta_i+1,\qquad s_j=\sum_{i=1}^{49}y_i^j\quad(0\le j\le6).
\]
Theorem~\ref{thm:diameter-three-score} gives \(|y_i|<\sqrt{10}\).

Girth and regularity give
\[
 \tr A=\tr A^3=0,\qquad \tr A^2=300,\qquad \tr A^4=3300.
\]
For the degree-six nonbacktracking polynomials from
Proposition~\ref{prop:higher-transfer}, the recurrence gives
\[
 F_5(x)=x^5-21x^3+85x,\qquad
 F_6(x)=x^6-26x^4+165x^2-150.
\]
A closed nonbacktracking walk of length five or six is a cycle: a repeated
vertex before the final return would leave a cycle of length at least five
and require at least two further steps.
Each cycle is counted once for every starting vertex and orientation.
Thus \(\tr F_5(A)=10N_5\) and \(\tr F_6(A)=12N_6\), yielding
\[
 (\tr A^j)_{j=0}^6
 =(50,0,300,0,3300,3600+2m,43800+12N_6).
\]
The principal adjacency eigenvalue is six.  Binomial expansion therefore
gives
\[
 s_j=\sum_{r=0}^j\binom jr\tr A^r-7^j,
\]
and hence
\[
 (s_0,\ldots,s_6)
 =(49,43,301,607,2749,6343+2m,1801+12m+12N_6).
\]

For the lower bound, take the vectors
\(u_i=(1,y_i,y_i^2,y_i^3)^{\mathsf T}\).
Their Gram sum is
\[
 \mathsf M=\sum_{i=1}^{49}u_i u_i^{\mathsf T}
 =(s_{a+b})_{a,b=0}^3
 =
 \begin{pmatrix}
 49&43&301&607\\
 43&301&607&2749\\
 301&607&2749&6343+2m\\
 607&2749&6343+2m&1801+12m+12N_6
 \end{pmatrix}\succeq0.
\]
For the upper bound, put \(v_i=(1,y_i,y_i^2)^{\mathsf T}\) and weight each
rank-one matrix by \(10-y_i^2>0\). This gives the following weighted moment
matrix, also called the localizing matrix:
\[
 \mathsf L=\sum_{i=1}^{49}(10-y_i^2)v_i v_i^{\mathsf T}
 =(10s_{a+b}-s_{a+b+2})_{a,b=0}^2
 =
 \begin{pmatrix}
 189&-177&261\\
 -177&261&-273-2m\\
 261&-273-2m&25689-12m-12N_6
 \end{pmatrix}\succ0.
\]
To show that this weighted moment matrix is positive definite, note that diameter three implies that \(I,A,A^2,A^3\) are
linearly independent: evaluate a linear relation successively at vertex
pairs of distances three, two, one, and zero. Since \(A\) is symmetric,
it has at least four distinct eigenvalues. Connectedness makes the
principal eigenvalue simple, so at least three distinct values occur among
the \(y_i\). A nonzero polynomial of degree at most two cannot vanish at all three.
Together with \(10-y_i^2>0\), this shows that the quadratic form of
\(\mathsf L\) is positive on every nonzero vector.
For the four-by-four matrix \(\mathsf M\), the test polynomials can have
degree three. Three distinct values do not rule out a nonzero polynomial
vanishing at all of them, so the same argument gives only semidefiniteness.

Expanding the determinants gives
\begin{align*}
 \det\mathsf M
 &=1200(58500N_6-43m^2+70200m-119632500),\\
 \det\mathsf L
 &=108(4220000-2200m-7m^2-2000N_6).
\end{align*}
Thus \(\det\mathsf M\ge0\) and \(\det\mathsf L>0\).
Dividing by the positive constants and isolating \(N_6\) gives
\[
 \frac{43m^2-70200m+119632500}{58500}
 \le N_6
 <\frac{4220000-2200m-7m^2}{2000}.
\]

\section{Further deletions inside a neighborhood}
\label{app:neighborhood-deletions}

The deletion sets in this appendix lie in the neighborhood of a vertex, so
Proposition~\ref{prop:neighborhood-puncture} applies to them. Neither result
is needed for the deletion bounds of Section~\ref{sec:punctures}.

\begin{proposition}[Six-neighbor punctures]
\label{prop:hs-neighborhood-six}
Let \(M\) be the Hoffman--Singleton graph.  For every \(w\in V(M)\) and
every \(S\subseteq N_M(w)\) with \(|S|=6\), the graph \(M-S\) is a
strict counterexample to WOW-284.
\end{proposition}

\begin{proof}
Put \(H=M-S\). Each vertex in \(\Gamma_2(w)\) has a unique
neighbor in \(N_M(w)\), which we call its parent; the vertices adjacent to
a given parent in the second sphere are its children.  Partition \(V(H)\)
into \(\{w\},\mathcal A,\mathcal B,\mathcal C\), where
\(\mathcal A=N_M(w)\setminus S\), \(\mathcal B\) consists of children of
the deleted parents, and \(\mathcal C\) consists of children of the
surviving parents.  Their sizes are
\[
 1,\qquad1,\qquad36,\qquad6.
\]
The Moore property gives no edge between children of the same parent and
exactly one neighbor in every other child class.  Indeed, if \(x\) has
parent \(p\) and \(q\in N_M(w)\setminus\{p\}\), the unique common
neighbor of the nonadjacent pair \(x,q\) belongs to \(\Gamma_2(w)\)
and has parent \(q\).

The distance sum from a vertex to each part depends only on its own part.
For example, a vertex of \(\mathcal B\) has five siblings at distance three,
five other vertices of \(\mathcal B\) at distance one, and twenty-five at
distance two. Its distance sum to \(\mathcal B\) is therefore
\(5\cdot3+5\cdot1+25\cdot2=70\).
Thus the partition is equitable for the
distance matrix. On vectors constant on each part, the distance operator
acts by the following row-sum quotient, in the order
\(\{w\},\mathcal A,\mathcal B,\mathcal C\):
\[
 Q=\begin{pmatrix}
 0&1&108&12\\
 1&0&72&6\\
 3&2&70&11\\
 2&1&66&10
 \end{pmatrix}.
\]
The distances in the original Moore graph, restricted to the surviving vertices, are given by
\(D_0=2(J-I)-A(H)\).  Proposition~\ref{prop:neighborhood-puncture} changes
only the distances from \(w\) to \(\mathcal B\) and between two children of
the same deleted parent, adding one in each case.

Let \(U\) be the four-dimensional space of vectors constant on each part.
Its orthogonal complement \(U^\perp\) consists of vectors whose coordinates
sum to zero separately on each part. Both \(D_0\) and the distance-increase
matrix \(E_S\) preserve \(U\). Their symmetry also makes
\(U^\perp\) invariant. We can therefore check the shifted distance matrix
on these two spaces separately.

We first bound the correction on \(U^\perp\). The \(w\)-coordinate is zero
because its part has only one vertex. For \(x\in U^\perp\),
\[
 x^{\mathsf T}E_Sx
 =\sum_{z\in S}\left(\sum_{y\in N_M(z)\setminus\{w\}}x_y\right)^2
   -\sum_{y\in\mathcal B}x_y^2
 \ge-\lVert x\rVert^2.
\]
The distance matrix of the original Moore graph has least eigenvalue \(-4\).
Taking its principal submatrix on the surviving vertices gives
\(D_0\succeq-4I\), and the correction estimate above gives
\(D(H)|_{U^\perp}\succeq-5I\).
Theorem~\ref{thm:small-puncture} gives
\(\delta^*(H)=43/7>5\).

It remains to check the restriction to \(U\).  In the basis of cell
indicators, the quadratic form of \(7D(H)+43I\) has symmetric matrix
\[
 T=\operatorname{diag}(1,1,36,6)(7Q+43I)
 =\begin{pmatrix}
 43&7&756&84\\
 7&43&504&42\\
 756&504&19188&2772\\
 84&42&2772&678
 \end{pmatrix}.
\]
Its leading principal minors are
\[
 43,\qquad1800,\qquad4374000,\qquad226834992.
\]
All four are positive, so Sylvester's criterion gives
\(T\succ0\).  The shifted distance matrix is therefore positive
definite on both \(U\) and \(U^\perp\), proving \(\Phi(H)>0\).
\end{proof}

These successful six-vertex deletions do not increase the universal radius:
the deletion in Theorem~\ref{thm:hs-radius} still fails.

For a hypothetical degree-\(57\) Moore graph, the neighborhood
calculation reaches the endpoint of the small-puncture theorem.

\begin{corollary}[Degree-\(57\) neighborhood punctures]
\label{cor:degree57-neighborhood}
If \(M\) is a degree-\(57\) Moore graph of diameter two, then for every
\(w\in V(M)\) and every \(S\subseteq N_M(w)\) with \(|S|\le56\),
\[
 \Phi(M-S)>\frac{685}{57}.
\]
In particular, every such induced graph is a strict counterexample to
WOW-284.
\end{corollary}

\begin{proof}
Theorem~\ref{thm:moore-threshold} gives the original Moore graph a score of \(48\), which
handles \(S=\varnothing\).  For nonempty \(S\), put \(s=|S|\).
Theorem~\ref{thm:small-puncture} applies through \(s=56\): the survivor
graph is connected and has minimum dual degree \(57-s/57\).  Its order
is at least \(3194\), and it inherits girth at least five.
Proposition~\ref{prop:neighborhood-puncture} gives
\[
 \lambda_{\min}(E_S)
 =\frac{55-\sqrt{55^2+224s}}2>-35,
\]
since
\[
 55^2+224s\le15569<15625=125^2.
\]
The loss of minimum dual degree is exactly \(s/57\), so
Proposition~\ref{prop:deletion-stability} yields
\[
 \Phi(M-S)\ge48-\frac{s}{57}+\lambda_{\min}(E_S)
 >48-\frac{56}{57}-35=\frac{685}{57}>0.
\]
\end{proof}

\section{Small orders, deletions of the Hoffman--Singleton graph, and an irregular identity}
\label{app:small-orders}

This appendix completes Section~\ref{sec:small-order-bounds} and uses its
labels \emph{proved}, \emph{computer-assisted}, \emph{Lean-checked} and
\emph{conjectured}, and its exact decision procedure.

Throughout, \(M\) is the coordinate Hoffman--Singleton graph of
Proposition~\ref{prop:hs-coordinate}, and \(R\), \(H_{38}\), \(H_{39}\) and \(X_{42}\)
are the graphs defined before Theorem~\ref{thm:explicit-examples}. For a
graph \(G\), \(\Delta_G=\operatorname{diag}(d(v))\) is the diagonal degree
matrix.

\subsection{An irregular form of the diameter-three identity}

\begin{theorem}[Irregular diameter-three identity; proved, identities Lean-checked]
\label{thm:irregular-d3}
Let \(G\) be a finite connected graph of girth at least five and diameter at
most three, with adjacency matrix \(A\) and distance matrix \(D\). Then
\[
 D=3J-3I-2A-A^2+\Delta_G,
 \qquad
 D+tI=3J+\Delta_G+(t-2)I-(A+I)^2\quad(t\in\mathbb R).
\]
Consequently \(\Phi(G)>0\) if and only if
\(3J+\Delta_G+(\delta^*(G)-2)I-(A+I)^2\) is positive definite.
\end{theorem}

\begin{proof}
Since \(G\) has no triangle, adjacent vertices have no common neighbor; since
it has no four-cycle, distinct vertices have at most one. Hence \((A^2)_{uu}=d(u)\),
and for \(u\ne v\) the entry \((A^2)_{uv}\) is \(1\) if \(d_G(u,v)=2\) and \(0\)
otherwise. Thus \(A_2=A^2-\Delta_G\) is the distance-two indicator, and the
diameter bound gives \(A_3=J-I-A-A_2\) for the distance-three indicator.
Substituting into \(D=A+2A_2+3A_3\) gives the first identity, and expanding
\((A+I)^2\) gives the second. The last assertion is
\(\lambda_{\min}(D)>-\delta^*(G)\), that is, positive definiteness of
\(D+\delta^*(G)I\).
\end{proof}

For a \(k\)-regular graph of diameter three, \(\Delta_G=kI\) and the first
identity is the one in Theorem~\ref{thm:diameter-three-score}. The two
identities and, for every real \(t\), the equivalence of positive
definiteness of the two sides of the second are Lean-checked
(Appendix~\ref{sec:verification}); the passage to \(\Phi(G)>0\) through the
least distance eigenvalue is the written argument above.

\begin{corollary}[Deletions from a Moore graph; proved]
\label{cor:moore-deletion-form}
Let \(M'\) be a Moore graph of diameter two and let \(S\subseteq V(M')\) be
such that \(H=M'-S\) is connected of diameter at most three. Let \(C_S\) be
the \(0/1\) matrix on \(V(H)\) whose \((x,y)\) entry is \(1\) exactly when
\(x,y\) are distinct and nonadjacent and their unique common neighbor in
\(M'\) lies in \(S\). Then
\[
 D(H)=2(J-I)-A(H)+C_S,
\]
so \(\Phi(H)>0\) if and only if \(2J-A(H)+C_S+(\delta^*(H)-2)I\) is positive
definite.
\end{corollary}

\begin{proof}
Nonadjacent vertices of \(M'\) have exactly one common neighbor. In \(H\), two
vertices are at distance one if adjacent, at distance two if their common
neighbor survives, and otherwise at distance at least three, hence exactly
three by the diameter hypothesis.
\end{proof}

With \(B\) as in Theorem~\ref{thm:small-puncture}, \(C_S=BB^{\mathsf T}-
\operatorname{diag}(BB^{\mathsf T})\); that theorem is the case \(|S|\le k-1\),
where the diameter hypothesis holds automatically. Every strict induced
subgraph of \(M\) found in Section~\ref{sec:small-order-bounds}, and the
equality graph \(E_{43}\) of
Proposition~\ref{prop:order43-equality}, has diameter at most three and is an
instance of Corollary~\ref{cor:moore-deletion-form}; connected induced
subgraphs of larger diameter already occur for \(|S|=7\). The corollary is a
reformulation and not a classification: for \(|S|=6,7,8\) in
Proposition~\ref{prop:hs-classification},
the number of edges inside \(S\), the largest number of deleted neighbors of
a surviving vertex, and the presence of a closed neighborhood in \(S\) do not
separate strict from non-strict deletions.

\subsection{The search on at most nineteen vertices}

We describe the computation behind Proposition~\ref{prop:nineteen-vertices}.
The graphs were generated by \texttt{geng} from nauty
\cite{McKayPiperno2014} with the options \texttt{-c -t -f} (connected,
triangle-free, four-cycle-free), one graph per isomorphism class. For
\(n\le14\), every generated graph was decided. For \(15\le n\le19\), the
maximum degree was capped by Theorem~\ref{thm:fifteen-vertices}
(\texttt{-D4} for \(n\le18\) and \texttt{-D5} for \(n=19\)), and only graphs
with \(\delta^*>2+\sqrt2\) were kept, the test being done in integer
arithmetic; both restrictions are necessary for strictness. Graphs are
decided as described at the start of Section~\ref{sec:small-order-bounds}.
\[
\begin{array}{lrrr}
\toprule
n&\text{graphs generated}&\text{tested exactly}&\text{strict}\\
\midrule
3\text{--}14&329{,}766&329{,}766&0\\
15&1{,}417{,}899&0&0\\
16&11{,}152{,}778&1&0\\
17&98{,}858{,}189&19&0\\
18&977{,}178{,}397&221&0\\
19&25{,}975{,}541{,}173&7{,}459&0\\
\bottomrule
\end{array}
\]
For \(15\le n\le19\), the column ``tested exactly'' counts the graphs that
pass the dual-degree test; a discarded graph has \(\delta^*<2+\sqrt2\) and
is therefore not strict, by Lemma~\ref{lem:small-diam-two} and the proof of
Theorem~\ref{thm:fifteen-vertices}. Exactly one graph is not certified by an integer vector: the
\(3\)-regular graph of girth five on ten vertices, the Petersen graph, for
which \(\Phi=0\) by Theorem~\ref{thm:moore-threshold}. The proposition is a statement about orders at most \(19\);
orders \(20\) to \(37\) are not covered by it.

\subsection{The orbit enumeration}

We describe the enumeration behind Proposition~\ref{prop:hs-classification}.
The orbits were enumerated by orderly generation. All \(252{,}000\)
automorphisms of \(M\) were found by an exhaustive backtracking search over
adjacency-preserving bijections. Number the vertices by
\(P_{i,j}\mapsto5i+j\) and \(Q_{k,\ell}\mapsto25+5k+\ell\)
(\(0\le i,j,k,\ell\le4\)), and take \(S\) as the representative of its orbit
when \(\sum_{z\in S}2^{49-z}\) is maximal over the orbit. If \(S\) is a
representative and \(y=\max S\), then \(T=S\setminus\{y\}\) is one too.
Otherwise \(g(T)\) beats \(T\) for some automorphism \(g\); let \(p\) be the
first vertex number at which they differ, so \(p\in g(T)\setminus T\). If
\(p\ge y\), then \(g(T)\supseteq T\), so \(g(T)=T\); hence \(p<y\), and
\(g(S)=g(T)\cup\{g(y)\}\) beats \(S\) at \(g(y)\) or at \(p\), whichever
comes first. Representatives of size \(s\) therefore arise from
representatives of size \(s-1\) by adding a larger vertex. At every size the
number of representatives equals the number of orbits obtained from the cycle
index of the automorphism group. Each connected representative is decided as
in Proposition~\ref{prop:nineteen-vertices}; the equality case is decided by
the explicit certificate of Proposition~\ref{prop:order43-equality}. The proposition
concerns induced subgraphs of \(M\) of order at least \(36\) only.

\subsection{An equality case of order \texorpdfstring{\(43\)}{43}}

\begin{proposition}[computer-assisted]\label{prop:order43-equality}
Let \(E_{43}=M-\{P_{0,0},P_{0,2},P_{1,0},P_{1,2},P_{2,1},P_{4,1},Q_{1,4}\}\);
the deleted set is independent. Then \(E_{43}\) has order \(43\), \(126\)
edges, degree counts \(5^{21}\,6^{7}\,7^{15}\), girth five and diameter three,
\(\delta^*(E_{43})=6\), and
\[
 \lambda_{\min}(D(E_{43}))=-6\quad\text{with multiplicity }2,
 \qquad \Phi(E_{43})=0.
\]
\end{proposition}

The certificate is explicit. A principal submatrix of \(D+6I\) of order
\(41\) has an exact rational \(LDL^{\mathsf T}\) factorization with all pivots
positive. For each of the two remaining indices \(r\), a rational vector equal
to \(e_r\) plus a vector supported on those \(41\) indices lies in the kernel
of \(D+6I\). Hence \(D+6I\) is congruent to the direct sum of a positive
definite matrix of order \(41\) and the zero matrix of order \(2\), so it is
positive semidefinite of rank \(41\). This gives
\(\lambda_{\min}(D(E_{43}))=-6\) with multiplicity \(2\), and
\(\Phi(E_{43})=6-6=0\). This is an irregular case of
equality in WOW-284; J{\o}rgensen's equality example after
Theorem~\ref{thm:equality-boundary} is regular. It is the only orbit with
\(\Phi=0\) in Proposition~\ref{prop:hs-classification}. This proposition is
not formalized in Lean.

\subsection{Two-vertex deletions of \texorpdfstring{\(R\)}{R} and \texorpdfstring{\(X_{42}\)}{X42}}

\begin{proposition}[computer-assisted]\label{prop:edge-deletion-dichotomy}
Let \(G\in\{R,X_{42}\}\) and let \(u,v\) be distinct vertices of \(G\). Then
\(G-\{u,v\}\) is a strict counterexample if and only if \(u\sim v\). Every
edge deletion has \(\delta^*=17/3\). For the edge \(P_{0,2}P_{0,3}\) of
\(X_{42}\), that is, for \(X_{40}\),
\[
 \det(xI-D(X_{40}))=x^4(x+5)^{11}(x^2-4x+2)(x^2+6x+2)^4
 (x^3-83x^2+390x-414)(x^3+2x^2-20x-4)^4,
\]
and the least root is \(-3-\sqrt7\), with multiplicity four, so
\(\Phi(X_{40})=17/3-3-\sqrt7=\Phi(H_{38})\).
\end{proposition}

For \(R\), the edge half is part of Theorem~\ref{thm:explicit-examples}. The
computation decides all \(780\) pairs of \(R\) (\(120\) edges) and all \(861\)
pairs of \(X_{42}\) (\(126\) edges) exactly, and the least root of the
displayed polynomial is located by exact real-root isolation. The value
\(17/3\) is elementary: after the endpoints of an edge \(uv\) of a
\(k\)-regular graph of girth five are deleted, the vertices of
\(N(u)\cup N(v)\) keep \(d^*=k\), and every other vertex \(x\) has
\(d^*(x)\ge k-2/k\), with equality exactly when \(x\) has a neighbor in
\(N(u)\setminus\{v\}\) and one in \(N(v)\setminus\{u\}\).

\begin{conjecture}[conjectured]\label{conj:edge-deletion}
Let \(G\) be a connected \(6\)-regular graph of girth five and diameter three
with \(\Phi(G)>0\). Then, for distinct vertices \(u,v\), the graph
\(G-\{u,v\}\) is a strict counterexample if and only if \(u\sim v\).
\end{conjecture}

The conjecture is supported only by \(R\) and \(X_{42}\), the two known
graphs satisfying its hypotheses. Both are induced subgraphs of \(M\), and
the least root \(-3-\sqrt7\) of the edge deletions of \(R\) (the polynomial
\(P_{38}\) of Appendix~\ref{app:small-example-certificates}) and of
\(X_{40}\) carries the degree of \(M\), as does \(-2-\sqrt7\) for an edge
deletion of \(M\) itself (Theorem~\ref{thm:edge-puncture}). The pattern may
therefore be specific to the Hoffman--Singleton graph rather than to degree
six.

\section{Spectral exclusions for two construction families}
\label{app:excluded-constructions}

We consider two possible sources of new examples: a prime-field
construction and certain perfect-matching deletions from Hoffman--Singleton.
The first gives no strict counterexample of diameter three, and none of the
deletions in the second family is strict. We give the definitions and proofs
together so these exclusions can be read independently of the main results.

\subsection{The prime-field construction in diameter three}

For an odd prime \(q\ge7\) and \(1\le m\le q\), define \(G(q,m)\) on
vertices \(P_{i,j},Q_{k,\ell}\), with \(0\le i,k<m\) and
\(j,\ell\in\mathbb F_q\), by
\begin{align*}
P_{i,j}&\sim P_{i,j\pm1},\\
Q_{k,\ell}&\sim Q_{k,\ell\pm2},\\
P_{i,j}&\sim Q_{k,ik+j}.
\end{align*}
This is a balanced specialization of known finite-field girth-five
constructions \cite{AbreuEtAl2008}.  It is \((m+2)\)-regular.  A coordinate
common-neighbor calculation shows that it has no triangle or \(4\)-cycle for
\(q\ge7\): for a cross pair, the possible same-side common neighbors require
residues in the disjoint sets \(\{\pm1\}\) and \(\{\pm2\}\).

\begin{theorem}\label{thm:prime-field}
If \(G(q,m)\) has diameter three, then it is not a strict counterexample to
WOW-284.
\end{theorem}

\begin{proof}
The degree restriction first excludes \(m\le3\) by
Theorem~\ref{thm:regular-degree-six}.  For the remaining degrees, we first
consider vectors that assign the same value to all vertices of a given
type with a fixed first coordinate.  Taking neighbor sums leaves these
vectors constant on the same sets.  The restricted adjacency matrix has
eigenvalues \(m+2\), \(2-m\), and \(2\) with multiplicity \(2m-2\);
the shifted WOW window leaves only \(m\in\{4,5,6\}\)
(\(\theta=2\) needs \(3<\sqrt{2m+2}\), i.e.\ \(m\ge4\); \(\theta=2-m\) needs
\((m-3)^2<2m+2\), i.e.\ \(m\le6\)).

We decompose vectors into Fourier modes in the second coordinate. A mode
assigns values proportional to successive powers of a root of unity as that
coordinate varies. The vectors just tested form the zero-frequency block.
A nonzero mode supplies the remaining restriction and excludes the last
three cases.
Let \(\omega=e^{2\pi\mathrm i/q}\).  On the nonzero
character \(t=1\), the adjacency block has form
\[
 \begin{pmatrix}aI&M\\M^*&bI\end{pmatrix},
 \qquad a=2\cos\frac{2\pi}{q},
 \quad b=2\cos\frac{4\pi}{q},
 \quad M_{ik}=\omega^{ik}.
\]
The star in the block matrix denotes conjugate transpose. Every vector
in this character space has coordinate sum zero, since
\(\sum_{j\in\mathbb F_q}\omega^j=0\). Its adjacency eigenvalues are
therefore nonprincipal.
If \(\sigma\) is a singular value of \(M\), the associated two-dimensional
invariant subspace carries
\(\bigl(\begin{smallmatrix}a&\sigma\\\sigma&b\end{smallmatrix}\bigr)\),
whose larger eigenvalue
\(\frac{a+b}2+\sqrt{(\frac{a-b}2)^2+\sigma^2}\) is at least
\(\frac{a+b}2+\sigma\).
We need only the largest singular value, rather than the full Fourier
spectrum. The squared Frobenius norm is the sum of the squared singular
values. Since \(\|M\|_F^2=m^2\), one singular value satisfies
\(\sigma^2\ge m\).  Its two-dimensional block therefore supplies a
nonprincipal eigenvalue at least
\[
 \sqrt m+\frac{a+b}{2}
 \ge \sqrt m+\cos(\pi/7)-\frac12.
\]
For \(q\ge7\), both \(\cos(2\pi/q)\) and \(\cos(4\pi/q)\) increase with
\(q\), so their sum is minimized at \(q=7\).  The last inequality then uses
\(2\cos(2\pi/7)+2\cos(4\pi/7)=2\cos(\pi/7)-1\).
Moreover \(\cos(\pi/7)>\sqrt3/2\), and
\(h(m)=\sqrt{2m+2}-\sqrt m\) is increasing for \(m\ge1\) and satisfies
\[
 h(m)\le h(6)=\sqrt{14}-\sqrt6
 <\frac{27}{20}<\frac{\sqrt3+1}{2}.
\]
Thus the displayed nonprincipal eigenvalue lies above
\(-1+\sqrt{2m+2}\) for \(m=4,5,6\).  The common-neighbor
case split, Fourier reduction, radical comparisons, and exact \(q=7\) controls
are independently checked by exact computation.
\end{proof}

\subsection{Why layer-respecting edge deletions satisfy WOW-284}

For a permutation \(\pi\in S_5\) of the five first-coordinate labels,
delete the perfect matching
\[
 \mathcal M_\pi=\bigl\{\{P_{i,j},Q_{\pi(i),i\pi(i)+j}\}:i,j\in\F\bigr\}
\]
from the Hoffman--Singleton graph.
We call the permutation affine when it multiplies a label by a nonzero
residue and then adds a fixed residue, all modulo five.

\begin{theorem}\label{thm:matching-deletions}
Each deletion produces a connected simple \(6\)-regular graph of order
\(50\), girth five, and diameter four.  The \(120\) labeled graphs form
exactly two isomorphism classes.  The \(20\) affine permutations have
\[
 \lambda_{\min}(D)=-13,
 \qquad \Phi=-7,
\]
and the \(100\) nonaffine permutations have
\[
 \lambda_{\min}(D)=-6-\sqrt{61},
 \qquad \Phi=-\sqrt{61}.
\]
Thus none of these graphs is a counterexample.
\end{theorem}

\begin{proof}
Every \(P_{i,j}\) occurs once in \(\mathcal M_\pi\); for a fixed
\(Q_{k,\ell}\), the unique incident matching edge is obtained from
\(i=\pi^{-1}(k)\) and \(j=\ell-ik\).  Thus \(\mathcal M_\pi\) is a
perfect matching, and its deletion leaves a simple \(6\)-regular graph.
Deleting edges cannot create a short cycle, while the same-layer pentagons
remain; exact breadth-first search gives connectedness and diameter four.
Non-strictness already follows from Theorem~\ref{thm:diameter-four} since
\(k=6\le9\); the exact least eigenvalues below sharpen it.

Explicit type-preserving and type-swapping coordinate automorphisms generate
orbits of sizes \(20\) and \(100\).  The representatives have different
adjacency characteristic polynomials, so the orbits are distinct isomorphism
classes.  Exact distance characteristic polynomials and Sturm separators give
the stated least roots.  All \(120\) matchings, \(400\) coordinate maps,
\(48{,}000\) matching images, orbit coverage, graph hypotheses, and root
certificates are checked
by exact computation.
\end{proof}

\section{Computations, code, and the scope of Lean verification}\label{sec:verification}

The computer-assisted arguments verify the labeled finite graphs, check
the complete enumeration of the stated finite orbit families, and evaluate the
symbolic identities derived in the text. All signs and eigenvalue orderings
used in the theorems are established by exact arithmetic.  The computations use
SymPy and NetworkX \cite{MeurerEtAl2017,HagbergSchultSwart2008}.
For J{\o}rgensen's graph after Theorem~\ref{thm:equality-boundary}, an exact computation
checks three local representations of the graph, its graph6 decoding, the
characteristic polynomials and the root intervals.
Wherever the text says that a step is checked by exact computation, or by
computer, the step is reproduced by a script in the code described in
the next subsection.
The Lean results below provide a separate layer of formal verification.

\subsection{Code and data}\label{sec:code-data}
The source, graph data, exact certificates, verification scripts, and Lean
development are available at
\href{https://github.com/SamPetkov/wow284}{\texttt{github.com/SamPetkov/wow284}}.
Release \texttt{\RepoTag} of that repository archives the graph encodings
(\codefile{data/graphs/} and \codefile{data/cages55/}), the exact rational and
polynomial certificates, the finite-orbit records, and the verification
scripts (\codefile{scripts/}). The Lean development is in \codefile{lean/};
its toolchain and dependency revisions are pinned in
\codefile{lean/lean-toolchain} and \codefile{lean/lake-manifest.json}.
The order-fifty computations of Appendix~\ref{app:order50-details} and those
of Section~\ref{sec:small-order-bounds} and Appendix~\ref{app:small-orders}
postdate release \texttt{\RepoTag}; they
are reproduced by the scripts at commit
\href{https://github.com/SamPetkov/wow284/tree/c546c8acf7d101fe6cbcd7e50902b65b2ddcd588}{\texttt{c546c8a}}
of the same repository, and the order-fifty profile scripts are also included
as ancillary files (\path{anc/scripts/}) with this article. The scripts for
Section~\ref{sec:small-order-bounds} and
Appendix~\ref{app:small-orders} are \path{scripts/small_order_core.py} (graph
constructions and exact decisions), \path{scripts/small_order_hs_orbits.py}
(the orbit enumeration of Proposition~\ref{prop:hs-classification}),
\path{scripts/small_order_deletions.py} (labeled deletion screens),
\path{scripts/small_order_g6filter.c} and
\path{scripts/small_order_certify_g6.py} (the dual-degree filter and the
exact decisions for the \texttt{geng} output of
Proposition~\ref{prop:nineteen-vertices}), \path{scripts/recertify_order40.py}
(the second check in Proposition~\ref{prop:order40-irregular}),
\path{scripts/verify_order43_equality.py}
(the certificate of Proposition~\ref{prop:order43-equality}) and
\path{scripts/verify_edge_deletion_dichotomy.py}
(Proposition~\ref{prop:edge-deletion-dichotomy});
\path{scripts/generate_lean40_hs_extra.py} generates the order-\(40\) Lean
certificates. The orbit counts, representative graphs, generation counts and
exact pivots are in \path{data/hs_induced_classification/},
\path{data/small_order_exhaustive/}, \path{data/hs_order40_extra/} and
\path{data/hs_two_vertex_deletions/} at the same commit. Analytic results not
covered by these computations are proved in the text.

\subsection{Lean verification}

For the explicit \(50\)-vertex graph, Lean 4.31 with Mathlib 4.31
\cite{deMouraUllrich2021,Mathlib2020} checks structural and exact rational
matrix certificates.  The development pins the toolchain and dependency
revisions (see the preceding subsection).  It checks the coordinate graph,
regularity, the exhaustive common-neighbor certificate, girth five,
the adjacency-square identity, the algebraic formula for the case-defined
\(0/1/2\) matrix, and an exact rational diagonalization with multiplicities.
The diameter-two certificate identifies that matrix with the actual
distance matrix in the mathematical proof.  Regularity and diagonalization
then give
\[
 \delta^*(G)=7,\qquad \lambda_{\min}(D(G))=-4,\qquad \Phi(G)=3.
\]
These conclusions are assembled in the paper; they are not exported from
Lean as a combined graph-distance theorem.

Lean also kernel-checks finite spectral certificates attached to the explicit
constructions of orders \(38,39,40,42\).  At orders \(38,39,42\), they
certify the dual-degree data and positive definiteness of the relevant shifted
distance matrix.  At order \(40\), they certify an invertible exact
diagonalization, multiplicities, an attained minimum diagonal entry of \(-5\), dual degree six,
and gap one.  These certificates concern the finite matrices supplied to
Lean and do not provide end-to-end formalizations via
\texttt{SimpleGraph.dist}.
Certificates of the same kind are checked for the order-\(40\) graphs
\(X_{40}\), \(Y_{40}\) and \(Z_{40}\) of
Proposition~\ref{prop:order40-irregular}, in the modules
\texttt{Wow284.Induced40X}, \texttt{Wow284.Induced40Y} and
\texttt{Wow284.Induced40Z}: the degree data, the exact minimum dual degree,
girth five through the embedding in the coordinate graph, diameter three, the
distance matrix, and positive definiteness of \(3D+17I\), respectively
\(7D+39I\).

The module \texttt{Wow284.DiameterThreeMatrix} checks graph-level statements
whose hypotheses use Mathlib's \texttt{SimpleGraph.Connected},
\texttt{SimpleGraph.egirth} and \texttt{SimpleGraph.diam}, and whose distance
matrix has the entries \texttt{SimpleGraph.dist}. For a connected graph of
girth at least five and diameter at most three, it proves the identities of
Theorem~\ref{thm:irregular-d3} and, for every real \(t\), the equivalence of
positive definiteness of the two sides of the second, as the declarations
\[
\begin{array}{l}
\texttt{distMatrix\_eq\_of\_diam\_le\_three},\\
\texttt{distMatrix\_add\_smul\_one\_eq\_of\_diam\_le\_three},\\
\texttt{posDef\_shift\_iff\_of\_diam\_le\_three}.
\end{array}
\]
For a connected \(k\)-regular graph of girth
at least five and diameter three, it proves the identities of
Theorem~\ref{thm:diameter-three-score}, the transfer of a zero-sum adjacency
eigenvector with eigenvalue \(\theta\) to a distance eigenvector with
eigenvalue \(k-2-(\theta+1)^2\), the principal distance eigenvalue
\(3n-k^2-k-3\), and, for every adjacency eigenvalue \(\theta\) with a nonzero
zero-sum eigenvector, the implication from positive definiteness of \(D+kI\)
to \(|\theta+1|<\sqrt{2k-2}\). The passage from these identities to the least
distance eigenvalue, and hence to \(\Phi\), is not formalized. Among the
small-order results of Section~\ref{sec:small-order-bounds} and
Appendix~\ref{app:small-orders}, only the identities and the
positive-definiteness equivalence of Theorem~\ref{thm:irregular-d3} and the
certificates of Proposition~\ref{prop:order40-irregular} are Lean-checked;
the final equivalence with \(\Phi(G)>0\) in Theorem~\ref{thm:irregular-d3}
is a written argument.

Separately, Lean formalizes the analytic optimization statement in
Theorem~\ref{thm:lp-ceiling}.  For every integer \(k\ge4\) and every admissible
finitely supported expansion
\[
  f=\sum_i c_iF_i,
  \qquad c_0>0,
  \qquad c_i\ge0\quad(i\ge5),
  \qquad f\vert_{I_k}\le0,
\]
it proves
\[
  B_kc_0\le f(k),
  \qquad B_k=\frac{(k+2)(k^2+3)}6.
\]
The formal development defines the coefficient family of the displayed
quartic \(f_*\), proves that it is admissible and attains equality, and proves
that the equality cases are exactly the positive scalar multiples of \(f_*\),
both as polynomials and at coefficient level.  The imported release
development contains no \texttt{sorry} placeholders and is kernel-checked by
Lean 4.31.

This LP formalization is deliberately graph-independent.  It does not
formalize the trace interpretation of the \(F_i(A)\), the girth-five
vanishing and nonnegativity statements, or the passage from the LP inequality
to graph-order bounds.  The integral-slack consequences, signed-complement
theory, the high-edge profile filter in
Corollary~\ref{cor:high-edge-profile-filter}, puncture spectra, and the
general deletion theorems lie outside this Lean
development.  They are proved in the text using the analytic and, where
explicitly identified, exact computer-assisted components.  The paper makes
no Lean claims beyond those stated in this appendix.

\end{document}